\documentclass[11pt]{article}
\usepackage[sectionbib]{natbib}
\usepackage{amssymb}
\usepackage{mathrsfs}
\usepackage{algorithm}

\usepackage{algpseudocode}

\usepackage{rotating} 
\usepackage{amsmath}
\usepackage{amsbsy}
\usepackage{epsfig}
\usepackage{enumerate}
\usepackage{bm}
\usepackage{color}

\newcommand{\blue}[1]{\textcolor{blue}{#1}}

\usepackage[colorlinks,linkcolor=blue]{hyperref}

\newtheorem{lemma}{Lemma}[section]
\newtheorem{theorem}{Theorem}[section]
\newtheorem{condition}{Condition}

\newtheorem{corollary}{Corollary}[section]
\newtheorem{remark}{Remark}[section]
\newtheorem{example}{Example}[section]
\newtheorem{definition}{Definition}[section]
\newtheorem{procedure}{Procedure}

\renewcommand{\theequation}{\thesection.\arabic{equation}}
\newbox\TempBox \newbox\TempBoxA

\def\pr{\textsf{P}} 
\def\ep{\textsf{E}} 
\def\Cov{\textsf{Cov}} 
\def\Var{\textsf{Var}} 
\def\Cal#1{{\mathcal #1}}
\def\underwiggle 1{
\ifmmode\setbox\TempBox=\hbox{$ 1$}\else\setbox\TempBox=\hbox{
1}\fi
\setbox\TempBoxA=\hbox to \wd\TempBox{\hss\char'176\hss}
\rlap{\copy\TempBox}\smash{\lower9pt\hbox{\copy\TempBoxA}}
}
\renewcommand{\baselinestretch}{1.5}
\begin{document}

\thispagestyle{empty}

\begin{center}
 { \LARGE\bf  Efficient Doubly Adaptive Biased Coin Designs for Multiple Treatments$^{\ast}$}
\end{center}

\begin{center} {\sc
\href{https://mypage.zjgsu.edu.cn/tjysjkx/zlx2_en/main.htm}{\blue{Li-Xin Zhang}}\footnote{Research supported by  NSF of China (Grant Nos. U23A2064), National Key
R\&D Program of China (No. 2024YFA1013502) and   the Summit Advancement Disciplines of Zhejiang Province (Zhejiang Gongshang University - Statistics)
}
}\\
{\sl \small School  of Statistics and Data Science, Zhejiang Gongshang University, Hangzhou 310018} \\
(Email:stazlx@mail.zjgsu.edu.cn)    \\
\end{center}

\begin{abstract} The randomness,   efficiency (power and variability), and desirable allocation proportions are important components for evaluating a response-adaptive design in clinical trials and conflicted demands in applications.  The aim of this paper is to provide designs dealing with these dilemmas.
 We first give a general framework  for
 efficient response-adaptive randomization procedures  that attain the Cram\'er-Rao
lower bounds of the allocation variances for any desired allocation proportions. The general framework is flexible for us to define new families of efficient designs with good properties for both two and multiple-treatment clinical trials.   We also prove that, among all response-adaptive randomization procedures with the same limit allocation proportions,  the selection biases and entropies as measures of the randomness of the designs have their optimal values. Basing on the theory on efficiency and randomness, we propose  a new family of  doubly adaptive biased coin designs for multi-treatment clinical trials  that can target any allocation proportion and are asymptotically best in terms both  the randomness and efficiency so that their randomness is asymptotic optimal and asymptotic allocation variance attains the Cram\'er-Rao lower bound. Theoretical properties,
including the strong consistency, the asymptotic normality, and the functional central limit theorem for both the sample allocation proportions and the estimators of the distribution parameters, are developed by using the technique of Gaussian approximation and Gaussian comparing theorems.


\end{abstract}

 \
\baselineskip 22pt

\renewcommand{\baselinestretch}{1.7}



\section{Introduction}\label{sectIntrodution}
\setcounter{equation}{0}

 In clinical trials, patients normally arrive sequentially. Response-adaptive designs are adaptive schemes to randomize treatments to patients with allocation probabilities depending on the results of previous assignments and previous treatment outcomes. The adoption of response-adaptive designs has proved to be beneficial to researchers, by
providing more efficient clinical trials, and to patients, by increasing the likelihood
of receiving better treatment.   Early important work on response-adaptive designs was carried out by \cite{Thompson33} and \cite{Robbins52}. Since then, a lot of response-adaptive designs have been proposed in the literature (see
\citet{RL02}, \citet{HR06}, and therein reference). Much recent
work has focused on proposing better randomized adaptive designs. The
three main components for evaluating a response-adaptive design are allocation
proportion, efficiency (power), and variability, as discussed in  \cite{HR03}.

{\em 1.1. Allocation proportion}. To achieve different goals of adaptive
designs, {\em optimal allocation proportions} are usually derived from certain
multiple-objective optimality criteria. This has been studied by \cite{Hayre79}
and \cite{JT00}.   For two-treatment trials with dichotomous responses, \cite{RSIHR01} proposed an optimal allocation proportion by minimizing the expected number of treatment failures for the fixed variance of the Wald test, which offers a perfect trade-off between the ethical problem and the power of the test.    \cite{TRH07} proposed a general framework to obtain optimal allocation proportions for comparing two or
more treatments. Modern research on response-adaptive randomization has aimed to develop optimal response-adaptive randomization procedures that increase or maintain power over traditional balanced randomization designs while targeting  optimal allocation proportions  derived from certain
 optimality criteria \citep{HR06}.   The doubly-adaptive biased coin design (DBCD) proposed by \cite{HZ04}, which builds upon Efron's biased coin design \citep{Efron71,Eisele94}, is an example of such a procedure. The DBCD can target any allocation, is asymptotically optimal, and require only mild regularity conditions.

{\em 1.2. Variability and efficiency}.  \cite{HR03}  analyzed the power of the classical statistical tests for comparing the efficacy of treatments in response-adaptive randomized experiments, showing that the power is a decreasing function of the variability induced by
the randomization procedure for any given allocation proportion.
\cite{HRZ06}  established a lower bound on the asymptotic variability of response-adaptive
designs which provide a baseline for comparison of existing designs
and further guidance in developing new designs. A response-adaptive design
that attains this lower bound will be said to be  {\em efficient} or {\em asymptotically best}. The DBCD proposed by \cite{HZ04} is not efficient though the asymptotic variability can approach the lower bound as a tune parameter of the design goes to infinity.   For more about the fundamental theory on response-adaptive designs, one can refer to the book of \cite{HR06}  who also stated a lot of important open problems and concluded that the most interesting open question of all is: {\em Is there a fully randomized
procedure that targets any optimal allocation that is asymptotically best?} (Page 157).   In comparing two treatments, \cite{HZH09} gave the first answer to this problem. As a generalization of \cite{Efron71}'s biased coin design (BCD),   a  family of efficient randomized-adaptive designs (ERADE) was proposed by them for any desired allocation proportion that may depend on the unknown parameters. Theoretical results and simulation study verified the advantage of using the ERADE  over existing designs including the randomized play-the-winner rule of \cite{WD78}. But the studies of \cite{HZH09}  are limited to the two-treatment case. Comparing multiple treatments simultaneously is popular in clinical trial studies. The triple-treatment case is typical. For example, in non-inferiority trials which are intended to show that the effect of a new treatment is not worse than a reference product, the new treatment, the reference, and a placebo are usually recommended to be compared simultaneously.
{\em How to generalize ERADE to multi-treatment case} is stated as an open problem in \cite{HZH09} (page 2555). Recently, \cite{Zhang2024} and \cite{AHZ2025} have generalized ERADE to multi-treatment case.

{\em 1.3. Selection bias}. Selection bias is another important issue in clinical trial studies. Selection bias in an adaptive design occurs if the allocation process is predictable. If the experimenter can predict the next
assignment, he may consciously or unconsciously bias the
experiment as to what treatment particular types of subjects should receive. The randomization technique is used for neutralizing such bias in clinical trials.  It is obvious that complete randomization eliminates
selection bias, and the deterministic allocation procedure maximizes it.  The randomness of a design will reduce the selection bias, but at the same time would increase the variability of the design and lose efficiency. As discussed in \cite{HRZ06}, it presents
numerous dilemmas regarding the tradeoffs among randomization, variability, and optimality. In this paper, we will prove that among all adaptive designs with the same target allocations, the asymptotic selection bias has a lower bound. This lower bound is the optimal value of the selection bias of a design. Since  the patients are allocated sequentially by a system of probabilities in a randomized adaptive design, the  Shannon entropy is also a good measure of the randomness of the allocations. We will also prove that among all adaptive designs with the same target allocations, the asymptotic entropy also has a upper bound. This upper bound is the optimal value of the entropy of a design.   Though the ERADE of \cite{HZH09}  is efficient, it loses randomness because  in most circumstances it does not attain the optimal value of the selection bias and the optimal value of the entropy. The doubly adaptive biased coin design (DBCD) of \cite{Eisele94} and  \cite{EW95}, its modification of \cite{HZ04}, and most generalized Friedman's urn models attain the optimal value of the selection bias and the optimal value of the entropy. But they are not efficient. This leads to a challenge for researchers to {\em find a procedure
that can target any allocation proportion, preserves enough degree of randomization so that its selection bias is optimal, and attains the lower bound of allocation variability so that it is efficient}.

{\em 1.4. Objectives and organization of this paper}. In this paper,  we focus on the problem of the trade-off between efficiency, selection bias, and allocation proportion. We find efficient and randomized procedures for multiple-treatment trials that can adapt to any desired allocation proportions.
A comprehensive framework of efficient designs is proposed to serve as a paradigm for treatment allocation procedures in clinical trials, in which the allocation function can be either continuous or discrete. Under some wide satisfied conditions, we obtain the asymptotic properties including the almost sure convergence and the functional central limit theorem of both the sample allocation proportions and the estimators of the distribution parameter of the responses. The asymptotic variance-covariance matrix of the sample allocation proportions is proved to attain the lower bound.
By choosing a flexible continuous allocation function, we propose a new family of efficient DBCDs for multiple-treatment trials that can adapt to any desired allocation proportions,  attain the lower bound of variability so that it is efficient, and at the same time attain the optimal value of the selection bias and the optimal value of the entropy.

In the literature, asymptotic properties of response-adaptive designs are usually studied under continuous and differentiable allocation probability function by using Taylor expansion (c.f., \cite{Smith84a}, \cite{MPG01}, \cite{EW95}, \cite{HR06} and therein references). \cite{HZH09}  introduced a stopping time method to derive asymptotic properties for the case that the allocation probability at each step has only three possible values. These techniques do not work anymore for our general framework for multi treatments.   The key component in this paper for proving the asymptotic properties is showing that the distance between the sample allocation proportion and an efficient estimator of the target proportion can be neglected.  The recursive stochastic approximation algorithm, the Gaussian process approximation, and the Gaussian comparison theorem will be used for estimating the largest possible distance.

To summarize, the major contributions of this paper are as follows.

\begin{description}
  \item[\rm (a)] It formulates a general flexible framework of efficient adaptive-randomized designs for multiple treatments that are flexible to target any desired allocation proportion including allocation proportions that are derived from certain optimality criteria.  The framework with a general condition that is easily  verified  enables us    to derive with freedom new  efficient randomized-adaptive designs with desirable properties.
\item[\rm (b)] The paper derives asymptotic properties of the general efficient randomized-adaptive designs.
These asymptotic properties cover the existing asymptotic properties of \cite{HZH09}, \cite{Zhang2024} and \cite{AHZ2025}   as very special cases,  and can be utilized for  considering adaptive designs stopped a random time as well as for sequential monitoring of an adaptive design and other applications.
\item[\rm (c)] The paper proves that for each desired allocation proportion, adaptive designs have an optimal selection bias and an optimal entropy, which provides a baseline for comparison the degree of randomization of adaptive designs.
 \item[\rm (d)] With the general framework, the paper finds a family of adaptive designs which are not only efficient in the sense that they attain  the lower bound of the asymptotic variability, but also most random in sense that they have the optimal selection bias and optimal entropy.

\end{description}

The paper is organized as follows. In Section \ref{sectAssumption}, we give general assumptions on the adaptive designs in which both the historical assignments and treatment outcomes are implicated to update the allocation probability. In Section \ref{sectERADE}, a general condition on the allocation function is given for the adaptive design to be efficient, and the general framework of multiple-treatment efficient randomized-adaptive designs is proposed. Several new families of designs are defined by examples. In Section \ref{sectSelectionbias}, we consider the selection bias and entropy as measures of the randomness of a response-adaptive design, and show that there is a lower bound of the asymptotic selection bias and a upper bound of the asymptotic entropy among all response-adaptive designs with the same target allocation proportions. In Section \ref{sectbest},  a  family of designs is  shown asymptotically best in sense that it  attains both the Cram\'er-Rao bound of the allocation variabilities and the optimal values of  the selection bias and  entropy. General asymptotic properties are stated in Section \ref{sectAsmptotics} with proofs being given in the supplementary materials.

\section{General  assumptions}\label{sectAssumption}
\setcounter{equation}{0}

 We consider   $K$-treatment clinical
trials, $K\ge 2$. Suppose that the patients come to the clinical trial
sequentially and respond immediately. After the first $m$ patients are assigned to treatments and the responses
observed, the $(m+1)$-th
patient will be assigned to treatment $k$ with a probability
$p_{m,k}$, $k=1,\ldots, K$. The probabilities $p_{m,k}$s may depend on both the treatments assigned to and the
responses observed of the previous $m$ patients. Let $\bm
X_m=(X_{m,1},\ldots, X_{m,K})$ be the result of the $m$-th
assignment, i.e., the $k$-th component $X_{m,k}$ of $\bm X_m$ is
$1$ and other components are $0$ if the $m$-th patient is assigned to the
treatment $k$. We assume that
the patient responses with each treatment are i.i.d. with the probability distribution
indexed by $\bm \theta_k\in \mathcal R^d$ $(k = 1, \cdots, K)$. The cases of d = 1 (for binary response) or d = 2
(for normally distributed response) are typical.

Let $\bm N_m=(N_{m,1},\ldots, N_{m,K})=\sum_{j=1}^m\bm X_m$, where
$N_{m,k}$ is the number of the patients assigned to the treatment
$k$ in the first $m$ stages, $k=1,\ldots, K$. We assume that  the desired
allocation proportion of patients assigned to each treatment is a
function of $\bm\Theta=(\bm\theta_1,\ldots,\bm\theta_K)$ (see \cite{MPG01} and \cite{RSIHR01} for a related discussion). More specifically, the goal of the allocation
scheme is to have $\bm N_m/m\to \bm v=\bm \rho(\bm \Theta)$, where
$\bm \rho(\cdot)=(\rho_1(\cdot),\ldots,\rho_K(\cdot)): \Cal R^{d
\times K}\to (0,1)^K$ is a vector-valued function satisfying
$\bm \rho(\bm y)\bm 1^{\prime}=1$. For the target proportion $\bm\rho(\cdot)$, we assume the following condition.

\begin{condition} \label{ConA} The proportion function
$\bm y=(\bm y_1,\ldots,\bm y_K)=(y_{11},\ldots,y_{1d},\ldots,
y_{K1},\ldots,y_{Kd})\to \bm  \rho(\bm y): \Cal R^{d\times K}\to
(0,1)^K$  is a continuous function and is   differentiable at $\bm\Theta$ with $\bm\rho(\bm\Theta)=\bm v$.
\end{condition}

We further assume that the parameter estimate $\widehat{\bm \theta}_{m,k}$ of $\bm\theta_k$ based on an $m$-patient study has the Bahadur-type representation
\begin{equation}\label{eq2.1}
 \widehat{\bm \theta}_{m,k}
=N_{m,k}^{-1}\sum_{j=1}^m X_{j,k}\bm \xi_{j,k}+o(N_{m,k}^{-1/2}) \; a.s. \; \text{ as } m\to \infty
\end{equation}
for some i.i.d. sequences of random variables $\{\bm\xi_{j,k}, j = 1,2,\ldots\}$, where $\bm\xi_{j,k}=(\xi_{j,k1},\ldots, \xi_{j,kd})$ is the
response or a function of the response of the $j$th patient on treatment $k$. We also
write $\widehat{\bm \Theta}_{m,k}=( \widehat{\bm \theta}_{m,1},\ldots, \widehat{\bm \theta}_{m,K})$ and  $\bm \xi_m=(\bm \xi_{m,1},\ldots,\bm \xi_{m,K})$.

In most applications, the response distributions belong to an exponential family.
Then, we take $\bm\theta_k = \ep [\bm\xi_{1,k}]$, where $\bm\xi_{1,k}$s are the natural sufficient statistics, and
$\widehat{\bm\theta}_{m,k}$ is the average of the observed sufficient statistics. In practice, we may start
with $\bm\Theta_0 = (\bm\theta_{0,1},\ldots,\theta_{0,K})$ as an initial estimate of $\bm\Theta$ and use a modified sample mean
\begin{equation}\label{eq2.2}
 \widehat{\bm \theta}_{m,k}
=\frac{\sum_{j=1}^m X_{j,k}\bm \xi_{j,k}+\bm \theta_{0,k}}{N_{m,k}+1},
\end{equation}
to ensure a well-defined estimator even when no patients are assigned to treatment
$k$. The initial estimate $\bm\Theta_0$ is a guessed value of $\bm\Theta$ or an estimate of $\bm\Theta$ from
early trials.

\begin{remark}
When the distribution of $\bm \xi_{1,k}$ belongs to a family of
distributions with parameter $\bm \theta_k$,  the popular maximum
likelihood estimator $\widehat{\bm \theta}_{m,k}^{MLE}$ is usually
employed to estimate the parameter $\bm \theta_k$,
$k=1,2,\ldots,K$. Under suitable regular conditions, we can find
functions $\bm f_k$ with $\ep[\bm
f_k(\bm \xi_{j,k})]=\bm \theta_k$ and $\Var\{\bm
f_k(\bm \xi_{j,k})\}=\bm I_k^{-1}(\bm \theta_k)$ such that
$$
 \widehat{\bm \theta}_{m,k}^{MLE}
=\frac{\sum_{j=1}^m X_{j,k}\bm
f_k(\bm \xi_{j,k})+\bm \theta_{0,k}}{N_{m,k}+1}+o(N_{m,k}^{-1/2})\;\;
a.s., \quad k=1,\ldots, K.
$$
$\widehat{\bm \theta}_{m,k}^{MLE}$ has the same form as in
(\ref{eq2.1}) with $\bm f_k( \bm \xi_{1,k})$ taking the place of
$\bm \xi_{1,k}$, $k=1,\ldots, K$. Here $\bm I_k(\bm \theta_k)$ is the Fishser
information matrix of $\bm \theta_k$.
\end{remark}

For the responses $\{\bm\xi_{j,k}\}$ we assume the following condition.

\begin{condition}\label{ConB}
In the Bahadur-type representation \eqref{eq2.1}, $\bm\theta_k= \ep\bm\xi_{1,k}$ and
$\ep \|\bm  \xi_{1,k}\|^{2+\epsilon}<\infty$ for some $\epsilon>0$, $k=1,\ldots,K$.
\end{condition}

For introducing our efficient randomized-adaptive designs  we first consider the following general framework of designs which implicate both the sample proportion of patients assigned to each treatment and the current estimate of the desired allocation proportion.

  Let
$\bm g(m,\bm x,\bm  \rho) =\big(g_1(m,\bm x,\bm  \rho)$, $\ldots$,
$g_K(m,\bm x,\bm  \rho)\big):\Cal Z^+\times [0,1]^K\times[0,1]^K\to
[0,1]^K$ be the allocation function with $\bm g(m,\bm x,\bm  \rho)\bm
1^{\prime}= 1$. A general framework of the allocation procedure with the allocation function $\bm g(m,\bm x,\bm  \rho)$ is defined as follows.\newline

\noindent\hrulefill
\vspace{-5mm}
\begin{procedure}\label{AP:A}  The general framework of allocation.

\vspace{-5mm}
 \noindent\hrulefill
  \begin{algorithmic}[1]
    \Require
      $\bm g(m,\bm x,\bm  \rho)$: the allocation function;
      $\bm \rho(\cdot)$: the target proportion function;
      $n$: totally number of patients to be allocated.
    \Ensure
      $\bm N_n$.

    \State To start, we assign the first patient by using a complete randomization, and estimate $\bm\theta_k$ by an initial value $\bm\theta_{0,k}$ when there are not enough responses on treatment $k$ to get a valid estimator;
    \State assume that $m$ $(m\ge 1)$ patients have been assigned in the trial, and their responses are observed.  Let $\widehat{\bm \Theta}_m$ be
estimator defined based on the the first $m$
observations satisfying  \eqref{eq2.1}, and $\widehat{\bm\rho}_m =\bm\rho(\widehat{\bm\Theta}_m)$;
\State the $(m+1)$-th patient is assigned
to the treatment $k$ with probability
\begin{equation} \label{allocation}p_{m,k}=g_k\big(m,\frac{\bm
N_m}{m},\widehat{\bm \rho}_m\big),
\end{equation}
 $k=1,\ldots, K$; update the value $\bm N_{m+1}=\bm N_m+\bm X_{m+1}$ and the value of $\widehat{\bm \Theta}_{m+1}$;
    \State repeat Steps  2 and 3 until  $m=n$;
  \end{algorithmic}

  \vspace{-5mm}
\noindent\hrulefill
  \end{procedure}

In this framework, the  probabilities of allocating a patient to treatments are  adaptive according to both the current allocation proportions and observed responses.
The multiple-treatment DBCD proposed by \cite{HZ04}  is a special case of the general framework with   continuous allocation functions
\begin{equation}\label{dbcd}
 g_k(m,\bm x, \bm \rho)=g_k(\bm x, \bm \rho)
=\frac{ \rho_k\big(\frac{\rho_k}{x_k}\big)^{\gamma} }{
\sum_{j=1}^K \rho_j\big(\frac{\rho_j}{x_j}\big)^{\gamma}}, \quad k=1,\ldots, K,
\end{equation}
 where $\gamma\ge 0$ is a constant. The procedure proposed by \cite{MPG01}  is a special case of the DBCD with $\gamma=0$. The DBCD has the advantage that it is flexible to target any desirable allocation proportions and usually has small variabilities.  For a DBCD with allocation function defined in (\ref{dbcd}), \cite{HZ04} and \cite{HZCC08} proved that
\begin{equation}\label{normalityofDBCD}
 n^{1/2}\big(\frac{\bm N_n}{n}-\bm\rho(\bm\Theta)\big)\overset{\Cal D}\to
 N(\bm 0, \bm \Lambda_{\gamma})\;\;\text{ and }\;\;
 n^{1/2}\big(\widehat{\bm \rho}_n-\bm\rho(\bm\Theta)\big)\overset{\Cal D}\to
 N(\bm 0, \bm \Sigma),
 \end{equation}
 where
 $$\bm \Lambda_{\gamma}=\frac{1}{1+2\gamma}\big(diag(\bm\rho(\bm\Theta))-\bm\rho(\bm\Theta)^{\prime}\bm\rho(\bm\Theta)\big)
 +\frac{2(1+\gamma)}{1+2\gamma}\bm \Sigma $$
and the value of $\bm \Sigma=\bm \Sigma(\bm \Theta)$ is defined as follows.  Let $\bm v=(v_1,\ldots,v_K)=\bm\rho(\bm\Theta)$,
\begin{equation*}
 \bm
V_k=\Var(\bm \xi_{1,k})=\left(\Cov\big[\xi_{1,ki},\xi_{1,kj}\big];
i,j=1,\ldots,d\right), \quad k=1,\ldots, K,
\end{equation*}
\begin{equation}\label{eqVarinceTheta}
\bm V=diag(\frac 1{v_1}\bm V_1,\ldots,\frac 1{v_K}\bm V_K),
\end{equation}
\begin{equation}
\label{eqLB}\bm \Sigma= \big(\frac{\partial\bm
\rho}{\partial\bm y}\big|_{\bm \Theta}\big)^{\prime} \bm V
\frac{\partial\bm \rho}{\partial\bm y}\big|_{\bm \Theta}
=\sum_{k=1}^K \frac 1{v_k} \big(\frac{\partial\bm  \rho}{\partial
\bm y_k}\big|_{\bm \Theta}\big)^{\prime}
   \bm V_k \frac{\partial\bm
\rho}{\partial \bm y_k}\big|_{\bm \Theta}.
\end{equation}
The matrix $\bm \Sigma$  in the definition of $\bm\Lambda_{\gamma}$ is obviously  a lower bound of $\bm\Lambda_{\gamma}$.

\section{Efficient randomized-adaptive design (ERADE)} \label{sectERADE}
\setcounter{equation}{0}

 \cite{HRZ06}  had shown that, under suitable
regular conditions on the distributions of the responses,  $\bm \Sigma$   is   the lower bound of the
asymptotic variance-covariance matrices of the allocation proportion $\bm N_n/n$ for all adaptive designs with the same target allocation proportion $\bm\rho(\bm\Theta)$, i.e., if
$$\sqrt{n}\left(\bm N_n/n-\bm \rho(\bm \Theta)\right)\overset{\Cal D}\to N\big(\bm 0, \bm \Lambda(\bm \Theta)\big),$$
then $\bm \Lambda(\bm \Theta)\ge \bm \Sigma(\bm \Theta)$ for almost all $\bm\Theta$.
 A design that attains this lower bound is called an {\em asymptotically best} or {\em efficient} randomized-adaptive design. The BDCD is not efficient because $\bm \Lambda_{\gamma}>\bm \Sigma$. However, $\bm\Lambda_{\gamma}$ approaches to the lower bound $\bm\Sigma$ as $\gamma\nearrow \infty$.

{\em \ref{sectERADE}.1. ERADE for two treatments}.
For   two-treatment clinical trials ($K=2$), \cite{HZH09}  proposed an efficient randomized-adaptive design (ERADE) by defining the allocation probability as
\begin{equation}\label{eqHuZhHe} p_{m,k}=\left\{\begin{array}{ll}
                          \alpha \widehat{\rho}_{m,k} & \text{ if } N_{m,k}>m\widehat{\rho}_{m,k},\\ \widehat{\rho}_{m,k} &\text{ if } N_{m,k}=m\widehat{\rho}_{m,k},\\
1-\alpha(1-\widehat{\rho}_{m,k})& \text{ if } N_{m,k}<m\widehat{\rho}_{m,k},
                         \end{array}\;\; k=1,2,
 \right.
 \end{equation}
where $0<\alpha<1$ is a constant that reflects the degree of randomization. Efron's biased coin design is a special case of Hu, Zhang and He's procedure with $\widehat{\rho}_m\equiv 1/2$.   \cite{HZH09} showed that their procedure attains the lower bound of the asymptotic variability. \cite{Zhang2024} and \cite{AHZ2025} generalized the ERADE to the multiple-treatment case.  In the next subsection we introduce the general framework of ERADEs including the Efferent DBCDs as we will propose. 

{\em \ref{sectERADE}.2. ERADEs for multi treatments}. Motivated by the DBCD with the allocation functions (\ref{dbcd}),  we propose the following procedure that is an  efficient
randomized-adaptive design for multiple-treatment clinical trials.

\noindent\hrulefill
\vspace{-5mm}
\begin{procedure}\label{AP:B}  An Efficient DBCD (EDBCD)  for multi treatments.

\vspace{-5mm}
 \noindent\hrulefill
  \begin{algorithmic}[1]
\State  Define a weight function as
\begin{equation}\label{eqEprade1.1}  \psi(x)=\begin{cases} 1+\sqrt{x^{2\gamma}-1}, & x\ge 1,\\
x^{\gamma}, &0\le x\le 1,
\end{cases}
\end{equation}
where $\gamma>0$ is a constant;
\State define the allocation functions $\bm g=(g_1,\ldots,g_K)$  as
\begin{equation}\label{eqEprade1.2}  g_{k}\big(m,\bm x ,\bm \rho\big) =g_{k}\big(\bm x ,\bm \rho\big)=\frac{\rho_k\psi\left(\frac{\rho_k}{x_k}\right)}{\sum_{j=1}^K\rho_j\psi\left(\frac{\rho_j}{x_j}\right)}, \; k=1,\ldots,K;
\end{equation}
\State then the patients are randomized to treatments sequentially  by the Allocation Procedure \ref{AP:A} with the allocation function $\bm g$, i.e. the $(m+1)$-th patient is randomized to treatment $k$ with a probability  $p_{m,k}=g_k\big(\frac{\bm
N_m}{m},\widehat{\bm \rho}_m\big)$.
  \end{algorithmic}

  \vspace{-5mm}
\noindent\hrulefill
  \end{procedure}
  It is obvious that $x^{\gamma}\le \psi(x)\le \sqrt{2}x^{\gamma}$. Notice that $\psi(y_k/x_k)\sim (y_k/x_k)^{\gamma}$ as $y_k/x_k\to \infty$ or $y_k/x_k\to 0$, which is equivalent to the weight  $(y_k/x_k)^{\gamma}$ in the allocation function (\ref{dbcd}) of the DBCD. Following the suggestion of \cite{HR03}, $\gamma$ is suggested to chosen between $2$ and $4$ in practice.

 
  Besides EDBCD, the following procedure behaves similarly. 
  
\noindent\hrulefill
\vspace{-5mm}
\begin{procedure}\label{AP:C}   EDBCD2. 

\vspace{-5mm}
 \noindent\hrulefill
  \begin{algorithmic}[1]
\State Define the allocation functions $\bm g=(g_1,\ldots,g_K)$  as
\begin{equation}\label{eqEprade3.1}  g_{k}\big(m,\bm x ,\bm \rho\big) =g_{k}\big(\bm x ,\bm \rho\big)= \rho_k-  \min_j\frac{\rho_j\wedge(1-\rho_j)}{|x_j-\rho_j|^{\alpha}}( x_k-\rho_k),  
\end{equation}
$k=1,\ldots,K,$ where $0<\alpha<1$, and $\frac{0}{0}$ is defined to be $0$;
\State then the patients are randomized to treatments sequentially  by the Allocation Procedure \ref{AP:A} with the allocation function $\bm g$, i.e. the $(m+1)$-th patient is randomized to treatment $k$ with a probability  $p_{m,k}=g_k\big(\frac{\bm
N_m}{m},\widehat{\bm \rho}_m\big)$.
  \end{algorithmic}

  \vspace{-5mm}
\noindent\hrulefill
  \end{procedure}

For showing the efficiency of the EDBCD and EDBCD2,  we  first give the following general conditions so that a  design   defined as     the Procedure \ref{AP:A} with allocation probabilities $p_{m,k}$s will be a multiple-treatment efficient
randomized-adaptive design.

\begin{condition}\label{ConC} Suppose that the
allocation probabilities of the $(m+1)$-th patient $\bm g(m,\bm
N_m/m,\widehat{\bm \rho}_m)$ can be written in the following form:
\begin{align} \label{eqConditionforEff}p_{m,k}= g_k(m,\frac{\bm
N_m}{m},\widehat{\bm \rho}_m) & \le \widehat{\rho}_{m,k}+\lambda_{m,k}\big[\frac{N_{m,k}}{m}-
\widehat{\rho}_{m,k}\big] +\gamma_{m,k}\|\frac{\bm N_m}{m}-\widehat{\bm \rho}_m\|\\
& \text{ if } \;
N_{m,k} -m \widehat{\rho}_{m,k}\ge L_0, \; k=1,2,\ldots, K\nonumber
\end{align}
for some  $L_0\ge 0$, with the conditions that
$$ \limsup_{m\to \infty} \gamma_{m,k}\le 0 \;a.s.\; \text{ and } \limsup_{m\to \infty}\lambda_{m,k}\le \lambda_0<1\;\; a.s.,
\quad k=1,2,\ldots, K, \eqno{(C1)} $$
$$  \lambda_{m,k}\to -\infty \; a.s.  \;\; \text{ on the event }
\;\; \Big\{\frac{\bm N_m}{m}\to \bm v,\; \widehat{\bm \rho}_m\to \bm v\Big\},
\quad k=1,2,\ldots, K. \eqno{(C2)}
$$
\end{condition}

By symmetry, the Condition \ref{ConC} can be replaced by the following one.

\noindent{\bf Condition C$^{\prime}$} {\it  Suppose that the
allocation probabilities of the $(m+1)$-th patient $\bm g(m,\bm
N_m/m,\widehat{\bm \rho}_m)$ can be written in the following form:
\begin{align*}p_{m,k}= g_k(m,\frac{\bm
N_m}{m},\widehat{\bm \rho}_m) & \ge \widehat{\rho}_{m,k}-\lambda_{m,k}\big[\widehat{\rho}_{m,k}-\frac{N_{m,k}}{m}
\big]-\gamma_{m,k}\|\frac{\bm N_m}{m}-\widehat{\bm \rho}_m\| \\
& \text{ if } \;
m \widehat{\rho}_{m,k}-N_{m,k} \ge L_0, \; k=1,2,\ldots, K
\end{align*}
for some  $L_0\ge 0$, and Conditions (C1) and (C2) are satisfied.
}

\smallskip
  The following two theorems tell us that the asymptotic variability of a design with allocation probabilities satisfying Condition \ref{ConC} (resp. Condition  C$^{\prime}$) actually attains the lower bound, as long as the parameter estimate $\widehat{\bm\Theta}_n$ is efficient.

\begin{theorem} \label{thNormality}
Under Conditions \ref{ConA}-\ref{ConC},
\begin{equation}\label{eqNormality}
 n^{1/2}\big(\frac{\bm N_n}{n}-\bm v\big)\overset{\Cal D}\to
 N(\bm 0, \bm \Sigma)
 \end{equation}
and
\begin{equation}\label{eqNormalityforTheta} n^{1/2}\big(\widehat{\bm\Theta}_n-\bm\Theta\big)\overset{\Cal D}\to
 N(\bm 0, \bm V),
 \end{equation}
 where $\bm V$ and $\bm\Sigma$ are defined as in (\ref{eqVarinceTheta}) and (\ref{eqLB}), respectively.
\end{theorem}

\begin{theorem}\label{thLowerbound} Let $\bm I_k$ be the Fisher information matrix for parameter $\bm\theta_k$, and suppose that $\bm \xi_{1,k}$ is the natural sufficient statistics with $\Var(\bm\xi_{1,k})=\bm I_k^{-1}$. Then under Conditions \ref{ConA}-\ref{ConC}, the asymptotic variance-covariance matrix $\bm \Sigma$ of $\bm N_n/\sqrt{n}$ attains the Cr\'ame-Rao lower bound
\begin{equation}\label{eq:Lowerbound}\big(\frac{\partial\bm
\rho}{\partial\bm y}\big|_{\bm \Theta}\big)^{\prime} diag\Big( (v_1\bm I_1)^{-1},\ldots,(v_K\bm I_K)^{-1}\Big)
\frac{\partial\bm \rho}{\partial\bm y}\big|_{\bm \Theta}.
\end{equation}
\end{theorem}

 Theorem \ref{thLowerbound} follows from  Theorem \ref{thNormality} and the result of \cite{HRZ06}. (\ref{eqNormality}) and (\ref{eqNormalityforTheta}) give the asymptotic normalities of both the sample allocation proportions and the estimates of the distribution parameters. 
 We will also show that when the trials are stopped at random time and so the   fix sample size $n$  is replaced by a random sample size $\tau_n$,  the asymptotic normalities remains true.  Such conclusions are implied by general asymptotic properties given in Section \ref{sectAsmptotics}, where    the strong consistency, functional central limit theorems and the Gaussian approximation which are capable of various application circumstances, will be given. And Theorem \ref{thNormality} is a corollary of them.

It is easily seen that the EDBCD2 satisfies Condition \ref{ConC} with  $\gamma_{m,k}=0$, $\lambda_0=0$  and $\lambda_{m,k}=-\min_j\frac{\widehat{\rho}_{m,j}\wedge(1-\widehat{\rho}_{m,j})}{| N_{m,j}/m-\widehat{\rho}_{m,j}|^{\alpha}}$. 
For
 verifying that  EDBCD satisfies  Condition \ref{ConC}, we let $\varpi_k=\psi(\rho_k/x_k)$, $g_{k}=g_{k}\big(m,\bm x ,\bm \rho\big)$. Note $\varpi_k\ge 1$ when $x_k\le \rho_k$, and $\varpi_k\le 1$ when $x_k\ge\rho_k$, and  $\psi(x)\ge x^{\gamma}$. It follows that
 $$  \sum_{j=1}^K \rho_j\varpi_j\ge \sum_{j=1}^K \rho_j\big(\rho_j/x_j\big)^{\gamma}\ge 1, $$
where the last inequality is due to the fact that
 $$\min\Big\{\sum_{j=1}^K \rho_j\big(\rho_j/x_j\big)^{\gamma}: \sum_{j=1}^K x_j=\sum_{j=1}^K \rho_j=1, \rho_j> 0, x_j\ge 0, j=1,\ldots,K\Big\}=1.$$
   Hence
$$ g_k=\frac{\rho_k\varpi_k}{\sum_{j=1}^K \rho_j\varpi_j}\le \frac{\rho_k}{\sum_{j=1}^K \rho_j\varpi_j}\le \rho_k\; \text{ if } \; x_k>\rho_k. $$
So (C1) is satisfied with $\lambda_0=0$ and $\gamma_{m,k}\equiv 0$. Further, for $x_k>\rho_k$,
\begin{align*}
& g_k- \rho_k\le -\frac{ \rho_k}{\sum_{j=1}^K \rho_j\varpi_j}\sum_{j=1}^K \rho_j(\varpi_j-1) \\
=&-\frac{ \rho_k}{\sum_{j=1}^K \rho_j\varpi_j}\sum_{j=1}^K \rho_j(\varpi_j-1)I\{x_j< \rho_j\}+\frac{ \rho_k}{\sum_{j=1}^K \rho_j\varpi_j}\sum_{j=1}^K \rho_j(1-\varpi_j)I\{x_j> \rho_j\}\\
=&-\frac{ \rho_k}{\sum_{j=1}^K \rho_j\varpi_j}\sum_{j=1}^K( \rho_j-x_j)^+ \frac{\rho_j}{x_j}\frac{\varpi_j-1}{( \rho_j/x_j-1)}I\{x_j< \rho_j\}\\
&+\frac{ \rho_k}{\sum_{j=1}^K \rho_j\varpi_j}\sum_{j=1}^K( x_j-\rho_j)^+ \frac{\rho_j}{x_j}\frac{1-\varpi_j}{(1- \rho_j/x_j)}I\{x_j> \rho_j\}.
\end{align*}
Note as $\bm x\to \bm v$ and $\bm \rho\to \bm v$,
$$\frac{ \rho_k}{\sum_{j=1}^K \rho_j\varpi_j}\to v_k $$
$$\frac{\rho_j}{x_j}\frac{\varpi_j-1}{( \rho_j/x_j-1)}=\frac{\rho_j}{x_j}\frac{\sqrt{(\rho_j/x_j)^{2\gamma}-1}}{( \rho_j/x_j-1)}\to +\infty\;\; \text{ along }\; x_j< \rho_j, \;\; \text{ and} $$
$$\frac{\rho_j}{x_j}\frac{1-\varpi_j}{( 1-\rho_j/x_j)}=\frac{\rho_j}{x_j}\frac{1-(\rho_j/x_j)^{\gamma}}{( 1-\rho_j/x_j)}\to \gamma\;\; \text{ along }\; x_j> \rho_j. $$
It follows that, for any $\lambda>0$, whenever $\|\bm x-\bm v\|+\|\bm \rho-\bm v\|$ is small enough we will have
\begin{align*}
g_k- \rho_k\le &   -(\lambda+2\gamma) \sum_{j=1}^K( \rho_j-x_j)^++ 2\gamma  \sum_{j=1}^K( x_j-\rho_j)^+ \\
 = &-\lambda \sum_{j=1}^K(x_j- \rho_j)^+\le -\lambda (x_k- \rho_k), \;\; x_k> \rho_k,
\end{align*}
because $\sum_{j=1}^K(x_j- \rho_j)^+-\sum_{j=1}^K( \rho_j-x_j)^+=\sum_{j=1}^K(x_j- \rho_j)=1-1=0$. Condition (C2) is satisfied.
Condition \ref{ConC} is verified. Hence, the randomization procedure is efficient as long as the parameter estimate $\widehat{\bm\Theta}_n$ is efficient.
\begin{corollary}\label{Cor3.1}  For EDBCD (Procedure \ref{AP:B}) and EDBCD2 (Procedure \ref{AP:C}), under Conditions \ref{ConA} and \ref{ConB}, (\ref{eqNormality}) and (\ref{eqNormalityforTheta}) hold.
\end{corollary}

Condition \ref{ConC} or C$^{\prime}$ provides us a flexible framework to define various ERADEs. It is easily seen that Hu, Zhang and He's procedure satisfies Condition \ref{ConC} (see Example \ref{example1} bellow). Besides the randomization  Procedures \ref{AP:B} and \ref{AP:C},  we can define other efficient response-adaptive randomization procedures by suitably choosing the allocation function. The first is a multi-treatment generalization of Hu, Zhang and He's procedure.

\begin{example}\label{example1} \citep{Zhang2024} Let $0<\alpha<1$. Write $G_k=N_{m,k}/m-\widehat{\rho}_{m,k}$. Define
$$ p_{m,k}=g_k\big(m,\frac{\bm
N_m}{m},\widehat{\bm \rho}_m\big)=\begin{cases}\alpha\widehat{\rho}_{m,k}, & \ G_k>0, \\
\widehat{\rho}_{m,k}, & \ G_k=0,\\
\beta\widehat{\rho}_{m,k}, &\   G_k<0,
\end{cases}
$$
where
$$\beta=\frac{1-\alpha\sum\limits_{j:G_j>0}\widehat{\rho}_{m,j}-\sum\limits_{j:G_j=0}\widehat{\rho}_{m,j}}{\sum\limits_{j:G_j<0}\widehat{\rho}_{m,j}}. $$
When $K=2$,  this procedure is the procedure  of \cite{HZH09}. In \cite{AHZ2025}, the definition of the allocation probability is a little different, where, when $G_k<0$, the probability $p_{m,k}$ is defined to be 
$$\frac{(1-\alpha)\sum_{j:G_j>0}\widehat{\rho}_{m,j}}{|\{j:G_j<0\}|}+\widehat{\rho}_{m,k}. $$
\end{example}

 It is easily seen that, if $N_{m,k}/m>\widehat{\rho}_{m,k}$, then
$$ p_{m,k}=\alpha \widehat{\rho}_{m,k}= \widehat{\rho}_{m,k}-\frac{(1-\alpha)\widehat{\rho}_{m,k}}{N_{m,k}/m-\widehat{\rho}_{m,k}}(N_{m,k}/m-\widehat{\rho}_{m,k}). $$
Condition \ref{ConC} is satisfied with $\lambda_0=0$, $\gamma_{m,k}\equiv 0$ and $\lambda_{m,k}=-\frac{(1-\alpha)\widehat{\rho}_{m,k}}{(N_{m,k}/m-\widehat{\rho}_{m,k})^+}$. Here $x^+=\max\{x,0\}$.

\bigskip

\begin{example}\label{example2}  {\em  (A step-up procedure)} Let $G_k=N_{m,k}-m\widehat{\rho}_{m,k}$ (or $G_k=N_{m,k}/(m\widehat{\rho}_{m,k})-1$), $k=1,2,\ldots,K$. One can rank treatments according to their values, treatment (1) having minimal $G_k$, treatment (2) having the next smallest $G_k$, etc., so that $(s)<(t)$ iff $G_{(s)}\le G_{(t)}$. In the case of ties a random ordering can be determined. 

Let $\alpha_1\ge \alpha_2\ge\cdots\ge \alpha_K>\alpha>0$. Define the weights of allocation probabilities for treatments $(1),(2),\ldots, (K)$ in a step-up way as
$$ \psi_{(k)}=\alpha_k\; \text{ if }\;  G_{(k)}\le 0\;\; \text{ and }\;\; \psi_{(k)}=\alpha\; \text{ if }\; G_{(k)}>0.  $$
Then the allocation probabilities are defined to be the ratios of the weighted targets
$$ p_{m,k}=g_{k}\big(m,\frac{\bm
N_m}{m},\widehat{\bm \rho}_m\big)=\frac{ \widehat{\rho}_{m,k}\psi_k}{\sum_{j=1}^K\widehat{\rho}_{m,j}\psi_j}.
$$
\end{example}
It is easily seen that, if $N_{m,k}/m>\widehat{\rho}_{m,k}$, then
 \begin{align*}
 p_{m,k}=& \frac{ \alpha\widehat{\rho}_{m,k} }{\sum_{j=1}^{k_0-1}\widehat{\rho}_{m,(j)}\alpha_j+\alpha\sum_{j=k_0}^K\widehat{\rho}_{m,(j)}}\le \frac{ \alpha\widehat{\rho}_{m,k} }{\alpha_K\sum_{j=1}^{k_0-1}\widehat{\rho}_{m,(j)}+\alpha\sum_{j=k_0}^K\widehat{\rho}_{m,(j)}}\\
 =& \widehat{\rho}_{m,k}-\frac{ \widehat{\rho}_{m,k}(\alpha_K-\alpha)\sum_{j=1}^{k_0-1}\widehat{\rho}_{m,(j) }}{\alpha_K\sum_{j=1}^{k_0-1}\widehat{\rho}_{m,(j)}+\alpha\sum_{j=k_0}^K\widehat{\rho}_{m,(j)}}
 \le \widehat{\rho}_{m,k}-\frac{ (\alpha_K-\alpha)\min_j\widehat{\rho}_{m,j }^2 }{\alpha_K}.
 \end{align*}
Condition \ref{ConC} is satisfied is satisfied with $\lambda_0=0$ and $\lambda_{m,k}=\frac{ (\alpha_K-\alpha)\min_j\widehat{\rho}_{m,j }^2}{\alpha_K(N_{m,k}/m-\widehat{\rho}_{m,k})^+}$.

\begin{example} \label{example3} {\em  (A step-down procedure)}  Let $0\le \alpha_K\le \alpha_{K-1}\le\cdots\le \alpha_1<\alpha$.  Define the weights of allocation probabilities for treatments $(K),(K-1),\ldots, (1)$ in a step-down way as
$$ \psi_{(k)}=\alpha_k \; \text{ if }\; G_{(k)}>0   \;\;\text{ and }\;\; \psi_{(k)}=\alpha\;\text{ if } \; G_{(k)}\le 0   $$
and let
$$ p_{m,k}=g_{k}\big(m,\frac{\bm
N_m}{m},\widehat{\bm \rho}_m\big)=\frac{ \widehat{\rho}_{m,k}\psi_k}{\sum_{j=1}^K\widehat{\rho}_{m,j}\psi_j}.
$$
\end{example}
Similarly as in Example \ref{example2},  Condition C$^{\prime}$  is satisfied.

\bigskip
A general allocation procedure can  defined as follows such that the allocation probability $p_{m,k}$  is proportional to the weighted target proportion $\widehat{\rho}_{m,k}$ with  the weight being a non-increasing function of $N_{m,k}-m\widehat{\rho}_{m,k}$.

\begin{example}\label{example4}  Let $\psi(\cdot)\ge 0$ be a monotone non-increasing function on $\mathcal{R}$ with $\psi(0)>\psi(+\infty)$. Define
\begin{equation}\label{eqexample4} p_{m,k}=g_{k}\big(m,\frac{\bm
N_m}{m},\widehat{\bm \rho}_m\big)=\frac{ \widehat{\rho}_{m,k}\psi(N_{m,k}-m\widehat{\rho}_{m,k})}{\sum_{j=1}^K\widehat{\rho}_{m,j}
\psi(N_{m,j}-m\widehat{\rho}_{m,j})}.
\end{equation}
In practices,  one can choose $\psi(x)=e^{-x}$ or $1-\Phi(x)$, where $\Phi(x)$ is a standard normal distribution. 
\end{example}

For verifying that the allocation probability in (\ref{eqexample4}) satisfies the Condition \ref{ConC}, we let $\alpha=\psi(+\infty)$, $\beta=\psi(0)$. Choose an $\epsilon$ with $0<\epsilon<(\beta-\alpha)\underset{m,j}\min \,\widehat{\rho}_{m,j}/2$ and a $L_0$ such that
$\psi(L_0)<\alpha+\epsilon$. Write $G_k=N_{m,k}/m-\widehat{\rho}_{m,k}$. If $N_{m,k}-m\widehat{\rho}_{m,k}>L_0$, then
\begin{align*}
p_{m,k}\le& \frac{(\alpha+\epsilon)\widehat{\rho}_{m,k}}{\sum_{j:G_j\le 0}\beta \widehat{\rho}_{m,j}+\alpha\sum_{j:G_j<0}\widehat{\rho}_{m,j}}\le \frac{(\alpha+\epsilon)\widehat{\rho}_{m,k}}{(\beta-\alpha)\min_j \widehat{\rho}_{m,j}+\alpha }\\
\le & \widehat{\rho}_{m,k}-\widehat{\rho}_{m,k}\frac{(\beta-\alpha)\min_j \widehat{\rho}_{m,j}-\epsilon}{(\beta-\alpha)\min_j \widehat{\rho}_{m,j}+\alpha }
\le \widehat{\rho}_{m,k}-\widehat{\rho}_{m,k}\frac{(\beta-\alpha)\min_j \widehat{\rho}_{m,j} }{2\beta }.
\end{align*}
Condition \ref{ConC} is verified.

 In (\ref{eqexample4}), a  function $\psi(x)=\frac{1}{|x|^{\alpha}+1}$ $(x>0)$, $\psi(-x)=1-\psi(x)$ ($x<0$)  is considered   for balancing two treatments ($\widehat{\rho}_{m,k}\equiv 1/2$) by \cite{AG04}, who
  modified Efron's procedure and  proposed a class of adjustable biased coin designs (ABCD) by defining the allocation probability as a decreasing function of the treatment  imbalance $(2N_{m,1}-m)$. \cite{AG04}    showed that within the class of ABCDs suitable designs can be chosen with less predictability and expected
imbalance than the Efron's BCD that only utilizes  the sign of $(2N_{m,1}-m)$ in the allocation probability. ABCD is also showed to have advantages over the designs of \cite{Wei78b} and \cite{Smith84a,Smith84b} in term of efficiency and power (c.f.,  \cite{Antognini08}).   By letting $\widehat{\rho}_{m,k}=1/K$, our allocation procedure with (\ref{eqexample4}) is a multiple-treatment generalization of the ABCDs.

The difference between the efficient DBCDs (Procedures \ref{AP:B} and \ref{AP:C}) and those in Example \ref{example1}-\ref{example4} is that Procedures \ref{AP:B} and \ref{AP:C} use an allocation function which is continuous both in the current sample allocation proportions $\bm N_m/m$ and the sequential  estimator $\widehat{\bm \rho}_m$ of the target allocation proportions to randomize patients sequentially. Besides the efficiency, such procedures have also good properties in term of randomness, which will be discussed in the next two sections.

\section{Degree of randomness  of a response-adaptive design}\label{sectSelectionbias}
\setcounter{equation}{0}

{\em \ref{sectSelectionbias}.1. Optimal value of selection bias and entropy}.  The selection bias (lack of randomness) of the design shows in the possibility for the experimenter to guess partially the sequence of treatment allocations. As advocated in \cite{Efron71},  \cite{AG04}, etc,  a natural measure of the selection bias of a sequential design is  the expected percentage of correct guesses the experimenter can make if he guesses optimally.
We let $J_m =1$ if the $m$th assignment
is guessed correctly, and $J_m =0$ otherwise.
The selection bias of the design is then defined by
$$SB_n=\ep\left[\frac{1}{n}\sum_{m=1}^n J_m\right]=\frac{1}{n}\sum_{m=1}^n \pr(J_m=1). $$
Every guessing strategy is equally useless against complete randomization, yielding the expected percentage
of correct guesses $SB_n=\frac{1}{K}$ in any trials with $K$ treatments, and this is the minimum value of the selection bias of randomization procedures for balancing $K$ treatments.   For response-adaptive designs, the optimal guessing strategy is to guess the under-represented treatment, at each step,
on the basis of which has the highest probability
of being allocated, with no preferred guess in case of a tie. Thus, for   the optimal guessing strategy against  an adaptive design with allocation probabilities $p_{m,k}$s, the guess corrected probability of the $m$th assignment is
$$ \pr(J_m=1|\mathscr{F}_{m-1})=\max_k\pr(X_{m,k}=1|\mathscr{F}_{m-1})=\max_k p_{m-1,k}, $$
where $\mathscr{F}_{m-1}$ is the history sigma field.
  Hence $SB_n= \frac{1}{n}\sum_{m=1}^n \ep[\max_kp_{m-1,k}]$. It is obvious that $\frac{1}{K}\le SB_n\le 1$. When the allocation is not balanced, the minimum value  $1/K$ can not be attained.
 We define $SB=\lim_nSB_n$ to be the asymptotic selection bias if the limit exists.  The next theorem tells us that, for   adaptive designs with a target allocation proportion $\bm \rho(\bm\Theta)$,  the lower bound of the asymptotic selection bias is $\max_k\rho_k(\bm\Theta)$. The distance between the selection bias and the lower bound is a measure of the degree of lack of randomness.

\begin{theorem}\label{thSB} For any adaptive design, if $\bm N_n/n\to \bm \rho(\bm\Theta)$ in probability, then \begin{equation}\label{eqthSB1.1}\liminf_{n\to \infty} SB_n\ge \max_k \rho_k(\bm \Theta).
\end{equation}
Further
\begin{equation}\label{eqthSB1.2}\liminf_{n\to \infty}\frac{1}{n}  \sum_{m=1}^n J_m\ge \max_k \rho_k(\bm \Theta)\;\; a.s.\; \text{ if }\; \frac{\bm N_n}{n}\to \bm \rho(\bm\Theta)\;\;a.s.
\end{equation}
\end{theorem}

Besides the selection bias, the entropy is also a good measure of the randomness. For a set of probabilities $\bm p=(p_1,\ldots,p_K)$, the entropy of $\bm p$ is defined by
$$ H(\bm p)=-\sum_{k=1}^K p_k\log p_k. $$
\begin{definition}
The entropy of the design with allocation probabilities $p_{m,k}$s is  defined by
$$Ent_n=\ep\left[\frac{1}{n}\sum_{m=1}^n H(\bm p_{m-1})\right] \text{ where } \bm p_{m}=(p_{m,1},\ldots,p_{m,k}). $$
And,   $Ent=\lim_nEnt_n$ is defined to be the asymptotic entropy if the limit exists.
\end{definition}
A larger entropy means more randomness.
It is obvious that $0\le Ent_n\le \log K$. $Ent_n$ takes its largest value $\log K$ when and only when the randomization is a complete randomization in which $p_{m,k}=1/K$ for all $k$ and $m$. $Ent_n$ takes its smallest value $0$ when and only when each patient is assigned deterministically.

 Note $  -x\log x\le -x_0\log x_0-(1+\log x_0)(x-x_0)-\frac{1}{2}(x-x_0)^2$ ($0\le x,x_0\le 1$). We have
\begin{align}\label{eq:entropy1}
\frac{1}{n}\sum_{m=1}^n H(\bm p_{m-1})-H(\bm v)\le & -\sum_{k=1}^K \frac{1+\log v_k}{n}\sum_{m=1}^n(p_{m-1,k}-v_k)-\frac{1}{2}\sum_{k=1}^n(p_{m-1,k}-v_k)^2\nonumber\\
=& \sum_{k=1}^K    \frac{\log v_k^{-1}}{n}\sum_{m=1}^n(p_{m-1,k}-v_k)-\frac{1}{2}\sum_{k=1}^n(p_{m-1,k}-v_k)^2,
\end{align}
where  $\bm v=\bm\rho(\bm\Theta)$. It follows that
\begin{equation}\label{eq:entropy2} Ent_n-H(\bm v)\le \sum_{k=1}^K \log v_k^{-1}\left(\ep\big[\frac{N_{n,k}}{n}\big]-v_k \right)-\frac{1}{2n}\sum_{m=1}^n \ep[(p_{m-1,k}-v_k)^2].
\end{equation}
Hence, under the restriction that $\ep\big[\frac{\bm N_n}{n}\big]=\bm v$,   $Ent_n$ attains its largest value $H(\bm v)$ when and only when $p_{m-1,k}=v_k$ a.s. for all $m$ and $k$. In this case, the randomization procedure is  purply random in sense that the patients are assigned independently and with the same probabilities $v_k$, $k=1,\ldots,K$.

In general, for a random vector $\bm Z$ with its probability density function or probability mass function $f(\bm Z)$, the Shannon entropy of $\bm Z$ is defined by
$$ H(\bm Z)=-\ep\left[\log f(\bm Z)\right]=-\int f(\bm z)\log f(\bm z)d\bm z. $$
In a randomized adaptive design, the collected data is $\bm Z_n=(X_{m,k}, X_{m,k}\bm\xi_{m,k};k=1,\ldots,K, m=1,\ldots, n)$ with the distribution as
$$ f(\bm Z)=\prod_{m=1}^n\prod_{k=1}^K\big(p_{m-1,k} f_k(\bm \xi_{m,k})\big)^{X_{m,k}}, $$
where $f_k(\bm \xi_{m,k})$ is the probability density function or probability mass function of $\bm \xi_{m,k}$. Hence,
\begin{align*}
 & H(\bm Z_n)=  -\ep\left[\sum_{m=1}^n \sum_{k=1}^K X_{m,k}\big(\log p_{m-1,k}+\log f_k(\bm \xi_{m,k})\big)\right]\\
 =& \ep\left[\sum_{m=1}^n H(\bm p_{m-1})\right]+\sum_{k=1}^K\ep[N_{n,k}]H(\bm\xi_{1,k})=n \left( Ent_n+ \sum_{k=1}^K\ep\big[\frac{N_{n,k}}{n}\big]H(\bm\xi_{1,k})\right).
 \end{align*}
 $H(\bm Z_n)$ can be regarded as a measure of the randomness of the whole system in a randomization design.  According to the above formula, the randomness of the whole system consists of two parts, one of which is the randomness of the allocations which is measured by $Ent_n$, and the other is the randomness of the responses which is measured by $H(\bm\xi_{1,k})$s. Under the restriction that $\ep\big[\frac{\bm N_n}{n}\big]=\bm v$,   $H(\bm Z_n)$ attains its largest value $n[H(\bm v)+\sum_{k=1}^Kv_kH(\bm \xi_{1,k})]$ when and only when $p_{m-1,k}=v_k$ a.s. for all $m$ and $k$. Further, if $\frac{\bm N_n}n\to \bm v$ in probability, then $\lim_n H(\bm Z_n)/n=\lim_n Ent_n+\sum_{k=1}^Kv_k H(\bm \xi_{1,k})$.

\begin{theorem}\label{thEntropy} (a) For any adaptive design, if $\bm N_n/n\to \bm \rho(\bm\Theta)$ in probability, then \begin{equation}\label{eqthEntropy1.1}\limsup_{n\to \infty}Ent_n\le  H\left(\bm \rho(\bm \Theta)\right).
\end{equation}
(b) If $\bm N_n/n\to \bm \rho(\bm\Theta)$ a.s., then
\begin{equation}\label{eqthEntropy1.2}\limsup_{n\to \infty}\frac{1}{n}\sum_{m=1}^n H(\bm p_{m-1})\le H\left(\rho(\bm \Theta)\right)\;\; a.s.
\end{equation}
(c) Suppose $\bm N_n/n\to \bm \rho(\bm\Theta)$ in probability. Then
\begin{equation}\label{eqthEntropy1.3}\lim_{n\to \infty}Ent_n=  H\left(\bm \rho(\bm \Theta)\right) \text{ if and only if } \frac{1}{n}\sum_{m=0}^{n-1}
\|\bm p_{m}-\bm \rho(\bm \Theta)\|\to 0 \text{ in probability}.
\end{equation}
(d) Suppose $\bm N_n/n\to \bm \rho(\bm\Theta)$ in probability and $Ent_n\to H\big(\bm \rho(\bm\Theta)\big)$. Then
$SB_n\to \max_k\rho_k(\bm \Theta)$.
\end{theorem}
Theorem \ref{thEntropy} (a) tells us that, for  adaptive designs with a same target allocation proportion $\bm \rho(\bm\Theta)$,  the upper  bound of the asymptotic entropy is $H\big(\bm\rho(\bm\Theta)\big)$, and Theorem \ref{thEntropy} (c) gives the sufficient and necessary condition for an adaptive design to attain this upper bound. Theorem \ref{thEntropy} (d) tells us that if the asymptotic entropy of a adaptive design is optimal, then the  selection bias is also optimal.

\begin{definition} For a  target allocation proportion $\bm\rho(\bm\Theta)$, the lower bound $\max_k\rho_k(\bm\Theta)$  of the asymptotic selection bias is called the optimal selection bias, and, the upper bound $H\big(\bm\rho(\bm\Theta)\big)$ of the asymptotic entropy is called the optimal entropy. An adaptive design of which the asymptotic selection bias attains the  optimal one is called the most random adaptive design.
\end{definition}

{\em \ref{sectSelectionbias}.2. Degree of Randomness of Hu, Zhang and He's procedure}. Intuitively, the value of $\alpha$ in the procedure  of \cite{HZH09}  (c.f., (\ref{eqHuZhHe})) reflects the degree of randomization, which is also verified by simulation studies.
   The next theorem gives a theoretical result on the selection bias and entropy of this procedure.
\begin{theorem} \label{thSBHuZhHe} Consider Hu, Zhang and He's procedure for two-treatment clinical trials with allocation probability defined in (\ref{eqHuZhHe}).  Without loss of generality, we assume $v_1=:\rho_1(\bm\Theta)\ge 1/2$. Suppose that $\rho_1(\widehat{\bm\Theta}_m)$ has no mass at a rational point, i.e., $\pr(\rho_1(\widehat{\bm\Theta}_m)=v)=0$   for any rational $v\in[0,1]$. This condition will be satisfied when the sets $\{\bm y_1: \rho_1(\bm y_1, \bm y_2)=v$ for fixed $\bm y_2\}$, $\{\bm y_2: \rho_1(\bm y_1, \bm y_2)=v$ for fixed $\bm y_1\}$ have Lebesgue measure zeros and the estimator $\widehat{\bm\Theta}_m$ is given by (\ref{eq2.2}) with continuous responses. Then for   the optimal guessing strategy we have
\begin{equation} \lim_{n\to \infty} SB_n=SB(\alpha)\;\;\text{ and }\;\; \lim_{n\to \infty}\frac{1}{n}  \sum_{m=1}^n J_m=SB(\alpha) \; a.s.,
\end{equation}
where
$$SB(\alpha)=\left\{\begin{array}{ll} v_1+(1-v_1) \{1-2\alpha v_1\}  &  \text{ if } v_1\le \frac{1}{2\alpha}, \\
 v_1 &  \text{ if } v_1\ge \frac{1}{2\alpha}
  \end{array}\right..$$
  Also
  \begin{align} &\lim_{n\to \infty} Ent_n=Ent(\alpha) \;\;\text{ and }\;\;   \lim_{n\to \infty}\frac{1}{n}\sum_{m=1}^n H(\bm p_{m-1})= Ent(\alpha) \;\; a.s.,
\end{align}
where $Ent(\alpha)=I(\alpha v_1)(1-v_1)+I(\alpha(1-v_1))v_1$ and  $I(x)=-x\log x-(1-x)\log (1-x)$.
\end{theorem}

The asymptotic selection bias $SB(\alpha)$ of Hu, Zhang and He's procedure is a    non-increasing function of the parameter $\alpha$. As $\alpha$ approaches $0$, it will approach $1$ which is the selection bias of a deterministic allocation procedure, and as $\alpha$ goes to $1$,  it will approach the lower bound. It is larger than the lower bound unless $1/2<\alpha<1$ and $\rho_1(\bm\Theta)\vee \rho_2(\bm\Theta)\ge 1/(2\alpha)$.  The asymptotic entropy $Ent(\alpha)$ of Hu, Zhang and He's procedure is a strictly monotone function of the parameter $\alpha$. As $\alpha$ approaches $0$, the asymptotic entropy will approach $0$ which is the entropy of a deterministic allocation procedure, and as $\alpha$ approaches $1$, the asymptotic entropy will approach the upper bound $H(v_1)$. $Ent(\alpha)$ is more sensitive to $\alpha$ than $SB(\alpha)$. 

\smallskip
{\em \ref{sectSelectionbias}.3. Proofs}. To end this section, we give the proofs of the theorems given in this section.

 {\bf Proof of Theorems \ref{thSB} and \ref{thEntropy}.}  Write $\bm v=\bm\rho(\bm\Theta)$. Note
 $ \pr(J_m=1|\mathscr{F}_{m-1})=\max_k p_{m-1,k} $ and $\pr(J_m=1)\ge \ep[p_{m-1,k}]=\ep [X_{m,k}]$. It follows that
 $$ SB_n\ge \ep\big[\frac{N_{n,k}}{n}\big]\to v_k,\;\; \forall \;k. $$
 (\ref{eqthSB1.1}) is proved.  For (\ref{eqthSB1.2}), note $ \ep[J_m|\mathscr{F}_{m-1}] \ge  \ep[X_{m,k}|\mathscr{F}_{m-1}])$ and
 $$ \lim_{n\to \infty} \frac{1}{n}\sum_{m=1}^n (J_m-\ep[J_m|\mathscr{F}_{m-1}])=0=\lim_{n\to \infty} \frac{1}{n}\sum_{m=1}^n (X_{m,k}-\ep[X_{m,k}|\mathscr{F}_{m-1}])\;\;a.s. $$
 It follows that
$$\liminf_{n\to \infty}\frac{1}{n}  \sum_{m=1}^n J_m\ge   \lim_{n\to \infty} \frac{N_{n,k}}{n}\ge v_k\; a.s.,\;\; \forall \;k. $$

On the other hand,  by \eqref{eq:entropy2} and \eqref{eq:entropy1} we have
$$ Ent_n-H(\bm v)\le \sum_{k=1}^K \log v_k^{-1}\left(\ep\big[\frac{N_{n,k}}{n}\big]-v_k \right)\to 0 \text{ if } \frac{\bm N_{n}}{n} \overset{P}\to \bm v$$
and
\begin{align*}
  \frac{1}{n}\sum_{m=1}^n H(\bm p_{m-1})-H(\bm v)
\le  & \sum_{k=1}^K\log v_k^{-1}\left( \frac{N_{n,k}}{n} -v_k -\frac{\sum_{m=1}(X_{m,k}-\ep[X_{m,k}|\mathscr{F}_{m-1}])}{n}\right) \\
& \to 0 \;\; a.s. \text{ if } \; \frac{\bm N_{n}}{n}  \to \bm v\;\; a.s.,
\end{align*}
which imples \eqref{eqthEntropy1.1} and \eqref{eqthEntropy1.2}, respectively. Further, if $\frac{1}{n}\sum_{m=0}^{n-1}
\|\bm p_{m}-\bm v\|\to 0$ in probability, then
\begin{align*}
 |Ent_n-H(\bm v)|\le & \sup_{\|\bm x-\bm v\|\le \epsilon}|H(\bm x)-h(\bm v)|+\ep\left[\frac{1}{n}\sum_{m=0}^{n-1}I\{\|\bm p_m-\bm v\|\ge \epsilon\}\right]\\
  \le & \sup_{\|\bm x-\bm v\|\le \epsilon}|H(\bm x)-H(\bm v)|+\epsilon^{-1}\ep\left[\frac{1}{n}\sum_{m=0}^{n-1}\|\bm p_m-\bm v\|\right]\\
  & \to 0 \;\;\text{ as } n\to \infty \text{ and then } \epsilon\to 0.
\end{align*}
If $Ent_n\to H(\bm v)$ and $\frac{\bm N_n}{n}\to \bm v$ in probability, then
\begin{align*}&  \ep\left[\frac{1}{n}\sum_{m=0}^{n-1}\sum_{k=1}^K\frac{1}{2}(p_{m,k}-v_k)^2\right]\\
\le & H(\bm v)-Ent_n+\sum_{k=1}^K \log v_k\left(\ep\left[\frac{N_{n,k}}{n}\right]-v_k\right)\to 0,
\end{align*}
which implies
$$\frac{1}{n} \sum_{m=0}^{n-1}\|\bm p_m-\bm v\|\le \left(\frac{1}{n} \sum_{m=0}^{n-1}\|\bm p_m-\bm v\|^2\right)^{1/2}\overset{P}\to 0. $$
The proof of \eqref{eqthEntropy1.3} is completed. Finally, note $|\max_kp_{m,k}-\max_k\rho_k(\bm\Theta)|\le \|\bm p_m-\bm\rho(\bm\Theta)\|$. (d) of Theorem \ref{thEntropy} follows from (c) of Theorem \ref{thEntropy} immediately.  $\Box$

\bigskip
{\bf Proof of Theorem \ref{thSBHuZhHe}}. Write $\bm v=\bm\rho(\bm\Theta)$ as before.  By noticing that $\{J_m-\ep[J_m|\mathscr{F}_{m-1}]\}$ is a sequence of bounded martingale differences,  it is sufficient to show that
\begin{equation}\label{eq:proofHZH2.1} \frac{1}{n}\sum_{m=1}^n\ep[J_m|\mathscr{F}_{m-1}]\to SB\;\; a.s.
\end{equation}
and
\begin{equation}\label{eq:proofHZH2.2} \frac{1}{n}\sum_{m=1}^n H(\bm p_{m-1})\to I(\alpha v_1)v_2+I(\alpha v_2)v_1\;\; a.s.
\end{equation}
Let $K_{n,>}=\sharp\{m\le n: N_{m,1}>m \widehat{\rho}_{m,1}\}$, and define $K_{n,=}$, $K_{n,<}$ similarly.
Then $\pr(N_{m,1}=m\widehat{\rho}_{m,1})=0$ by the fact that $\rho_1(\cdot)$ has no mass at rational points. Hence $I\{N_{m,1}=m\widehat{\rho}_{m,1}\} =0$ a.s. It follows that
$$\frac{K_{n,=}}{n}=\frac{1}{n}\sum_{m=0}^{n-1} I\{N_{m,1}=m\widehat{\rho}_{m,1}\}= 0\;\; a.s. $$
  Notice
 \begin{align*}\ep[J_{m+1}|\mathscr{F}_m]=&\big[(\alpha \widehat{\rho}_{m,1})\vee(1-\alpha \widehat{\rho}_{m,1})\big]I\{N_{m,1}>m \widehat{\rho}_{m,1}\}
 +\big[ \widehat{\rho}_{m,1}\vee \widehat{\rho}_{m,2}\big]I\{N_{m,1}=m \widehat{\rho}_{m,1}\}\\
&+\big[(\alpha \widehat{\rho}_{m,2})\vee (1-\alpha \widehat{\rho}_{m,2})\big]I\{N_{m,1}<m \widehat{\rho}_{m,1}\}\\
=&\big[(\alpha v_1)\vee(1-\alpha v_1)\big]I\{N_{m,1}>m \widehat{\rho}_{m,1}\}
 +\big[ v_1\vee v_2\big]I\{N_{m,1}=m \widehat{\rho}_{m,1}\}\\
&+\big[(\alpha v_2)\vee (1-\alpha v_2)\big]I\{N_{m,1}<m \widehat{\rho}_{m,1}\}+o(1)\;\; a.s.
\end{align*}
It follows that
\begin{align*}\frac{1}{n}\sum_{m=1}^n\ep[J_m|\mathscr{F}_{m-1}]=\big[(\alpha v_1)\vee(1-\alpha v_1)\big]\frac{K_{n,>}}{n}
 +\big[(\alpha v_2)\vee (1-\alpha v_2)\big]\big(1-\frac{K_{n,<}}{n}\big)+o(1)\;\; a.s.
\end{align*}
Similarly,
 \begin{align*}\ep[X_{m+1,1}|\mathscr{F}_m]=& (\alpha v_1) I\{N_{m,1}>m \widehat{\rho}_{m,1}\}
 +  v_1 I\{N_{m,1}=m \widehat{\rho}_{m,1}\}\\
&+ (1-\alpha v_2) I\{N_{m,1}<m \widehat{\rho}_{m,1}\}+o(1)\;\; a.s.
\end{align*}
and
$$ \frac{N_{n,1}}{n}= \frac{1}{n}\sum_{m=1}^n\ep[X_{m,1}|\mathscr{F}_{m-1}]+o(1)
=  \alpha v_1   \frac{K_{n,>}}{n}+ (1-\alpha v_2)\big(1-\frac{K_{n,>}}{n}\big) +o(1)\;\; a.s.
 $$
On the other hand, $N_{n,1}/n\to v_1$ a.s. by Theorem 2.1 of \cite{HZH09}.
Hence $K_{n,>}/n \to v_2$ a.s. It follows that
\begin{align*}
&\lim_{n\to \infty}\frac{1}{n}\sum_{m=1}^n J_m = \lim_{n\to \infty}\frac{1}{n}\sum_{m=1}^n\ep[J_m|\mathscr{F}_{m-1}]\\
=&\big[(\alpha v_1)\vee(1-\alpha v_1)\big]v_2
 +\big[(\alpha v_2)\vee (1-\alpha v_2)\big]v_1=SB\;\; a.s.
\end{align*}
The proof of \eqref{eq:proofHZH2.1} is completed.

Similarly,
\begin{align*}
\frac{1}{n} \sum_{m=1}^n &  H(\bm p_{m-1})=  \frac{1}{n}\sum_{m=0}^{n-1} \left[ I(\alpha\widehat{\rho}_{m,1})I\{N_{m,1}/m>\widehat{\rho}_{m,1}\}
+ I(\widehat{\rho}_{m,1})I\{N_{m,1}/m=\widehat{\rho}_{m,1}\}\right.\\
& \qquad \qquad\qquad \left. + I(\alpha(1-\widehat{\rho}_{m,1}))I\{N_{m,1}/m<\widehat{\rho}_{m,1}\}\right] \\
=& I(\alpha v_1)\frac{K_{n,>}}{n}+I(v_1)\frac{K_{n,=}}{n}+I(\alpha v_2)\frac{K_{n,<}}{n}+o(1)\; a.s.\\
=&  I(\alpha v_1)v_2+0+I(v_1)\cdot 0+I(\alpha v_2)v_1+o(1)= I(\alpha v_1)v_2 +I(\alpha v_2)v_1+o(1)\; a.s.
\end{align*}
The proof is completed. $\Box$

\section{Doubly asymptotically best randomized-adaptive design}\label{sectbest}
  If $p_{m,k}\to \rho_k(\bm\Theta)$,  then $\pr(J_m=1)\to \max_k\rho_k(\bm\Theta)$ and $Ent_n\to H\left(\bm\rho(\bm\Theta)\right)$, which implies that the asymptotic selection bias  $SB=\max_k\rho_k(\bm\Theta)$  attains the lower bound and the asymptotic entropy attains the upper bound.  Hence, if the allocation function  $\bm g\big(m,\frac{\bm
N_m}{m},\widehat{\bm \rho}_m\big)\equiv\bm g\big(\frac{\bm
N_m}{m},\widehat{\bm \rho}_m\big)$ in (\ref{allocation}) is a {\em  continuous} function of the current sample allocation proportion $\frac{\bm
N_m}{m}$ and the current estimator $\widehat{\bm \rho}_m$ of the target allocation proportion, then the asymptotic selection bias and asymptotic entropy are optimal.

 The DBCD of \cite{HZ04}  attains the optimal selection bias and optimal entropy because its allocation function is continuous. However, it is not efficient. Most generalized Friedman's urn models such as the randomized play-the-winner rule of \cite{WD78}  and its generalization  \cite{Wei79} also attain their optimal selection bias and optimal entropy because the urn proportions (i.e., the allocation probabilities) will converge to the same limit of the sample allocation proportions.  But they are far away from efficient designs as discussed in \cite{HRZ06}  and \cite{HZH09}.

A response-adaptive design is called a {\em doubly asymptotically best randomized-adaptive design} (DaBrade)  if it both is  asymptotically efficient in sense that the asymptotic variance of $\bm N_n/\sqrt{n}$ attains the lower bound \eqref{eq:Lowerbound} and is asymptotically most random in the sense that the asymptotic
entropy attains the upper bound $H(\bm \rho(\bm\Theta))$.
   For defining a design that is efficient and also has the optimal entropy, as discussed, we shall choose an allocation probability function that is continuous and satisfies Condition \ref{ConC}. So, the efficient DBCDs (Procedures \ref{AP:B} and \ref{AP:C})   are such  designs.

\begin{theorem}\label{thEprade}    Under Conditions \ref{ConA} and \ref{ConB},  EDBCD (Procedure \ref{AP:B}) and EDBCD2 (Procedure \ref{AP:C}) 
are efficient adaptive designs in the sense that (\ref{eqNormality}) hold, and also most random in the sense that the asymptotic entropy $\lim_nEnt_n$ and the asymptotic selection bias $\lim_n SB_n$ attain the optimal values $H\left(\bm \rho(\bm\Theta)\right)$ and $\max_k\rho_k(\bm \Theta)$, respectively. Further
\begin{equation}\label{eqEprade1.3}  \pr( X_{m,k}|\mathscr{F}_{m-1})\to \rho_k(\bm\Theta)\;\; a.s.,
\end{equation}
\begin{equation}\label{eqEprade1.4} \lim_{n\to \infty}\frac{1}{n}\sum_{m=1}^n J_m=SB=\max_k\rho_k(\bm\Theta) \;\; a.s.,
\end{equation}
\begin{equation}\label{eqEprade1.5} \lim_{n\to \infty}\frac{1}{n}\sum_{m=1}^n H\left(\bm p_{m-1}\right)=H\left(\bm \rho(\bm\Theta)\right) \;\; a.s.
\end{equation}
\end{theorem}

 An adaptive design with (\ref{eqEprade1.3}) has both nice interpretative properties in practice and in theory.    In practice, each patient asymptotically has the same chance to receive a desirable treatment. In theory,  the sequence $\{\bm X_m\}$ of the results of assignments is asymptotically equivalent to a sequence $\{\widetilde{\bm X}_m\}$ of i.i.d. random vectors with $\pr(\widetilde{ X}_{m,k}=1)=\rho_k(\bm\Theta)$, which is a purely random sequence. No adaptive procedure can be better in terms of randomization than the purely random assignment even when the value of the parameter $\bm\Theta$ is known.

The weight function in \eqref{eqEprade1.1} can be replaced by a more general function.
\begin{theorem}\label{thEprade2}    Let $\psi(x)\ge 0, x\ge 0$ be a non-decreasing function with the left derivative $\dot{\psi}_{-}(1)<\infty$ and the right derivative   $\dot{\psi}_{+}(1)=+\infty$ (resp. $\dot{\psi}_{-}(1)=\infty$ and    $\dot{\psi}_{+}(1)<\infty$), and $\psi(1)=1$. Suppose that $\psi(x)$ is continuous at $x=1$.
Then under Conditions \ref{ConA} and \ref{ConB}, an adaptive randomization procedure \ref{AP:A} with allocation function $\bm g=(g_1,\ldots,g_K)$ being defined as
\begin{equation}\label{eq:Eprade2.1}  g_{k}\big(m,\bm x ,\bm \rho\big)=g_{k}\big(\bm x ,\bm \rho\big)=\frac{\rho_k\psi\left(\frac{\rho_k}{x_k}\right)}{\sum_{j=1}^K\rho_j\psi\left(\frac{\rho_j}{x_j}\right)}, \; k=1,\ldots,K
\end{equation}
  is   efficient  in sense that \eqref{eqNormality} hold, and most random in sense that the asymptotic entropy $\lim_n Ent_n$  and the asymptotic selection bias $\lim_n SB_n$ attain the optimal values $H\big(\bm  \rho(\bm\Theta)\big)$ and $\max_k\rho_k(\bm \Theta)$, respectively. Further, \eqref{eqEprade1.3}-\eqref{eqEprade1.5} hold.
\end{theorem}

 Besides (\ref{eqEprade1.1}), the wight function $\psi(x)$ can be chosen as
$$  \psi(x)=\begin{cases}x^{\gamma}, &x\ge 1,\\
1-\sqrt{1-x^{2\gamma}}, & 0\le x\le 1
\end{cases}
\; \text{ or }\;   \psi(x)=\begin{cases} 1+\sqrt{x^{2\gamma}-1}, &  x\ge 1,\\
1, &0\le x\le 1
\end{cases} \;\; \text{ etc.}
$$

{\bf Proof of Theorem \ref{thEprade2}. } Suppose $\dot{\psi}_{-}(1)<\infty$ and $\dot{\psi}_{+}(1)=+\infty$.
For (\ref{eqNormality}), it sufficient to verify the Condition \ref{ConC}.
Note that $\psi(x)$ is non-decreasing. So, $g_i(m,\bm x,\bm y)/y_i\le g_j(m,\bm x,\bm y)/y_j$ whenever $x_i/y_i\ge x_j/y_j$.
By Theorem \ref{Th1},
\begin{equation}\label{eqproofthEprade2.1}\lim_{n\to \infty}\frac{\bm N_n}{n}=\lim_{n\to \infty}\widehat{\bm\rho}_n=\bm\rho(\bm\Theta)=\bm v\;\; a.s.
\end{equation}
  With the same arguments as proving Corollary \ref{Cor3.1}, it follows that, for any given $\lambda>0$,
\begin{align*}
g_k- \rho_k \le  -\lambda (x_k- \rho_k), \;\; x_k> \rho_k,
\end{align*}
whenever $\|\bm x-\bm v\|+\|\bm \rho-\bm v\|$ is small enough. By (\ref{eqproofthEprade2.1}),  Conditions (C1) and (C2) are verified.  For (\ref{eqEprade1.3}) and (\ref{eqEprade1.4}), note that $\psi(x)$ is continuous at $x=1$. So, $\psi\big(\widehat{\rho}_{m-1,k}/(N_{m-1,k}/(m-1))\big)\to 1$ a.s. $k=1,\ldots,K$. It follows that
$$\pr(X_{m,k}=1|\mathscr{F}_{m-1})=p_{m-1,k}=g_k\left(\frac{\bm N_{m-1}}{m-1},\widehat{\bm \rho}_{m-1}\right)\to \rho_k(\bm\Theta)\;\; a.s. $$
and
\begin{align*}
 \frac{1}{n}\sum_{m=1}^n J_m= & \frac{1}{n}\sum_{m=1}^n\big(J_m-\ep[J_m|\mathscr{F}_{m-1}]\big)
+\lim_{n\to \infty}\frac{1}{n}\sum_{m=1}^n \max_k p_{m-1,k}\to
 \max_k\rho_k(\bm\Theta)\;\; a.s.,\\
&  \frac{1}{n}\sum_{m=0}^{n-1} H(\bm p_m)\to H(\bm \rho(\bm \Theta)) \;\; a.s.
\end{align*}
Thus, (\ref{eqEprade1.3})-(\ref{eqEprade1.5}) hold and $\lim_n SB_n=\max_k\rho_k(\bm\Theta)$, $\lim_n Ent_n=  H(\bm \rho(\bm \Theta))$. $\Box$.

\section{General asymptotic properties.}\label{sectAsmptotics}
\setcounter{equation}{0}
Besides the asymptotic normality in Theorem \ref{thNormality}, we will derive asymptotic properties including the strong consistency, the convergence rate of the consistency, the functional central limit theorems, and the Gaussian approximation. The functional central limit theorem and the Gaussian approximation are fruitful tools for deriving the asymptotic normality of randomly stopped processes (c.f., \cite{B68},  Page 146, and \cite{CR81}, Pages 258-259) and for deriving the asymptotic distribution of a  function of the processes. If sequential monitoring of an adaptive design is considered, the functional central limit theorem is also needed for deriving the distribution of the sequential test (c.f., \cite{ZhuHu10}).

\begin{theorem} \label{Th1}
{\rm (Strong Consistency)} 
 Suppose that  in the Bahadur-type representation \eqref{eq2.1}, $\ep \|\bm  \xi_{1,k}\|<\infty$ and $\bm\theta_k= \ep\bm\xi_{1,k}$,
 $k=1,\ldots,K$, and that the target proportion function
$\bm y\to \bm  \rho(\bm y): \Cal R^{d\times K}\to
(0,1)^K$  is a continuous function. Let $0\le \lambda<1$, $\Pi=\{(\pi_1,\ldots, \pi_K)\in (0,1)^K: \sum_i\pi_i=1\}$.   Assume that the allocation function $\bm g(n,\bm x, \bm y)$, $\bm x,\bm y\in \Pi$ satisfies   one of the following conditions:
\begin{description}
  \item[\rm (i)] For any $\epsilon>0$, there exists $n_0$ such that $g_k(n,\bm x,\bm y)-y_k\le  \lambda(x_k-y_k)+\epsilon$ whenever $x_k>y_k+\epsilon$, and $n\ge n_0$;
  \item[\rm (ii)] For any $\epsilon>0$, there exists $n_0$ such that
  \begin{equation}\label{eq:Th1.1} \frac{g_i(m,\bm x,\bm y)}{y_i}-\frac{ g_j(m,\bm x,\bm y)}{y_j}\le \lambda \left(\frac{x_i}{y_i}-\frac{x_j}{y_j}\right)+\epsilon
  \end{equation} 
  whenever $x_i/y_i> x_j/y_j+\epsilon$, and $n\ge n_0$;
  \item[\rm (iii)] For any $\epsilon>0$, there exists $n_0$ such that \eqref{eq:Th1.1} holds whenever $x_i-y_i> x_j-y_j+\epsilon$, and $n\ge n_0$;
  \item[\rm (iv)] There exists positive and continuous functions $\alpha_k(y)$ ($0<y<1$), $k=1,\ldots, K$, such that  any $\epsilon>0$, there exists $n_0$ for which \eqref{eq:Th1.1} holds whenever $\frac{x_i-y_i}{\alpha_i(y_i)}> \frac{x_j-y_j}{\alpha_j(y_j)}+\epsilon$  and $n\ge n_0$.
\end{description}
Then $
\bm N_n/n\to \bm v$ a.s., $\widehat{\bm\Theta}_n\to \bm\Theta$ a.s. and $\widehat{\bm \rho}_n\to \bm v$ a.s.
\end{theorem}

If Condition (C1) is satisfied, then (i) is satisfied.

The next theorem is on the rate of the strong consistency.

\begin{theorem} \label{Th2} Suppose that Conditions \ref{ConA}, \ref{ConB} and
(C1) with $\lambda_0<1/2$ are satisfied. Then
\begin{equation}\label{eqTh2.1}\widehat{\bm \rho}_n- \bm v=O\Big(\sqrt{\frac{\log\log
n}n}\Big) \quad  \text{ and  }\quad \bm N_n-n\bm
v=O\big(\sqrt{n\log\log n}\big) \; a.s.,
\end{equation}
\begin{equation}\label{eqTh2.2} \max_{m\le n}\|\bm N_m-m\bm
v\|=O\big(\sqrt{n}\big) \; \text{in probability}.
\end{equation}
Further, there exists an $(K\times d)$-dimensional standard Brownian motion $\bm
B(t)=(\bm B_1(t),\ldots, \bm B_K(t))$ such that
\begin{equation}\label{eqTh2.3ad}
n\big(\widehat{\bm\Theta}_n-\bm\Theta)=\bm B_t\bm V^{1/2}+o(n^{1/2})\;\; a.s.
\end{equation}
and
\begin{equation}
\label{eqTh2.3}  \widehat{\bm \rho}_n-\bm v =\frac{\bm W(n)}{n} +o\Big(\frac{\|\bm W(n)\|}{n}\Big)
  +o(n^{-1/2}) \;\; a.s.,
\end{equation}
where $$\bm W(t)=\bm B(t)\bm V^{1/2}\frac{\partial
\bm \rho}{\partial\bm y}\big|_{\bm \Theta}$$
 is a $K$-dimensional Browian motion with the variance-covariance matrix $\bm\Sigma$ given in (\ref{eqLB}).
\end{theorem}

The next theorem tells us that the process of numbers of patients can be approximated in probability by a Brownian motion.
\begin{theorem} \label{Th3} Suppose that Conditions \ref{ConA}, \ref{ConB} and \ref{ConC}
are satisfied. Then (\ref{eqTh2.1}), (\ref{eqTh2.3ad}), (\ref{eqTh2.3}) hold, and
\begin{equation}\label{eqTh3.1}
\max_{m\le n}|\bm N_m - m \widehat{\bm \rho}_m|=o(\sqrt{n})\;\;
\text{ in probability}.
\end{equation}
And accordingly,
\begin{equation}\label{eqApproximation}\max_{m\le n}|\bm N_m - m\bm v- \bm W(m)
|=o(\sqrt{n})\;\; \text{ in probability}.
\end{equation}
In particular,
\begin{equation}\label{eqWeakconvergence}
 n^{-1/2}(\bm N_{[nt]}-nt\bm v)\overset{\mathcal D}\to \bm W(t) \;\text{ and }\;
t n^{1/2}(\widehat{\bm \Theta}_{[nt]}-\bm \Theta)\overset{\mathcal D}\to \bm B(t)\bm V^{1/2}
\end{equation}
in the space $\mathscr{D}_{[0,1]}$ with the Skorohod topology.
\end{theorem}

The proofs of these three theorems will be stated in  supplementary materials, since they are long and quite technical.  Theorem \ref{thNormality} is a direct corollary of (\ref{eqWeakconvergence}).  By a standard argument given in \cite{CR81} (Pages 258-259), (\ref{eqApproximation}) implies that (\ref{eqNormality}) and (\ref{eqNormalityforTheta}) remain true when the sample size $n$ is replaced by a random variable.

\begin{corollary}\label{thStoptime} Suppose that $\tau_n$ is a positive random variable that takes integer values and $\tau_n/f(n)\to \eta$ in probability, where $f(n)\nearrow \infty$ is a non-random sequence and $\pr(0<\eta<\infty)=1$. Then under Condition \ref{ConA}-\ref{ConC},
$$
 \tau_n^{1/2}(\frac{\bm N_{\tau_n}}{\tau_n}-\bm v)\overset{\mathcal D}\to N(\bm 0,\bm \Sigma) \;\text{ and }\;
 \tau_n^{1/2}(\widehat{\bm \Theta}_{\tau_n}-\bm \Theta)\overset{\mathcal D}\to N(\bm 0, \bm V).
$$
\end{corollary}

  Corollary \ref{thStoptime} is useful in the situation that the adaptive design is stopped at a random time. In all of the above results, the Condition \ref{ConC} can be replaced by Condition C$^{\prime}$ by symmetry.

\section{Conclusions remarks}
We have proposed a general framework of randomized-adaptive designs for general multiple-treatment clinical trials that are asymptotic efficient and can target any desired allocation. The general framework is flexible for us to define new efficient designs with desirable properties. With the framework, we find an efficient  randomized response-adaptive design that seems to give us everything we want in a response-adaptive randomization procedure: it is most random with minimum asymptotic selection bias,  it attains the lower bound of the allocation variabilities and it can target any allocation proportion.
The theoretical results in this paper are basing on the assumption that the sample size is large.  In a fixed sample size, it may still present dilemmas regarding the tradeoffs among randomization, variability, and optimality. A good estimator of the distribution parameter will improve the convergence rate of the asymptotics.  We suggest using the Bayesian approach to estimate the parameters in using the response-adaptive procedures for a small sample size.

 In practice, responses may not be available immediately after the patients have been treated. However, there are no logistical difficulties in incorporating delayed responses into the framework of this paper. One can update the estimators when responses become available.  Under widely satisfied conditions for the delay mechanism, \cite{HZCC08}  showed that the estimator $\widetilde{\bm\Theta}_m$  with delayed responses has the same asymptotic properties as the estimator $\widehat{\bm\Theta}_m$ with the assumption of immediate responses. So, substituting $\widehat{\bm\Theta}_m$ by  $\widetilde{\bm\Theta}_m$ will not affect the asymptotic properties of the sample allocation proportions and the proofs. \cite{ZR07} developed response-adaptive randomization procedures for survival trials in which the outcomes of treatments are both delayed and censored. The procedures proposed in  this paper can be generalized to randomization procedures for survival trials by following the discussion of \cite{ZR07}.

 In this paper, we prove that there is a lower bound among the asymptotic selection biases (resp. an upper bound among the asymptotic entropies) of all adaptive designs with the same target allocation proportion. However, it is not an easy work to verify whether a design attains the lower bound (resp. upper bound) or not when the allocation probabilities do not converge. The drop-the-loser rule that was proposed by \cite{Ivanva03,Ivanva06} is shown to be an efficient randomized-adaptive design that attains the lower bound of variability for a special kind of allocation proportions (see the discussion of \cite{HR06}). However, previously reported simulations have shown that it becomes more deterministic for small values of success rates of treatments (see \cite{HR03}).   How to find a theoretical verification is an open problem. The same problem is also open for the generalized drop-the-loser rule of \cite{ZCCH06} and the IMU model of \cite{ZHCC11}.

\bigskip



\bibliographystyle{chicago}      
\bibliography{GDBCD}   


\newpage

\appendix
\setcounter{equation}{0}
\renewcommand{\theequation}{A.\arabic{equation}}
\setcounter{table}{0}
\renewcommand{\thetable}{C.\arabic{table}}
\setcounter{table}{0}
\renewcommand{\thefigure}{C.\arabic{table}}

\begin{center}
 { \LARGE\bf  Multiple-treatment Doubly Asymptotically Best  Randomized-adaptive Designs}

 {\Large \bf Supplementary Materials}
\end{center}

\begin{center} {\sc
\href{https://mypage.zjgsu.edu.cn/tjysjkx/zlx2_en/main.htm}{\blue{Li-Xin Zhang}}
}\\
{\sl \small School  of Statistics and Data Science, Zhejiang Gongshang University, Hangzhou 310018} \\
(Email:stazlx@mail.zjgsu.edu.cn)    \\
\end{center}

\renewcommand{\thesection}{\Alph{section}}
 In the supplementary materials, we give the proofs of the asymptotic properties and simulation studies.

\section{Proofs of the asymptotic properties}
\setcounter{equation}{0}

Recall that the $(m+1)$-th patient is assigned
to the treatment $k$ with probability
\begin{equation} \label{allocationappend}p_{m,k}=g_k\big(m,\frac{\bm
N_m}{m},\widehat{\bm \rho}_m\big),\;\; k=1,\ldots, K,
\end{equation}
 where $\widehat{\bm\Theta}_m=(\widehat{\theta}_{m,1},\ldots,\widehat{\bm\theta}_{m,K})$ is the estimator of $\bm\Theta=(\bm\theta_1,\ldots,\bm\theta_K)$, and $\widehat{\bm \rho}_m=\bm\rho(\widehat{\bm\Theta}_m)$.
  We only consider the Condition C. When   Condition C$^{\prime}$ is satisfied, the arguments  in the following proofs are the same with  $N_{m,k}-m\widehat{\rho}_{m,k}$ being replaced by $m\widehat{\rho}_{m,k}-N_{m,k}$.
We repeat the main conditions here.

\setcounter{condition}{0}
\begin{condition} \label{ConAappend} The proportion function
$\bm y=(\bm y_1,\ldots,\bm y_K)=(y_{11}, \ldots, y_{1d}, \ldots,
y_{K1}, \ldots,y_{Kd})\to \bm  \rho(\bm y): \Cal R^{d\times K} \to
(0,1)^K$  is a continuous function and is   differentiable at $\bm\Theta$ with $\bm\rho(\bm\Theta)=\bm v$.
\end{condition}
\begin{condition}\label{ConBappend}
The parameter estimator $\widehat{\bm \theta}_{m,k}$ of $\bm\theta_k$ based on an $m$-patient study has the Bahadur-type representation
\begin{equation}\label{eq2.1append}
 \widehat{\bm \theta}_{m,k}
=N_{m,k}^{-1}\sum_{j=1}^m X_{j,k}\bm \xi_{j,k}+o(N_{m,k}^{-1/2}) \; a.s. \; \text{ as } m\to \infty
\end{equation}
with $\bm\theta_k= \ep\bm\xi_{1,k}$ and
$\ep \|\bm  \xi_{1,k}\|^{2+\epsilon}<\infty$ for some $\epsilon>0$, $k=1,\ldots,K$.
\end{condition}

\begin{condition}\label{ConCappend} Suppose that the
allocation probabilities of the $(m+1)$-th patient $\bm g(m,\bm
N_m/m, \widehat{\bm \rho}_m)$ can be written in the following form:
\begin{align} \label{eqConditionforEffappend}p_{m,k}= g_k(m,\frac{\bm
N_m}{m},\widehat{\bm \rho}_m) & \le \widehat{\rho}_{m,k}+\lambda_{m,k}\big[\frac{N_{m,k}}{m}-
\widehat{\rho}_{m,k}\big] +\gamma_{m,k}\|\frac{\bm N_m}{m}-\widehat{\bm \rho}_m\|\\
& \text{ if } \;
N_{m,k} -m \widehat{\rho}_{m,k}\ge L_0, \; k=1,2,\ldots, K\nonumber
\end{align}
for some  $L_0\ge 0$, with the conditions that
$$ \limsup_{m\to \infty} \gamma_{m,k}\le 0 \;a.s.\; \text{ and } \limsup_{m\to \infty}\lambda_{m,k}\le \lambda_0<1\;\; a.s.,
\quad k=1,2,\ldots, K, \eqno{(C1)} $$
$$  \lambda_{m,k}\to -\infty \; a.s.  \;\; \text{ on the event }
\;\; \Big\{\frac{\bm N_m}{m}\to \bm v,\; \widehat{\bm \rho}_m\to \bm v\Big\},
\quad k=1,2,\ldots, K. \eqno{(C2)}
$$
\end{condition}

The main asymptotic properties are stated as follows which are Theorems 6.1-6.3.

\begin{theorem} \label{Th1append}
 Suppose that  in the Bahadur-type representation \eqref{eq2.1append}, $\ep \|\bm  \xi_{1,k}\|<\infty$ and $\bm\theta_k= \ep\bm\xi_{1,k}$,
 $k=1,\ldots,K$, and that the target proportion function
$\bm y\to \bm  \rho(\bm y): \Cal R^{d\times K}\to
(0,1)^K$  is a continuous function. Let $0\le \lambda<1$, $\Pi=\{(\pi_1,\ldots, \pi_K)\in (0,1)^K: \sum_i\pi_i=1\}$.   Assume that the allocation function $\bm g(n,\bm x, \bm y)$, $\bm x,\bm y\in \Pi$ satisfies   one of the following conditions:
\begin{description}
  \item[\rm (i)] For any $\epsilon>0$, there exists $n_0$ such that $g_k(n,\bm x,\bm y)-y_k\le  \lambda(x_k-y_k)+\epsilon$ whenever $x_k>y_k+\epsilon$, and $n\ge n_0$;
  \item[\rm (ii)] For any $\epsilon>0$, there exists $n_0$ such that
  \begin{equation}\label{eq:Th1append.1} \frac{g_i(m,\bm x,\bm y)}{y_i}-\frac{ g_j(m,\bm x,\bm y)}{y_j}\le \lambda \left(\frac{x_i}{y_i}-\frac{x_j}{y_j}\right)+\epsilon
  \end{equation} 
  whenever $x_i/y_i> x_j/y_j+\epsilon$, and $n\ge n_0$;
  \item[\rm (iii)] For any $\epsilon>0$, there exists $n_0$ such that \eqref{eq:Th1append.1} holds whenever $x_i-y_i> x_j-y_j+\epsilon$, and $n\ge n_0$;
  \item[\rm (iv)] There exists positive and continuous functions $\alpha_k(y)$ ($0<y<1$), $k=1,\ldots, K$, such that  any $\epsilon>0$, there exists $n_0$ for which \eqref{eq:Th1append.1} holds whenever $\frac{x_i-y_i}{\alpha_i(y_i)}> \frac{x_j-y_j}{\alpha_j(y_j)}+\epsilon$  and $n\ge n_0$.
\end{description}
Then $
\bm N_n/n\to \bm v$ a.s., $\widehat{\bm\Theta}_n\to \bm\Theta$ a.s. and $\widehat{\bm \rho}_n\to \bm v$ a.s.
\end{theorem}

If Condition (C1) is satisfied, then (i) is satisfied. 

\begin{theorem} \label{Th2append} Suppose that Conditions \ref{ConAappend}, \ref{ConBappend} and
(C1) with $\lambda_0<1/2$ are satisfied. Then
\begin{equation}\label{eqThappend2.1}\widehat{\bm \rho}_n- \bm v=O\Big(\sqrt{\frac{\log\log
n}n}\Big) \quad  \text{ and  }\quad \bm N_n-n\bm
v=O\big(\sqrt{n\log\log n}\big) \; a.s.,
\end{equation}
\begin{equation}\label{eqThappend2.2} \max_{m\le n}\|\bm N_m-m\bm
v\|=O\big(\sqrt{n}\big) \; \text{in probability}.
\end{equation}
Further, there exists an $(K\times d)$-dimensional standard Brownian motion $\bm
B(t)=(\bm B_1(t),\ldots, \bm B_K(t))$ such that
\begin{equation}\label{eqThappend2.3}
 \widehat{\bm\Theta}_n-\bm\Theta =\frac{\bm B_n\bm V^{1/2}}{n}+o(n^{-1/2})\;\; a.s.
\end{equation}
and
\begin{equation}
\label{eqThappend2.4}  \widehat{\bm \rho}_n-\bm v =\frac{\bm W(n)}{n} +o\Big(\frac{\|\bm W(n)\|}{n}\Big)
  +o(n^{-1/2}) \;\; a.s.,
\end{equation}
where $$\bm W(t)=\bm B(t)\bm V^{1/2}\frac{\partial
\bm \rho}{\partial\bm y}\big|_{\bm \Theta}$$
 is a $K$-dimensional Browian motion with the variance-covariance matrix $\bm\Sigma$,
 \begin{equation*}
 \bm
V_k=\Var(\bm \xi_{1,k})=\left(\Cov\big[\xi_{1,ki},\xi_{1,kj}\big];
i,j=1,\ldots,d\right), \quad k=1,\ldots, K,
\end{equation*}
\begin{equation}\label{eqVarinceThetaappend}
\bm V=diag(\frac 1{v_1}\bm V_1,\ldots,\frac 1{v_K}\bm V_K),
\end{equation}
\begin{equation}
\label{eqLBappend}\bm \Sigma= \big(\frac{\partial\bm
\rho}{\partial\bm y}\big|_{\bm \Theta}\big)^{\prime} \bm V
\frac{\partial\bm \rho}{\partial\bm y}\big|_{\bm \Theta}
=\sum_{k=1}^K \frac 1{v_k} \big(\frac{\partial\bm  \rho}{\partial
\bm y_k}\big|_{\bm \Theta}\big)^{\prime}
   \bm V_k \frac{\partial\bm
\rho}{\partial \bm y_k}\big|_{\bm \Theta}.
\end{equation}
\end{theorem}

\begin{theorem} \label{Th3append} Suppose that Conditions \ref{ConAappend}, \ref{ConBappend} and \ref{ConCappend}
are satisfied. Then (\ref{eqThappend2.1}), (\ref{eqThappend2.3}), (\ref{eqThappend2.4}) hold, and
\begin{equation}\label{eqThappend3.1}
\max_{m\le n}|\bm N_m - m \widehat{\bm \rho}_m|=o(\sqrt{n})\;\;
\text{ in probability}.
\end{equation}
And accordingly,
\begin{equation}\label{eqappend3.2}\max_{m\le n}|\bm N_m - m\bm v- \bm W(m)
|=o(\sqrt{n})\;\; \text{ in probability}.
\end{equation}
\end{theorem}

  Let
$\mathscr{F}_m=\sigma(\bm X_1,\ldots, \bm X_m,
\bm \xi_1,\ldots,\bm \xi_m)$ be the sigma field generated by
previous $m$ stages. Then under $\mathscr{F}_{m-1}$, $\bm X_m$ and
$\xi_m$ are independent, and
$$\ep[\bm X_m|\mathscr{F}_{m-1}]=\bm g(m,\frac{\bm N_{m-1}}{m-1},\widehat{\bm \rho}_{m-1}). $$

To prove our theorems, we first need some lemmas. The first one is
a refinement of Lemma A.3 of \cite{HZ04}.

\begin{lemma}\label{lem1}
 Let $\lambda_0<1 $ and $K_0> 0$ be two real numbers, and let $\{q_n\}$ and $\{p_n\}$
 be two sequences of real numbers for which
\begin{equation} \label{eqlem1.1}
q_n\le q_{n-1}+K_0 +\Delta p_n, \quad n\ge n_0, \end{equation}
\begin{equation} \label{eqlem1.2}
q_n\le (1+\frac{\lambda_0}{n-1}) q_{n-1} +\Delta p_n, \;\; \text{
if } \;\; q_{n-1}> 0, \quad n\ge n_0, \end{equation}
 where $\Delta p_1=p_1$ and $\Delta
p_n=p_n-p_{n-1}$, $n\ge 2$. Further, let $b_{n,n}=1$,
$$ b_{n,m}= \prod_{i=m}^{n-1}(1+\frac{\lambda_0}{i}),\quad m=1,2,\cdots, n-1, \quad
n=1,2,\cdots,
$$ and
$d_{n,m}=b_{n,m}\Delta p_m+b_{n,m+1}\Delta
p_{m+1}+\ldots+b_{n,n}\Delta p_n$, $1\le m\le n$.
  Then there exists a constant $C>0$ which depends only on $n_0$ and $\lambda_0$ such that
\begin{align}\label{eqlem1.3}
q_n^+ &\le \max_{|\lambda_0|\le m\le n}|d_{n,m}|+Cn^{\lambda_0}
\max_{m\le n_0 \vee (|\lambda_0|+1)}q_m^+
+C\max_{1\le m\le n}\Big(\frac{n}{m}\Big)^{\lambda_0} K_0 \nonumber\\
&\le C\Big\{
\sum_{m=1}^n\frac{|p_m|}{m}\Big(\frac{n}{m}\Big)^{\lambda_0}+
\max_{1\le m\le
n}\Big(\frac{n}{m}\Big)^{\lambda_0}|p_m|\Big\}+O(n^{\lambda_0}+1).
\end{align}
\end{lemma}

{\bf Proof } Without loss of generality, we assume that $n_0\ge
|\lambda_0|+1$. Then $1+\frac{\lambda_0}{n-1}\ge 0$ if $n\ge n_0$.
The results is obvious if $n\le n_0$. For $n>n_0$, define $l=l(n)$
by $l=\max\{i: n_0<i\le n, q_i\le 0\}$, where $\max\emptyset
=n_0$. Then $q_m> 0$ for $m=l+1,\ldots, n$. According to
(\ref{eqlem1.2}) and (\ref{eqlem1.1}),
\begin{align*}
q_n&\le b_{n,n-1} q_{n-1}+d_{n,n}
 \le b_{n,n-1}\big[(1+\frac{\lambda_0}{n-2})q_{n-2}+\Delta
p_{n-1}\big]
+d_{n,n} \\
&=b_{n,n-2}q_{n-2}+d_{n,n-1}\\
 &\le \ldots \le
b_{n,l+1}q_{l+1}+d_{n,l+2} \le b_{n,l+1}(q_l+K_0+\Delta
p_{l+1})+d_{n,l+2}\\
&= b_{n,l+1}q_l+b_{n,l+1}K_0+d_{n,l+1}.
\end{align*}
Notice that $q_l\le 0$ if $l>n_0$ and $q_l=q_{n_0}$ if $l=n_0$,
and also
\begin{equation}\label{eqprooflem1.1}
 |b_{n,m}|\le C (n/m)^{\lambda_0}, \quad m=1, \ldots, n, n=1,2,\ldots.
 \end{equation}
The first inequality in (\ref{eqlem1.3}) is proved. The second
inequality follows from (\ref{eqprooflem1.1}) and the fact that
\begin{equation}\label{eqprooflem1.2}
d_{n,m} =\lambda_0
\sum_{k=m}^{n-1}\frac{p_k}{k}b_{n,k+1}+p_n-b_{n,m}p_{m-1}.
\end{equation}

\begin{lemma}\label{lem2} {\em (Hu and Zhang, 2004)} For each $k=1,\ldots, K$, we have that
$$
\{N_{n,k}\to \infty\}\; \text{almost surely implies }
\begin{cases}\; \widehat{\theta}_{n,ki}\to \theta_{ki} & \text{ if } \ep\|\bm \xi_{1,k}\|<\infty, \\
\widehat{\theta}_{n,ki}-\theta_{ki}
=O(\sqrt{ \frac{ \log\log N_{n,k} }{ N_{n,k} } }) &  \text{ if } \ep\|\bm \xi_{1,k}\|^2<\infty,
\end{cases}
$$
$i=1,\ldots, d$.
\end{lemma}

The next lemma is a refinement of the theorem of  \cite{Eb86}
on the strong approximation for martingale vectors. The proof can
be found in \cite{Z04}.

\begin{lemma}\label{lem3}Let  $\{\bm Y_n, \mathscr{G}_n; n\ge 1\}$ be
 a sequence of
martingale differences in $\Cal R^{K\times d}$. Suppose that there exist
constants $0<\theta, \epsilon<1$ and a covariance matrix $\bm T$,
measurable with respect to $\mathscr{G}_m$ for some $m\ge 0$, such that
$$
\ep \|\bm Y_n\|^{2+\epsilon}<\infty,\;\; n=1,2,\cdots,
$$
\begin{eqnarray}\label{eqlem3.2}
\bm  \sum_{m=1}^n \ep\left[ \bm Y_m^{\prime} \bm Y_m | \mathscr{G}_{m-1}\right]-n \bm T=O(n^{1-\theta}) \quad a.s. \quad (\text{ or
in }\;\; L_1).
\end{eqnarray}
Then there exists a $(K\times d)$-dimensional standard Brownian motion $\bm
B(t)$ which is independent of $\bm T$ such that
$$ \sum_{m=1}^n \bm Y_m -\bm B(n) \bm T^{1/2}=O(n^{1/2-\tau}) \; a.s. $$
Here $\tau>0$ is a constant depending only on $\theta$, $\epsilon$
and $K,d$.
\end{lemma}

The proof of the next lemma can be found in \cite{Zhang2024}. 
\begin{lemma}\label{lemODE}  Suppose the real vectors $\bm z_n$ and $\bm r_n$ on $ \Theta\subset\mathcal{R}^d$ satisfies   the following recursive procedure:
\begin{equation}\label{eq:lemODE.1} \bm z_{n+1}=\bm z_n+\gamma_{n+1}\bm h(n, \bm z_n)+\gamma_{n+1}\bm r_{n+1},
\end{equation}
with bounded regression function $\bm h(n,\bm z)$ and   errors $\bm r_n$  and,  $\{\gamma_n\}$ being a positive sequence that tends toward zero, such that $\sum_{n=1}^{\infty}\gamma_n$ diverges,  
\begin{equation}\label{eq:lemODE.2} \sum_{n=1}^{\infty} \gamma_n \bm r_n \text{ converges }   \text{ and } \; dist(\bm z_n,\Theta_c)\to 0,   
\end{equation}
where $\Theta_c\subset \Theta$ is a compact set. 

Suppose that $V(\bm z):\mathcal R^d\to \mathcal R_+$ is a continuous function having the property that   for any $\epsilon>0$, there exists $n_0$ for which
\begin{equation}\label{eq:lemODE.3} \left\langle \frac{\partial}{\partial \bm z}  V(\bm z), \bm h(n,\bm z)\right\rangle\le -c_0 V(\bm z)+\epsilon,\;\; n\ge n_0 \text{ and } \bm z\in \Theta_c^{\delta},
\end{equation}
where   $c_0>0$ is a constant,  $\Theta_c^{\delta}=\{\bm z\in \Theta: \|\bm z-\bm y\|\le \delta, \bm y\in \Theta_c\}$.  Then 
\begin{equation}\label{eq:lemODE.4} dist(\bm z_n, \Theta_c\bigcap\Theta_0)\to 0,
\end{equation}
where $\Theta_0=\{\bm z:V(\bm z)=0\}$ is the set of zeros of $V$.
\end{lemma}

\bigskip
{\bf Proof of Theorem \ref{Th1append}.}
For each $k$, if $N_{n,k}\to \infty$, then
$\widehat{\theta}_{n,ki}\to \theta_{ki}$ by Lemma \ref{lem2}.
 If $\sup_n N_{n,k}< \infty$, then
$\widehat{\theta}_{n,ki}$ will not change its value when $n$ is large
enough. In either cases, $\widehat{\theta}_{n,ki}$ has a finite
limit for each $i=1,2,\ldots, d$ and $k=1,2,\ldots, K$. It follows
that there is an $\bm u=(u_1,\cdots,u_K)$ with $0<u_k<1$,
$k=1,2,\ldots, K$ and $u_1+u_2+\ldots+u_K=1$ such that
\begin{equation}\label{eqproofTh1.2}
\widehat{\bm \rho}_n\to \bm u \;\; a.s.
\end{equation}
by the continuity of $\bm \rho(\cdot)$. For proving the theorem, it is sufficient to show that
\begin{equation}\label{eqproofTh1.4}
 \lim_{n\to \infty} \frac{\bm N_n}{n}  =\bm u \quad a.s.
\end{equation}
In fact, (\ref{eqproofTh1.4}) implies that $N_{n,k}\to
\infty$ a.s., $k=1,2,\cdots, K$. By Lemma \ref{lem2} again,
$\widehat{\bm \Theta}_n\to \bm \Theta$ a.s. So, the limit $\bm u$
in (\ref{eqproofTh1.2}) and (\ref{eqproofTh1.4}) must be $\bm
v=\bm \rho(\bm \Theta)$ according to the continuity of
$\bm \rho(\cdot)$.

For showing (\ref{eqproofTh1.4}),   we let  $\bm
M_n=\sum_{m=1}^n\Delta \bm M_m$, where $\Delta \bm M_m=\bm
X_m-\ep[\bm X_m|\mathscr{F}_{m-1}]$ is a sequence of bounded martingale differences. Then
$$ \bm N_n=\bm N_{n-1}+\bm g\big(n-1,\frac{\bm N_{n-1}}{n-1},\widehat{\bm \rho}_{n-1}\big)
+\Delta \bm M_n$$
and
$$ \sum_{n=1}^{\infty}\frac{\ep[(\Delta X_{m,i})^2|\mathscr{F}_{m-1}]}{n^2}\le \sum_{n=1}^{\infty}\frac{1}{n^2}<\infty, \;\; i=1,\ldots,K. $$
It follows that 
$$ \sum_{n=1}^{\infty} \frac{\Delta \bm M_n }{n} \text{ converges almost surely}. $$
Write  $\bm x_n=\frac{\bm N_n}{n}$, $\bm y_n=\widehat{\bm \rho}_n$ and $\bm Z_n=(\bm x_n,\bm y_n)$.  Then 
\begin{equation}\label{eqSA} \bm x_{n+1}-\bm x_n =\frac{\bm g\big(n, \bm x_n,\bm y_n\big)-\bm x_n}{n+1}+\frac{\Delta \bm M_{n+1} }{n+1},
\end{equation}
$$ \bm y_{n+1}=\bm y_n+\frac{\bm 0}{n+1}+\frac{\bm r_{n+1,2}}{n+1}, \;\; \bm r_{n+1,2}=(n+1)(\bm y_{n+1}-\bm y_n). $$
By \eqref{eqproofTh1.2}, we have
$$ \sum_{n\le j\le m}\frac{\bm r_{j,2}}{j}=\bm y_m-\bm y_{n-1}\to 0 \; a.s. \text{ as } n,m\to \infty. $$
Thus, for any $\delta>0$, with probability $1$, when $n$ is large enough, $\bm Z_n$ satisfies \eqref{eq:lemODE.1} and \eqref{eq:lemODE.2} with $\gamma_n=1/n$, $\Theta_c=\{\bm z=(\bm x,\bm y): \bm y=\bm u, 0\le x_i\le 1, i=1,2,\ldots, K, \sum_{i=1}^Kx_i=1\}$,  
$$\bm h(n, \bm z)=\left(\bm g\big(n,\bm x,\bm y)\big)-\bm x, \bm 0\right), \bm z=(\bm x,\bm y), \;\text{ and }  \;\bm r_{n+1}=\left(\Delta\bm M_{n+1}, \bm r_{n+1,2}\right). $$

When the condition (i) is satisfied, we choose $V(\bm z)=\sum_{i=1}^K(x_i-y_i)^2I\{x_i>y_i\}$. Then
$\frac{\partial V}{\partial x_i}=2(x_i-y_i)I\{x_i>y_i\}$,
\begin{align*} \left\langle \frac{\partial }{\partial \bm z} V(\bm z), \bm h(n,\bm z)\right\rangle=&
2\sum_{i=1}^K(x_i-y_i)I\{x_i>y_i\}\big(g_i(n,\bm x,\bm y)-x_i\big)\\
=&-2V(\bm z)+2\sum_{i=1}^K(x_i-y_i)I\{x_i>y_i\}\big(g_i\big(n,\bm x,\bm y)-y_i\big).
\end{align*}
By the condition (i), we have for $n\ge n_0$,
\begin{align*}
&\sum_{i=1}^K(x_i-y_i)I\{x_i>y_i+\epsilon\}(g_i\big(n,\bm x,\bm y)-y_i)\\
 \le &  \sum_{i=1}^K(x_i-y_i)I\{x_i>y_i+\epsilon\}[ \lambda(x_i-y_i)+\epsilon]\le  \lambda V(\bm z)+\epsilon
\end{align*}
and
$$\big|\sum_{i=1}^K(x_i-y_i)I\{y_i<x_i\le y_i+\epsilon\}(x_i-g_i\big(n,\bm x,\bm y))\big|\le 2\epsilon.$$
It follows that
$$\left\langle \frac{\partial }{\partial \bm z} V(\bm z), \bm h(n,\bm z)\right\rangle\le -2(1-\lambda)V(\bm z)+5\epsilon,\;\; n\ge n_0. $$
Hence, \eqref{eq:lemODE.3} is satisfied. By Lemma \ref{lemODE}, $dist(\bm  Z_n,\Theta_0\bigcap\Theta_c)\to 0$ a.s., where $\Theta_0\bigcap\Theta_c=\{\bm z: V(\bm z)=0, \bm z\in \Theta_c\}= \{(\bm u,\bm u) \}$.  \eqref{eqproofTh1.4} is shown.

(ii) and (iii) are special cases of (iv) with $\alpha_i(y_i)=y_i$ and $\alpha_i(y_i)=1$, respectively.  For (iv), we choose $V(\bm z)=\sum_{i=1}^K(x_i-y_i)^2/a_i(y_i)$, $\bm z=(\bm x,\bm y)$. Then
$\frac{\partial V}{\partial x_i}=2(x_i-y_i)/a_i(y_i)$,
\begin{align*}
  \left\langle \frac{\partial }{\partial \bm z} V(\bm z), \bm h(n,\bm z)\right\rangle=& 2\sum_{i=1}^K (g_i\big(n,\bm x,\bm y)-x_i)\frac{x_i-y_i}{a_i(y_i)}\\
=&-2V(\bm z)+2\sum_{i=1}^K (g_i\big(n,\bm x,\bm y)-y_i)\frac{x_i-y_i}{a_i(y_i)}.
\end{align*}
Write $g_i=g_i(n,\bm x,\bm y)$, $l_i=\frac{x_i-y_i}{\alpha_i(y_i)}$.  Notice that  $\bm x, \bm y,\bm g\in \{\bm x:x_i\ge 0,  i=1,\cdots, K, \sum_{i=1}^K x_i=1\}$, and $y_i>0$, $1/a_i(y_i)\le c_0$,  $i=1,\ldots,K$, when $\|\bm y-\bm u\|\le \delta$ and $\delta>0$ is small enough. We have for $n\ge n_0$, 
\begin{align*}
&\sum_{i=1}^K\left( g_i- y_i\right)l_i
=\frac{1}{2}\sum_{i,j} y_iy_j \big( g_i/y_i- g_j/y_j\big)\big(l_i-l_j\big) \\
\le &  \frac{1}{2}\sum_{i,j} y_iy_j \big| g_i/y_i- g_j/y_j\big|\epsilon+\frac{1}{2}\sum_{i,j} y_iy_j \big( g_i/y_i- g_j/y_j\big)\big(l_i-l_j\big)I\{|l_i-l_j|>\epsilon\}\\
\le 
& \epsilon + \sum_{i,j} y_iy_j \big( g_i/y_i- g_j/y_j\big)\big(l_i-l_j\big)I\{l_i-l_j>\epsilon\}\\
\le 
& \epsilon + \sum_{i,j} y_iy_j \left[\lambda\big( x_i/y_i- x_j/y_j\big)+\epsilon\right]\big(l_i-l_j\big)I\{l_i-l_j>\epsilon\}\\
\le 
& \epsilon +c_0\epsilon+  \lambda \sum_{i,j} y_iy_j  (x_i/y_i- x_j/y_j) \big(l_i-l_j\big)\left[-I\{0\le l_i-l_j\le \epsilon\}+I\{l_i-l_j>0\}\right]\\
\le 
& \epsilon +c_0\epsilon+  \epsilon  \sum_{i,j} y_iy_j  (x_i/y_i+ x_j/y_j) +\lambda \sum_{i=1}^K\left( x_i- y_i\right)l_i\\
=& (2+c_0)\epsilon+\lambda V(\bm z). 
\end{align*}
 
 Thus,
 \begin{align*}
  \left\langle \frac{\partial }{\partial \bm z} V(\bm z), \bm h(n,\bm z)\right\rangle=-2(1-\lambda)V(\bm z)+(2+c_0)\epsilon, \;\; n\ge n_0, \bm z\in \Theta_c^{\delta}.
\end{align*}
Hence, \eqref{eq:lemODE.3} is satisfied. By Lemma \ref{lemODE}, $dist(\bm Z_n,\Theta_0\bigcap\Theta_c)\to 0$ a.s., where $\Theta_0\bigcap\Theta_c=\{\bm z: V(\bm z)=0, \bm z\in \Theta_c\}= \{(\bm u,\bm u) \}$. \eqref{eqproofTh1.4} is also shown. 
$\Box$

\bigskip
{\bf Proof of Theorem \ref{Th2append}.} Recall
$$ \bm N_n=\bm N_{n-1}+\bm g\big(n-1,\frac{\bm N_{n-1}}{n-1},\widehat{\bm \rho}_n\big)
+\Delta \bm M_n. $$
 According to the Condition (C1), we can assume
that
\begin{align*}
p_{n-1,k}=g_k\big(n-1,\frac{\bm
N_{n-1}}{n-1},\widehat{\bm \rho}_{n-1}\big)  \le &
\widehat{\rho}_{n-1,k}+\lambda_0\big(\frac{N_{n-1,k}}{n-1}-\widehat{\rho}_{n-1,k}\big)+\delta\Big\|\frac{\bm N_{n-1}}{n-1}-\widehat{\bm \rho}_{n-1}\Big\|\\
&\;\; \text{ if } N_{n-1,k}>(n-1)\widehat{\rho}_{n-1,k}+L_0,
\end{align*}
for all $k=1,2,\ldots, K$ and $n\ge n_0$. Write
$$\bm
U_n^{(1)}=\sum_{m=1}^n\widehat{\bm \rho}_{m-1}+\bm
M_n-n\widehat{\bm \rho}_n, \; U_{n}^{(2)}=\sum_{m=1}^{n-1}\|\frac{\bm N_m}{m}-\widehat{\bm \rho}_{m}\|, \; U_{n,k}=U_{n,k}^{(1)}+\delta U_{n}^{(2)}.
$$
Then
\begin{align*}
&N_{n,k}-n\widehat{\rho}_{n,k}
=N_{n-1,k}-(n-1)\widehat{\rho}_{n-1,k}+\Delta
U_{n,k}^{(1)}+p_{n-1,k}-\widehat{\rho}_{n-1,k}
\\
\le & \Big(1+\frac{\lambda_0}{n-1}\Big)
\Big(N_{n-1,k}-(n-1)\widehat{\rho}_{n-1,k}\Big)+\Delta U_{n,k}
\;\; \text{ if } \;\; N_{n-1,k}-(n-1)\widehat{\rho}_{n-1,k}> L_0.
\end{align*}
So
\begin{align*}
N_{n,k}-n\widehat{\rho}_{n,k}-L_0
\le & \Big(1+\frac{\lambda_0}{n-1}\Big)
\Big(N_{n-1,k}-(n-1)\widehat{\rho}_{n-1,k}-L_0\Big)+\Delta U_{n,k}+\frac{\lambda_0L_0}{n-1}\\
&\;\; \text{ if } \;\; N_{n-1,k}-(n-1)\widehat{\rho}_{n-1,k}> L_0
\end{align*}
for all $k=1,2,\ldots, K$ and $n\ge n_0$. Also, it is obvious that
$|p_{n-1,k}-\widehat{\rho}_{n-1,k}|\le 2$. Applying Lemma
\ref{lem1} with  $q_n=N_{n,k}-n\widehat{\rho}_{n,k}-L_0$,
$\Delta p_n=\Delta U_{n,k}+\frac{\lambda_0^+L_0}{n-1}$  and $K_0=2$ yields
\begin{align}
&\big(N_{n,k}-n\widehat{\rho}_{n,k}\big)^+\le \big(N_{n,k}-n\widehat{\rho}_{n,k}-L_0\big)^+ +L_0
\nonumber\\
\le &
\max_{|\lambda_0|\le m\le n}\Big|b_{n,m}\big(\Delta
U_{m,k}+\frac{\lambda_0L_0}{m-1}\big)+b_{n,m+1}\big(\Delta
U_{m+1,k}+\frac{\lambda_0L_0}{m}\big)+\ldots\nonumber \\
& \qquad+b_{n,n}\big(\Delta U_{n,k}+\frac{\lambda_0L_0}{n-1}\big)\Big|
+O(n^{\lambda_0}+1)\nonumber\\
\le &
\max_{|\lambda_0|\le m\le n}\Big|b_{n,m}\Delta
U_{m,k}+b_{n,m+1}\Delta
U_{m+1,k}+\ldots+b_{n,n}\Delta U_{n,k}\Big|
+O(n^{\lambda_0}+1) \label{eqproofTh1.1ad}\\
\le & \delta
\max_{|\lambda_0|\le m\le n}\big|b_{n,m}\Delta
U_{m}^{(2)}+b_{n,m+1}\Delta
U_{m+1}^{(2)}+\ldots+b_{n,n}\Delta U_{n}^{(2)}\big|
\label{eqproofTh1.1ad2}\\
&  + C\Big\{
\sum_{m=1}^n\frac{|U_{m,k}^{(1)}|}{m}\Big(\frac{n}{m}\Big)^{\lambda_0}+
\max_{1\le m\le
n}\Big(\frac{n}{m}\Big)^{\lambda_0}|U_{m,k}^{(1)}|\Big\}+O(n^{\lambda_0}+1)=:\delta I_1+I_2.
\label{eqproofTh1.1}
\end{align}
For $\lambda_0<1/2$, choose $0<\alpha<1/2$ such that $\lambda_0<1/2-\alpha$. Write
$$B_{n,k}=\max_{1\le m\le n} \frac{|U_{m,k}^{(1)}|}{m^{1/2-\alpha}} \text{ and }
 A_{n}=\max_{1\le m\le n} \frac{\|\bm N_{m}-m\widehat{\bm \rho}_{m}|}{m^{1/2-\alpha}}. $$ We first show that
\begin{equation}\label{eq:proofTh2.17} A_{n}=O\left(\sum_{k=1}^KB_{n,k}\right) \;\; a.s.
\end{equation}
Note  $\lambda_0< 1/2-\alpha$. By (\ref{eqproofTh1.1}), we have
\begin{align*}
I_2
 \le & O(1)\Big\{
\sum_{m=1}^n\frac{B_{m,k} m^{1/2-\alpha}}{m}\Big(\frac{n}{m}\Big)^{\lambda_0}\Big\}+O(n^{\lambda_0}+1) \\
= & B_{n,k}O(1)\sum_{m=1}^n\frac{m^{1/2-\alpha}}{m}\Big(\frac{n}{m}\Big)^{\lambda_0} +O(n^{\lambda_0}+1)=B_{n,k}O(n^{1/2-\alpha})\;\; a.s.
\end{align*}

For $I_1$,   observe that
\begin{align}\label{eqproofTh3.4}
b_{n,m}=&\exp\Big\{\sum_{l=m}^{n-1}\log\big(1+\frac{\lambda_0}{l}\big)\Big\}
=\exp\Big\{\lambda_0\sum_{l=m}^{n-1}\big(\frac{1}{l}+\frac{O(1)}{l^2}\big)\Big\}
\nonumber\\
=&\exp\Big\{\lambda_0\log\frac{n}{m}+\frac{O(1)}{m}\Big\}
=\big(\frac{n}{m}\big)^{\lambda_0}\Big(1+\frac{O(1)}{m}\Big),\\
& \text{ as } n\ge m\to \infty. \nonumber
\end{align}
Note $b_{n,m}=b_{l,m}b_{n,l}$. (\ref{eqproofTh3.4}) also holds for fixed $m$, i.e. $b_{n,m}=O(n^{\lambda_0})$. Notice that
$$ \sum_{j=m}^n \Delta U_{j}^{(2)}\big(\frac{n}{j}\big)^{\lambda_0}\frac{O(1)}{j}=O(1)\sum_{j=m}^n \big(\frac{n}{j}\big)^{\lambda_0}\frac{1}{j}=O(n^{\lambda_0}+1).$$
and
$$ \sum_{j=m}^n \big(\frac{n}{j}\big)^{\lambda_0}\frac{(j-1)^{1/2-\alpha}}{j-1}\le \frac{n^{1/2-\alpha}}{1/2-\alpha-\lambda_0}+O(n^{\lambda_0}+1).$$
Then
$$ \delta I_1\le A_{n}  \frac{2\delta n^{1/2-\alpha}}{1/2-\alpha-\lambda_0}+o(n^{1/2-\alpha})\;\; a.s. $$
It follows that
$$ (N_{n,k}-n\widehat{\rho}_{n,k})^+\le A_{n}\frac{2\delta n^{1/2-\alpha}}{1/2-\alpha-\lambda_0}+O(B_{n,k}n^{1/2-\alpha})\; a.s. $$
Note the fact
$\sum_{k=1}^K(N_{n,k}-n\widehat{\rho}_{n,k})^-=\sum_{k=1}^K(N_{n,k}-n\widehat{\rho}_{n,k})^+$. It follows that
\begin{align*}
& \frac{\|\bm N_{n}-n\widehat{\bm \rho}_{n}\|}{n^{1/2-\alpha}}\le \frac{\sum_{k=1}^K|N_{n,k}-n\widehat{\rho}_{n,k}|}{n^{1/2-\alpha}}\\
\le & 2\frac{\sum_{k=1}^K(N_{n,k}-n\widehat{\rho}_{n,k})^+}{n^{1/2-\alpha}}
\le   \frac{4K\delta A_{n}}{1/2-\alpha-\lambda_0}+O(\sum_{k=1}^KB_{n,k}).
\end{align*}
Note that $A_{n}$ and $B_{n,k}$ are non-decreasing in $n$. Choosing $\delta$ small enough such that $\frac{4\delta K}{1/2-\alpha-\lambda_0}<1/2$. Then \eqref{eq:proofTh2.17} holds. That is
\begin{equation}\label{eq:proofTh2.19}
\max_{m\le n} \frac{\|\bm N_m-m\widehat{\bm \rho}_m\|}{m^{1/2-\alpha}}=O\left(\max_{m\le n} \frac{\|\bm U_m^{(1)}\|}{m^{1/2-\alpha}}\right)\; a.s.
\end{equation}

Now, by Theorem
\ref{Th1append}, we have $\bm N_n/n\to \bm v$ a.s. and $v_k>0$,
$k=1,\ldots, K$. Then by Lemma \ref{lem2} we have
\begin{equation}\label{eqproofTh2.1}
\widehat{\bm  \Theta}_n-\bm  \Theta=O\big(\sqrt{\frac{\log\log n}n}\big) \quad a.s.
\end{equation}
So, by the Condition \ref{ConAappend},
\begin{equation}\label{eqproofTh2.2}\widehat{\bm \rho}_n-\bm v
=\bm \rho(\widehat{\bm \Theta}_n)-\bm \rho(\bm \Theta)
=O\big(\sqrt{\frac{\log\log n}n}\big) \quad a.s.
\end{equation}
  On the other hand,
according to the law of iterated logarithm for martingales, we have
$\bm M_n =O(\sqrt{n\log\log n}) \;\; a.s.
$
since $\{\bm M_n\}$ is a martingale with $|\Delta M_{n,k}|\le 1$,
$k=1,2,\ldots,K$. Hence, we conclude that
$$ \bm U_n^{(1)}=O\left(\sqrt{n\log\log n }\right) \quad a.s.,  $$
 which, together with \eqref{eq:proofTh2.19}, yields
 $$ \max_{m\le n} \frac{\|\bm N_m-m\widehat{\bm \rho}_m\|}{m^{1/2-\alpha}}=O(n^{\alpha}\sqrt{\log \log n})\; a.s. $$
 Hence $\|\bm N_n-n\widehat{\bm \rho}_n\|=O(\sqrt{n\log \log n})$ a.s. \eqref{eqThappend2.1} is proved.

Next, let  $\bm Q_n=\sum_{m=1}^n\Delta
\bm Q_m$, where $\Delta \bm Q_m=(\Delta \bm Q_{m,1},\ldots,\Delta
\bm Q_{m,K})=(\Delta Q_{m,ki};i=1,\ldots,d, k=1,\ldots, K)$ and
$\Delta Q_{m,ki}=X_{m,k}\frac{\xi_{m,ki}-\theta_{ki}}{v_k}$,
$i=1,\ldots, d$, $k=1,\ldots, K$. Then $\{\bm Q_n,\mathscr{F}_n\}$ is a
 martingale in $\Cal R^{K\times d}$ with
$$ \ep[\Delta Q_{m,ki}\Delta Q_{m,lj}|\mathscr{F}_{m-1}]=0 \;\; \text{
for all }\;\; i,j \text{ and }\;\; k\ne l $$
and
\begin{align*}
&\sum_{m=1}^n\ep[\Delta Q_{m,ki}\Delta Q_{m,kj}|\mathscr{A}_{m-1}]
=\sum_{m=1}^n \ep[X_{m,k}|\mathscr{F}_{m-1}]\Cov\big[\xi_{1,ki},\xi_{1,kj}\big]/v_k^2\\
=& (N_{n,k}-M_{n,k})\Cov\big[\xi_{1,ki},\xi_{1,kj}\big]/v_k^2 \\
=& n\Cov\big[\xi_{1,ki},\xi_{1,kj}\big]/v_k+O(\sqrt{n\log\log n
}+n^{\lambda_0}) \;\; a.s. \;\; \text{ for all }\;\; i,j, k
\end{align*}
according to (\ref{eqThappend2.1}). So, $\{\bm Q_n,\mathscr{A}_n\}$
is  a martingale with
$\ep\|\Delta\bm Q_n\|^{2+\epsilon}\le C_0
$
and
$$ \sum_{m=1}^n \ep\big[\big\{\Delta\bm Q_m\big\}^{\prime}  \Delta\bm
Q_m \big|\mathscr{A}_{m-1}\big]=n
\bm V+O(\sqrt{n\log\log n }+n^{\lambda_0}) \;\; a.s.
$$
By Lemma \ref{lem3}, we can define an $(K\times d)$-dimensional
 standard Brownian motion $\bm B(t)$ such that
\begin{equation}\label{eq:approxforQ} \bm Q_n =\bm B(n)
\bm V^{1/2}+o(n^{1/2-\tau}) \;\; a.s. \; \text{ for some }
\tau>0.
\end{equation}
On the other hand, with the same argument for proving (A.11) of \cite{HZ04}, we can derived that
$$ \widehat{\bm \Theta}_n-\bm \Theta=\frac{\bm Q_n}{n} +o(n^{-1/2})\;\; a.s.\;\text{ and }$$
$$  \widehat{\bm \rho}_n-\bm v= \big(\widehat{\bm \Theta}_n-\bm \Theta\big)\frac{\partial \bm \rho}{\partial\bm
y}\big|_{\bm \Theta} +o\big(\|\widehat{\bm \Theta}_n-\bm \Theta\|\big)=\frac{\bm Q_n}{n}\frac{\partial \bm \rho}{\partial\bm
y}\big|_{\bm \Theta}+o\big(\frac{\|\bm Q_n\|}{n}\big)+o(n^{-1/2})\;\; a.s.
$$
The proof of (\ref{eqThappend2.3}) and (\ref{eqThappend2.4}) is now completed.

Finally, notice that
\begin{align*}
\bm U_n^{(1)}=\bm M_n+\sum_{m=1}^{n-1} \frac{\bm B(m)}{m}\bm\Sigma^{1/2}-\bm B(n)\bm\Sigma
+o\left(\frac{\|\bm B(m)\|}{m}+\|\bm B(n)\|\right)+o(n^{1/2})\; a.s.
\end{align*}
It follows that
\begin{align*}
\max_{m\le n}\frac{\|\bm U_m^{(1)}\|}{m^{1/2-\alpha}}\le \max_{m\le n}\frac{\|\bm M_m\|}{m^{1/2-\alpha}}+C \max_{m\le n}\frac{\|\bm B(m)\|}{m^{1/2-\alpha}}+o(n^{\alpha}) \;\; a.s.
\end{align*}
Note that for $2^p\le n<2^{p+1}$,
$$ \ep\left[\max_{m\le n}\frac{\|\bm M_n\|}{m^{1/2-\alpha}}\right]
\le \sum_{l=0}^p\ep\left[\max_{2^l\le m\le 2^{l+1}}\frac{\|\bm M_m\|}{2^{(1/2-\alpha)l}}\right]
\le   \sum_{l=0}^p \frac{C2^{(l+1)/2}}{2^{(1/2-\alpha)l}} \le C n^{\alpha},
$$
and similarly,
\begin{equation}\label{eq:boundofB(m)} \ep\left[\max_{m\le n}\frac{\|\bm B(m)\|}{m^{1/2-\alpha}}\right]\le C n^{\alpha}.
\end{equation}
By \eqref{eq:proofTh2.19}, it follows that
\begin{equation}\label{eq:boundofN(m)}
\max_{m\le n} \frac{\|\bm N_m-m\widehat{\bm \rho}_m\|}{m^{1/2-\alpha}}=O\left(\max_{m\le n} \frac{\|\bm U_m^{(1)}\|}{m^{1/2-\alpha}}\right)=O_P(n^{\alpha}).
\end{equation}
Hence, $\max_{m\le n}\|\bm N_m-m\widehat{\bm \rho}_m\|=O_P(n^{1/2})$, which, together with \eqref{eqThappend2.3}, yields \eqref{eqThappend2.2}. The proof is now completed.
$\Box$ 
\bigskip

{\bf Proof of Theorem \ref{Th3append}.} It is sufficient to show \eqref{eqThappend3.1}. Recall that Conditions 
\ref{ConAappend}, \ref{ConBappend} and \ref{ConCappend} are now satisfied. By
Theorem \ref{Th1append} and the Condition (C2), the Condition (C1) is
satisfied for any $\lambda_0<0$. And also, by Theorem \ref{Th2append},
(\ref{eqThappend2.1}), (\ref{eqThappend2.3})  and (\ref{eqThappend2.4}) hold. We will proved
(\ref{eqThappend3.1}) by, for any fixed $\lambda_0<0$,
finding  two  sequences $\{Z_n=Z_n(\lambda_0)\}$ and $\{F(n)\}$ of
positive random variables, such that
\begin{equation}\label{eqproofTh3key}
\begin{aligned}
& \max_{i\le n}\big\|\bm N_i-i\widehat{\bm\rho}_i\big\|\le
Z_n+o(F(n))+o(n^{1/2}) \;\; a.s. \;\;\text{ with }\\
& \;\; \ep Z_n\le C_0
\sqrt{n}/(1-2\lambda_0)^{1/4}  \text{ and } \ep[F(n)]\le C_0\sqrt{n},
\end{aligned}
\end{equation}
for all $n$, where $C_0$ is a   positive constant which does not depend on
$\lambda_0$ and $n$.  The technique of Gaussian approximation and Gaussian comparison theorem will be used to show this estimate.

According to (\ref{eqThappend2.4}), we have 
\begin{align}\label{eqproofTh3.1}
\bm U_n^{(1)} = & \sum_{m=1}^n (\widehat{\bm \rho}_m-\bm
v)-n(\widehat{\bm \rho}_n-\bm v)+\bm M_n \nonumber\\
=& \sum_{m=1}^{n-1} \frac{\bm W(m)}{m}   -\bm W(n)
 +\bm M_n+o\left(\sum_{m=1}^{n-1} \frac{\|\bm W(m)\|}{m}   +\|\bm W(n)\|
 \right)+o(n^{1/2})\nonumber
\\
=& \bm G(n)+\bm M_n+o(F(n))+o(n^{1/2}) \;\; a.s.,
\end{align}
where
$$\bm G(t)= -\bm W(t)+\int_0^t\frac{\bm W(x)}{x}dx\; \text{ and }\; F(t)=\max_{s\le t}\|\bm W(s)\|+ \int_0^t\frac{\|\bm W(x)\|}{x}dx$$
Write $\bm R_n=\bm U_n^{(1)}-\bm G(n)-\bm M_n$. Then $\bm
R_n=o(n^{1/2}+F(n))$ a.s. according (\ref{eqproofTh3.1}). Remember now $\lambda_0<0$.
 By
(\ref{eqproofTh1.1ad}), we have
\begin{align}\label{eqproofTh3.3}
&\big(N_{n,k}-n\widehat{\rho}_{n,k}\big)^+
\nonumber\\
\le &\max_{|\lambda_0|\le m\le n}\big|b_{n,m}\Delta
U_{m,k}+b_{n,m+1}\Delta U_{m+1,k}+\ldots+b_{n,n}\Delta
U_{n,k}\big|+O(1)
\nonumber\\
\le &\delta
\max_{-\lambda_0\le m\le n}\big|b_{n,m}\Delta
U_{m}^{(2)}+b_{n,m+1}\Delta
U_{m+1}^{(2)}+\ldots+b_{n,n}\Delta U_{n}^{(2)}\big|\\
+ &\max_{-\lambda_0\le m\le n}\big|b_{n,m}\Delta
M_{m,k}+b_{n,m+1}\Delta M_{m+1,k}+\ldots+b_{n,n}\Delta
M_{n,k}\big|
\nonumber\\
&+ \max_{-\lambda_0\le m\le n}\big|b_{n,m}\Delta
G_k(m)+b_{n,m+1}\Delta G_k(m+1)+\ldots+b_{n,n}\Delta G_k(n)\big|
\nonumber\\
& +\max_{-\lambda_0\le m\le n}\big|b_{n,m}\Delta
R_{m,k}+b_{n,m+1}\Delta R_{m+1,k}+\ldots+b_{n,n}\Delta
R_{n,k}\big| +O(1)\nonumber \\
=:& \delta I_{n}+II_{n,k}+III_{n,k}+IV_{n,k}+O(1)\;\; a.s.
\end{align}
Applying (\ref{eqprooflem1.2}) and (\ref{eqprooflem1.1}) yields
\begin{align*}
&IV_{n,k}=\max_{-\lambda_0\le m\le n}\left| \lambda_0
\sum_{l=m}^{n-1}\frac{R_{l,k}}{l}b_{n,l+1}+R_{n,k}-b_{n,m}R_{m-1,k}\right|\\
=&\sum_{-\lambda_0\le l\le
n-1}\frac{o(l^{1/2}+F(l))}{l}\big(\frac{n}{l}\big)^{\lambda_0} \\
& +
o(n^{1/2}+F(n))+ \max_{-\lambda_0\le m\le
n}\big(\frac{n}{m}\big)^{\lambda_0}o(m^{1/2}+F(m))
= o(F(n))+o(n^{1/2})\;\; a.s.
\end{align*}
For $III_{n,k}$, observe  \eqref{eqproofTh3.4}.  It follows that
\begin{align*}
& b_{n,m}\Delta G_k(m)+b_{n,m+1}\Delta
G_k(m+1)+\ldots+b_{n,n}\Delta G_k(n) \\
=&\lambda_0
\sum_{l=m}^{n-1}\frac{G_k(l)}{l}b_{n,l+1}+G_k(n)-b_{n,m}G_k(m-1)\\
=&\lambda_0
\sum_{l=m}^{n-1}\frac{G_k(l)}{l}\big(\frac{n}{l}\big)^{\lambda_0}
+G_k(n)-\big(\frac{n}{m}\big)^{\lambda_0}G_k(m-1)\\
&+\lambda_0
\sum_{l=m}^{n-1}\frac{G_k(l)}{l}\big(\frac{n}{l}\big)^{\lambda_0}\frac{O(1)}{l}
-\big(\frac{n}{m}\big)^{-\lambda_0}G_k(m-1)\frac{O(1)}{m}\\
=&\lambda_0\int_m^n\frac{G_k(x)}{x}\big(\frac{n}{x}\big)^{\lambda_0}dx
+G_k(n)-\big(\frac{n}{m}\big)^{\lambda_0}G_k(m)+O(1)\\
&+\sum_{l=1}^{n-1}\frac{O(\sqrt{l\log\log
l})}{l^2}\big(\frac{n}{l}\big)^{\lambda_0}
+\big(\frac{n}{m}\big)^{\lambda_0}O(\frac{\sqrt{m\log\log m}}{m})
\\
=&\int_m^n \big(\frac{x}{n}\big)^{-\lambda_0}d
G_k(x)+O(1)\\
=&\int_0^n \big(\frac{x}{n}\big)^{-\lambda_0}d
G_k(x)-\big(\frac{m}{n})^{-\lambda_0}\int_0^m
\big(\frac{x}{m}\big)^{-\lambda_0}d G_k(x)+O(1)\quad a.s.
\end{align*}
It follows that
\begin{align}\label{eqproofTh3II}
 \max_{i\le n}III_{i,k} \le   2\max_{1\le m\le n}\Big|\int_0^m
\big(\frac{x}{m})^{-\lambda_0}d G_k(x)\Big| +O(1)
=:2\overline{III}_{n,k} +O(1)\;\; a.s.
\end{align}
Note that  the process
\begin{align*}
&\int_0^t\big(\frac{x}{t}\big)^{-\lambda_0}d
 G_k(x) = -\int_0^t\big(\frac{x}{t}\big)^{-\lambda_0}d
W_k(x)+\int_0^t\big(\frac{x}{t}\big)^{-\lambda_0}\frac{ W_k(x)}{x}d
x  \\
=& \big(-1+\frac{1}{\lambda_0}\big)\int_0^t\big(\frac{x}{t}\big)^{-\lambda_0}d
 W_k(x)-\frac{1}{\lambda_0}  W_k(t)
 =\int_0^t\Big[\big(-1+\frac{1}{\lambda_0}\big)\big(\frac{x}{t}\big)^{-\lambda_0}-\frac{1}{\lambda_0}\Big]d
 W_k(x)
\end{align*}
is a mean
zero Gaussian process with covariance function
\begin{align*}&\sigma_{kk}\int_0^s\Big[\big(-1+\frac{1}{\lambda_0}\big)\big(\frac{x}{t}\big)^{-\lambda_0}-\frac{1}{\lambda_0}\Big]
\Big[\big(-1+\frac{1}{\lambda_0}\big)\big(\frac{x}{s}\big)^{-\lambda_0}-\frac{1}{\lambda_0}\Big]d
x \\
= & \frac{\sigma_{kk}}{1-2\lambda_0}\big(\frac{s}{t}\big)^{-\lambda_0}s
\begin{cases} \le \ep\Big[ \frac{  W_k(s)}{\sqrt{1-2\lambda_0}}\frac{  W_k(t)}{\sqrt{1-2\lambda_0}}\Big], &
  0<s\le t, \\
  = \ep\Big[ \frac{  W_k(s)}{\sqrt{1-2\lambda_0}}\frac{  W_k(t)}{\sqrt{1-2\lambda_0}}\Big], &
  t=s>0,
  \end{cases}
\end{align*}
 where $\sigma_{kk}$ is the $(k,k)$th element of $\bm\Sigma$ and is the variance of $W_k(1)$.
Applying the comparison theorem for Gaussian processes (Slepian's lemma, c.f. Corollary 3.12 of \cite{LedouxTalagrand}) yields
$$\ep\Big[ \max_{0< t\le
n}\int_0^t\big(\frac{x}{t}\big)^{-\lambda_0}d G_k(x)\Big]\le
\frac{1}{\sqrt{1-2\lambda_0}}\ep\big[\max_{0< t\le n}W_k(t)\big]
\le \frac{2 }{\sqrt{1-2\lambda_0}}\sqrt{n\sigma_{kk}}.
$$
We conclude that the expectation of the random variable
$\overline{III}_{n,k}$ in (\ref{eqproofTh3II}) satisfies
$$\ep[\overline{III}_{n,k}]\le 2\ep\Big[ \max_{1\le m\le
n}\int_0^m\big(\frac{x}{m}\big)^{-\lambda_0}d G_k(x)\Big]\le
\frac{4}{\sqrt{1-2\lambda_0}}\sqrt{n\sigma_{kk}}.
$$

Next, we consider $II_{n,k}$. If the martingale $\bm M_n$ can be
approximated by a Brownian motion, it can be treated in the same
way as $II_{n,k}$. Unfortunately, this martingale may not satisfy
the condition (\ref{eqlem3.2}) in Lemma \ref{lem3}. We shall
estimate it directly.
 By (\ref{eqproofTh3.4}) and the fact
that $|\Delta M_{n,k}|\le 1$, we have
\begin{align*}
&b_{n,m}\Delta M_{m,k}+b_{n,m+1}\Delta
M_{m+1,k}+\ldots+b_{n,n}\Delta M_{n,k}\\
=&\sum_{l=m}^n\big(\frac{n}{l}\big)^{\lambda_0}\Delta M_{l,k} +
\sum_{l=m}^n\big(\frac{n}{l}\big)^{\lambda_0}\frac{O(1)}{l}  
 =n^{\lambda_0}\sum_{l=m}^{n-1}l^{-\lambda_0}\Delta M_{l,k}+O(1)
 \;\; a.s.
\end{align*}
It follows that
$$II_{n,k}\le 2n^{\lambda_0}\max_{1\le m\le n}\big|
\sum_{l=1}^{m-1}l^{-\lambda_0}\Delta
M_{l,k}\big|+O(1)=:2\overline{II}_{n,k}+O(1) \;\; a.s.$$ Notice
that $\{l^{-\lambda_0}\Delta M_{l,k}\}$ is a sequence of martingale
differences. According to Burkholder's  inequality (c.f. Theorem 2.10 of \cite{hall1980martingale}), we have
\begin{align*}
&\ep\big[\max_{1\le m\le n}\big|
\sum_{l=1}^{m-1}l^{-\lambda_0}\Delta M_{l,k}\big|^4\big] \le
C_0^4\ep  \big|
\sum_{l=1}^{n-1}(l^{-\lambda_0}\Delta M_{l,k})^2\big|^2 \\
 \le &C_0^4\big[ \sum_{l=1}^{n-1}l^{-2\lambda_0}\big]^2\le C_0^4\big[
\int_1^n
x^{-2\lambda_0}dx\big]^2=\frac{C_0^4}{(1-2\lambda_0)^2}n^{2(1-2\lambda_0)},
\end{align*}
where $C_0=18\times 4\times \sqrt{4/3}$. Let $\theta=1-\frac{1}{4\lambda_0}$ for $\lambda_0\le -1$, $\theta=\sqrt{2}$ for $-1<\lambda_0<0$, then
\begin{align*}
&\ep\big[\max_{i< \theta^{q+1}}(\overline{II}_{i,k})^4\big] \le
\sum_{p=0}^q\ep\big[\max_{\theta^p\le i<
\theta^{p+1}}(\overline{II}_{i,k})^4\big]
\le  \sum_{p=0}^q\theta^{4p\lambda_0}\ep\big[\max_{1\le m<
\theta^{p+1}}\big| \sum_{l=1}^{m-1}l^{-\lambda_0}\Delta
M_{l,k}\big|^4\big] \\
\le &\frac{C_0^4}{(1-2\lambda_0)^2}
\sum_{p=0}^q\theta^{2(1-2\lambda_0)(p+1)+4p\lambda_0} = \frac{C_0^4
\theta^{2-4\lambda_0}}{(1-2\lambda_0)^2}
\sum_{p=0}^q\theta^{2p}
\le
\theta^{2q}\frac{C_0^4}{(1-2\lambda_0)^2}\frac{\theta^4\theta^{-4\lambda_0}}{\theta^2-1}.
\end{align*}
Now, for each $n\ge 1$, we can find a
$q$ such that $\theta^q\le n<\theta^{q+1}$.  It follows that
\begin{align*}
&\ep\big[\max_{i\le n}\overline{II}_{i,k}\big]\le
\big\{\ep\big[\max_{i<
\theta^{q+1}}(\overline{II}_{i,k})^4\big]\big\}^{1/4}\le \frac{C_0\theta\theta^{-\lambda_0}}{(1-2\lambda_0)^{1/2}(\theta^2-1)^{1/4}}\sqrt{n}.
\end{align*}
When $\lambda_0\le -1$,  $\theta\le 5/4$,
$\theta^{-\lambda_0}=(1-1/(4\lambda_0))^{-\lambda_0}\le e^{1/4}$ and
$\theta^2-1\ge 1/(-2\lambda_0)$. When $-1< \lambda_0< 0$, $\theta=\sqrt{2}$. Hence
\begin{align*}
&\ep\big[\max_{i\le n}\overline{II}_{i,k}\big] \le \frac{2C_0}{(1-2\lambda_0)^{1/4}}\sqrt{n}.
\end{align*}

At last, we consider $I_{n}$. Write $a_n=\max_{m\le n}\|\bm N_m-m\widehat{\bm \rho}_m\|$. Recall \eqref{eqproofTh3.4},
\begin{align*}
I_{n}\le & \max_{|\lambda_0|\le m\le n}\left| \sum_{j=m}^n \frac{\|\bm N_{j-1}-(j-1)\widehat{\bm \rho}_{j-1}\|}{j-1}\big(\frac{n}{j}\big)^{\lambda_0}\Big(1+\frac{O(1)}{j}\Big)\right|\\
\le &\max_{|\lambda_0|\le m\le n}\left|a_n\int_{m-1}^n\big(\frac{n}{x}\big)^{\lambda_0}\frac{1}{x}dx +O(1)\int_{m-1}^n\big(\frac{n}{x}\big)^{\lambda_0}\frac{1}{x}dx\right|\\
\le & a_n/|\lambda_0|+O(1) \; a.s.
\end{align*}
It follows that
$$\big(N_{n,k}-n\widehat{\rho}_{n,k}\big)^+\le \delta a_n/|\lambda_0|+2(\max_{i\le
n}\overline{II}_{i,k}+\overline{III}_{n,k})+o(F(n))+o(n^{1/2}) \;\; a.s. $$
Now, define
$$ Z_n=Z_n(\lambda_0)=8\sum_{k=1}^K (\max_{i\le
n}\overline{II}_{i,k}+\overline{III}_{n,k}),$$
where
$$ \overline{II}_{i,k}=i^{\lambda_0}\max_{1\le m\le i}\big|
\sum_{l=1}^{m-1}l^{-\lambda_0}\Delta
M_{l,k}\big|, \;\;  \overline{III}_{n,k}= \max_{1\le m\le n}\Big|\int_0^m
\big(\frac{x}{m})^{-\lambda_0}d G_k(x)\Big|.   $$
Then
$$\|\bm N_n-n\widehat{\bm \rho}_n\|\le 2\sum_{k=1}^K (N_{n,k}-n\widehat{\rho}_{n,k})^+\le 2\delta K a_n /|\lambda_0|+ Z_n/2+o(n^{1/2})+o(F(n)) \;\; a.s.$$
Note that $a_n$, $Z_n$ and $F(n)$ are non-decreasing in $n$. Hence
$$a_n=\max_{i\le n}\|\bm N_i-i\widehat{\bm \rho}_i\| \le 2\delta K a_n /|\lambda_0|+ Z_n/2+o(n^{1/2})+o(F(n)) \;\; a.s.$$
By choosing $\delta>0$ such that  $2\delta K/|\lambda_0|<1/2$, we conclude that
\begin{align*}
&\max_{i\le n}\|\bm N_i-i\widehat{\bm \rho}_i\|=a_n\le Z_n +o(F(n))+o(n^{1/2}) \;\; a.s.,  \\
&\ep[Z_n]\le \sqrt{n}\left(\frac{16KC_0}{(1-2\lambda_0)^{1/4}}+\frac{32\sum_{k=1}^K\sqrt{\sigma_{kk}}}{\sqrt{1-2\lambda_0}}\right), \;\; \lambda_0>0,\\
&\ep[F(n)]=\ep[\sup_{s\le n}\|\bm W_s\|]+\int_0^n \frac{\ep[\|\bm W(x)\|]}{x}dx\le 6\sqrt{n}\sqrt{\sum_{k=1}^n \sigma_{kk}}.
\end{align*}
\eqref{eqproofTh3key} is proved. $\Box$

\bigskip
To close this section, we finally show that  if the function $\psi(\cdot)$ in
\begin{equation}\label{eq:Eprade2.1append}  g_{k}\big(m,\bm x ,\bm \rho\big)=g_{k}\big(\bm x ,\bm \rho\big)=\frac{\rho_k\psi\left(\frac{\rho_k}{x_k}\right)}{\sum_{j=1}^K\rho_j\psi\left(\frac{\rho_j}{x_j}\right)}, \; k=1,\ldots,K
\end{equation} is differentiable, then the design behaves the same as the DBCD.

\begin{theorem}  Let $\psi(x)\ge 0, x\ge 0$ be a non-decreasing function with $\psi(1)\ne 0$ and $\dot{\psi}(1)/\psi(1)=\gamma\ge 0$. Assume   that the conditions \ref{ConAappend} and \ref{ConBappend} are satisfied. Then
\begin{equation}\label{eq:likeDBCD}
 n^{1/2}\big(\frac{\bm N_n}{n}-\bm\rho(\bm\Theta)\big)\overset{\Cal D}\to
 N(\bm 0, \bm \Lambda_{\gamma})\;\;\text{ and }\;\;
 n^{1/2}\big(\widehat{\bm \rho}_n-\bm\rho(\bm\Theta)\big)\overset{\Cal D}\to
 N(\bm 0, \bm \Sigma),
 \end{equation}
 where
 $$\bm \Lambda_{\gamma}=\frac{1}{1+2\gamma}\bm\Sigma_1
 +\frac{2(1+\gamma)}{1+2\gamma}\bm \Sigma, \;\; \bm\Sigma_1=diag(\bm v)-\bm v^{\prime}\bm v. $$
 \end{theorem}
{\bf Proof.} Firstly, since $g_i(\bm x,\bm \rho)/\rho_i\le g_j(\bm x,\bm \rho)/\rho_j$ whenever $x_i/\rho_i\ge x_j/\rho_j$,  by Theorem \ref{Th1append} we still have $\bm N_n/n\to \bm v$ a.s. and $\widehat{\bm \rho}_n\to \bm v$ a.s.  Without loss of generality, we assume  $\psi(1)=1$. Then, as $\bm x\to \bm v$ and $\bm y\to \bm v$,
$$ y_k\psi(y_k/x_k)=y_k+\gamma y_k(y_k/x_k-1)+o(\| x_k- y_k\|)=y_k+\gamma (y_k-x_k)+o(\|\bm x-\bm y\|),  $$
$$\sum_{k=1}^Ky_k\psi(y_k/x_k)=1+o(\|\bm x-\bm y\|),  $$
\begin{equation} \label{eq:approxforgk} g_k(\bm x,\bm y)=y_k-\gamma (x_k-y_k)+o(\|\bm x-\bm y\|).
\end{equation}
Hence, Condition (C1) is satisfied with $\lambda_0=-\gamma\le 0$.  By Theorem \ref{Th2append},
$$\bm N_n-n\bm v=O(\sqrt{n\log\log n}) \; a.s. \text{ and } n(\widehat{\bm \rho}_n-\bm v)=O(\sqrt{n\log\log n}) \; a.s., $$
which implies
$$p_{m,k}=g_k\left(\frac{\bm N_m}{n},\widehat{\bm\rho}_m\right)=v_k+O\left(\sqrt{\frac{\log\log n}{n}}\right)\; a.s. $$
Then $  \Delta \bm M_n=\bm X_n-\bm p_{n-1}$ is a sequence of martingale differences with
$$ \ep[(\Delta \bm M_n)^{\prime}\Delta \bm M_n|\mathcal{F}_{n-1}]=diag(\bm p_{n-1})-(\bm p_{n-1})^{\prime}\bm p_{n-1}=\bm \Sigma_1+O\left(\sqrt{\frac{\log\log n}{n}}\right) a.s. $$
 On the other hand, it is obvious that
$$ \ep[\Delta M_{m,k}\Delta Q_{m,li}|\mathscr{F}_{m-1}]=0 \;\; \text{
for all }\;\; i \text{ and }\;\; k,l. $$
By Lemma \ref{lem3}, we can define a $(K\times d)$-dimensional
 standard Brownian motion $\bm B(t)$  and a $K$-dimensional standard Brownian motion $\bm b(t)$ such that they are independent and \eqref{eq:approxforQ},
\begin{equation}\label{eq:approxforM} \bm M_n =\bm b(n)
\bm \Sigma_1^{1/2}+o(n^{1/2-\tau})
\end{equation}
hold for some $\tau>0$. Then \eqref{eqThappend2.3} and \eqref{eqThappend2.4} hold. Write $a_n=\max_{i\le n}\|\bm N_i-i\widehat{\bm \rho}_i\|/i^{1/2-\alpha}$. By \eqref{eq:approxforgk}, we have
$$ X_{n,k}=\Delta M_{n,k}+ \widehat{\rho}_{n-1,k}-\gamma\left(\frac{N_{n-1,k}}{n-1}-\widehat{\rho}_{n-1,k}\right)+\frac{o(a_{n-1})}{(n-1)^{1/2+\alpha}}\; a.s. $$
Let $\bm G_t$ be a solution of the equation
\begin{equation}\label{eq:PDE} \bm G_t=\bm b(t)\bm\Sigma_1^{1/2}+(\gamma+1)\int_0^t\frac{\bm B(s)\bm\Sigma^{1/2}}{s}dx-\gamma\int_0^t\frac{\bm G_s}{s}ds,\;\; \bm G_0=0,
\end{equation}
 where $\bm\Sigma$ is defended as in \eqref{eqLBappend}.  Write $\widetilde{a}_m=\sup_{i\le m}\|\bm B(i)\|/i^{1/2-\alpha}$. Then
 \begin{align*}
  & \bm N_n-n\bm v=    \bm M_n-\gamma\sum_{m=1}^{n-1}\frac{\bm N_m-m\bm v}{m}+(\gamma+1)\sum_{m=1}^{n-1}\big(\widehat{\bm \rho}_m-\bm v)+\sum_{m=1}^{n-1}\frac{o(a_m)}{m^{1/2+\alpha}} \; a.s.\\
=& \bm b(n)\bm\Sigma_1^{1/2}+o(n^{1/2-\tau})-\gamma\sum_{m=1}^{n-1}\frac{\bm N_m-m\bm v}{m} \\
 & +(\gamma+1)\sum_{m=1}^{n-1}\frac{\bm B(m)\bm\Sigma^{1/2}+o(\|\bm B(m)\|)+o(m^{1/2})}{m}
   +\sum_{m=1}^{n-1}\frac{o(a_m)}{m^{1/2+\alpha}}\; a.s.\\
=& \bm G_n-\gamma\sum_{m=1}^{n-1}\frac{\bm N_m-m\bm v-\bm G_m}{m}+\sum_{m=1}^{n-1}\frac{o(a_m+\widetilde{a}_m)}{m^{1/2+\alpha}}+o(n^{1/2})\; a.s.
 \end{align*}
 That is
 $$
   \bm N_n-n\bm v-\bm G_n= -\gamma\sum_{m=1}^{n-1}\frac{\bm N_m-m\bm v-\bm G_m}{m}+ o(a_n+\widetilde{a}_n)n^{1/2-\alpha}+o(n^{1/2})\; a.s.
$$
By Lemma \ref{lem1},
 \begin{align*}
 \bm N_n-n\bm v-\bm G_n =& \sum_{m=1}^n \frac{o(a_m+\widetilde{a}_m)m^{1/2-\alpha}+o(m^{1/2})}{m}\big(\frac{n}{m}\big)^{-\gamma}\\
 =&   o(a_n+\widetilde{a}_n)n^{1/2-\alpha}+o(n^{1/2}) \;\; a.s.
 \end{align*}
 It follows that
 $$ \max_{m\le n}\|\bm N_m-m\bm v-\bm G_m\|=o(a_n+\widetilde{a}_n)n^{1/2-\alpha}+o(n^{1/2}) \;\; a.s.  $$
 By \eqref{eq:boundofB(m)} and $\eqref{eq:boundofN(m)}$, $a_n+\widetilde{a}_n=O_P(n^{\alpha})$.  We conclude that
 $$ \max_{m\le n}\|\bm N_n-n\bm v-\bm G_n\|=o_P(\sqrt{n}). $$
 Hence \eqref{eq:likeDBCD} is proved by noting that
 $$\bm G_t=t^{-\gamma}\int_0^t x^{\gamma} d \bm b(x)\bm\Sigma_1^{1/2}+t^{-\gamma}\int_0^t x^{\gamma-1} \bm B(x)\bm\Sigma^{1/2}dx $$
 is the solution of \eqref{eq:PDE} and is  a Gaussian process with
 $ Var(\bm G_t)=t \left(\frac{1}{1+2\gamma}\bm\Sigma_1+\frac{2(1+\gamma)}{1+2\gamma}\bm\Sigma\right)$. $\Box$

\section{Simulation study}

In this section, we conduct a numerical study to compare the performance of the efficient doubly
adaptive biased coin design (Procedure B, EDBCD)  proposed in this article.
For clinical trials with two treatments, the compared procedures to the EDBCD  is the Doubly
Adaptive Biased Coin Design (DBCD) by \cite{HZ04}, and the efficient randomized-adaptive designs (ERADE) by \cite{HZH09}. For clinical trials with three treatments, since the ERADE of Hu, Zhang and He (2009) is limited to two-treatment trials, we consider its generalization defined as in Example 3.1 and  refer it to GERADE. The compared procedures to the  EDBCD   is    DBCD and GERADE.   We consider the clinical trials with dichotomic responses $\xi_{m,k}\sim binomial(1,p_k)$. The estimator of $p_k$ is given by
$$\widehat{p}_{m,k}=\frac{\sum_{i=1}^m X_{i,k}\xi_{i,k}+0.5}{\sum_{i=1}^n   X_{i,k}+1}. $$
Throughout our simulations, we used the  the   optimal proportions $\sqrt{p_k}/\sum_j\sqrt{p_j}$ proposed by Rosenberger et al. (2001) as   target proportions, and simulate   the allocation proportion ($\rho_k$), the  variability  ($\sigma^2$) of the normalized allocation proportions $N_{n,k}/\sqrt{n}$, the proportion of corrected  guesses (SB), then entropy (Ent) of the allocations.  We utilized the  set of parameters similar to those  described by \cite{HR03} and report
our findings in Tables \ref{tab:optimal-2treatment-1}-\ref{tab:optimal-3treatment-2}.
For the two-treatment case, DBCD, ERADE and EDBCD all converged well to the target allocation
proportion, but  ERADE and   EDBCD performed  better than DBCD with much smaller variability. The degree of randomness of the EDBCD is between  those of DBCD and ERADE  in most circumstances.
For the three-treatment case, DBCD, GERADE and EDBCD performed very similarly, but   EDBCD and GERADE have smaller variability.  We also consider the urn proportions $(1/q_k)/\sum_j(1/q_j)$ as   target proportions, where $q_j=1-p_j$, and have similar findings reported in Tables \ref{tab:urn-2treatment-1}-\ref{tab:urn-3treatment-2}.

\begin{table}[ht]
\centering
\caption{Simulated allocation proportion ($N_{n,1}/n$), selection bias (SB) and entropy (Ent)  of three
designs  with allocation target $\rho_1=\sqrt{p_1}/(\sqrt{p_1}+\sqrt{p_2})$, sample size=200, and replication =1000, in a 2-treatment trial.}
\label{tab:optimal-2treatment-1}
\medskip
  \tiny
\begin{tabular}{|c|ccc |ccc|ccc|ccc|}
\hline
Response & \multicolumn{3}{c|}{Asymptotic} & \multicolumn{3}{c|}{DBCD $\gamma=2$} & \multicolumn{3}{c|}{ERADE $\alpha=2/3$}&  \multicolumn{3}{c|}{EDBCD $\gamma=2$}\\
  $p_1$, $p_2$  & $\rho_1(\sigma^2)$ & SB & Ent & $\rho_1(\sigma^2)$ & SB & Ent & $\rho_1(\sigma^2)$ & SB & Ent & $\rho_1(\sigma^2)$ & SB & Ent \\
  \hline
0.9,	0.9	&	0.50(0.01) 	&	0.50 	&	0.69 	&	0.50 	(0.06) &	0.54 	&	0.68 	&	0.50 	(0.01) &	0.60 	&	0.66 	&	0.50 	(0.02) &	0.58 	&	0.67 	\\
0.9,	0.7	&	0.53(0.02) 	&	0.53 	&	0.69 	&	0.53 	(0.07) &	0.56 	&	0.68 	&	0.53 	(0.02) &	0.61 	&	0.65 	&	0.53 	(0.04) &	0.59 	&	0.66 	\\
0.9,	0.6	&	0.55(0.03) 	&	0.55 	&	0.69 	&	0.55 	(0.08) &	0.57 	&	0.67 	&	0.55 	(0.03) &	0.62 	&	0.65 	&	0.55 	(0.04) &	0.59 	&	0.66 	\\
0.9,	0.5	&	0.57(0.04) 	&	0.57 	&	0.68 	&	0.57 	(0.10) &	0.59 	&	0.67 	&	0.57 	(0.05) &	0.63 	&	0.64 	&	0.57 	(0.06) &	0.60 	&	0.65 	\\
0.9,	0.3	&	0.63(0.09) 	&	0.63 	&	0.66 	&	0.63 	(0.17) &	0.64 	&	0.63 	&	0.63 	(0.12) &	0.67 	&	0.61 	&	0.63 	(0.12) &	0.65 	&	0.62 	\\
0.8,	0.8	&	0.50(0.02) 	&	0.50 	&	0.69 	&	0.50 	(0.07) &	0.55 	&	0.68 	&	0.50 	(0.02) &	0.60 	&	0.65 	&	0.50 	(0.04) &	0.58 	&	0.67 	\\
0.8,	0.7	&	0.52(0.02) 	&	0.52 	&	0.69 	&	0.51 	(0.08) &	0.55 	&	0.68 	&	0.51 	(0.03) &	0.61 	&	0.65 	&	0.51 	(0.04) &	0.58 	&	0.67 	\\
0.8,	0.6	&	0.54(0.03) 	&	0.54 	&	0.69 	&	0.53 	(0.08) &	0.56 	&	0.68 	&	0.53 	(0.04) &	0.61 	&	0.65 	&	0.53 	(0.05) &	0.59 	&	0.66 	\\
0.7,	0.5	&	0.54(0.05) 	&	0.54 	&	0.69 	&	0.54 	(0.11) &	0.57 	&	0.67 	&	0.54 	(0.05) &	0.62 	&	0.65 	&	0.54 	(0.07) &	0.59 	&	0.66 	\\
0.7,	0.3	&	0.60(0.09) 	&	0.60 	&	0.67 	&	0.60 	(0.18) &	0.61 	&	0.65 	&	0.60 	(0.12) &	0.65 	&	0.63 	&	0.60 	(0.13) &	0.63 	&	0.63 	\\
0.6,	0.4	&	0.55(0.07) 	&	0.55 	&	0.69 	&	0.55 	(0.13) &	0.57 	&	0.67 	&	0.55 	(0.08) &	0.63 	&	0.64 	&	0.55 	(0.10) &	0.60 	&	0.66 	\\
0.5,	0.5	&	0.50(0.06) 	&	0.50 	&	0.69 	&	0.50 	(0.12) &	0.56 	&	0.67 	&	0.50 	(0.07) &	0.62 	&	0.65 	&	0.50 	(0.09) &	0.59 	&	0.66 	\\
0.5,	0.2	&	0.61(0.17) 	&	0.61 	&	0.67 	&	0.61 	(0.32) &	0.63 	&	0.64 	&	0.61 	(0.25) &	0.67 	&	0.62 	&	0.61 	(0.23) &	0.64 	&	0.63 	\\
0.4,    0.3	&	0.54(0.12) 	&	0.54 	&	0.69    &   0.53 	(0.21) &	0.57 	&	0.67 	&	0.53 	(0.13) &	0.63 	&	0.64 	&	0.53 	(0.16) &	0.60 	&	0.65  	\\
0.2,	0.2	&	0.50(0.25) 	&	0.50 	&	0.69 	&	0.50 	(0.41) &	0.58 	&	0.66 	&	0.50 	(0.32) &	0.64 	&	0.64 	&	0.50 	(0.32) &	0.62 	&	0.64 	\\
\hline
Response & \multicolumn{3}{c|}{Asymptotic} & \multicolumn{3}{c|}{DBCD $\gamma=4$} & \multicolumn{3}{c|}{ERADE $\alpha=1/2$}&  \multicolumn{3}{c|}{EDBCD $\gamma=4$}\\
  $p_1$, $p_2$  & $\rho_1(\sigma^2)$ & SB & Ent & $\rho_1(\sigma^2)$ & SB & Ent & $\rho_1(\sigma^2)$ & SB & Ent & $\rho_1(\sigma^2)$ & SB & Ent \\
  \hline
0.9,	0.9	&	0.50(0.01) 	&	0.50 	&	0.69 	&	0.50 	(0.04) &	0.57 	&	0.67 	&	0.50 	(0.01) &	0.63 	&	0.62 	&	0.50 	(0.02) &	0.60 	&	0.65 	\\
0.9,	0.7	&	0.53(0.02) 	&	0.53 	&	0.69 	&	0.53 	(0.04) &	0.57 	&	0.67 	&	0.53 	(0.02) &	0.64 	&	0.62 	&	0.53 	(0.03) &	0.61 	&	0.65 	\\
0.9,	0.6	&	0.55(0.03) 	&	0.55 	&	0.69 	&	0.55 	(0.06) &	0.58 	&	0.66 	&	0.55 	(0.03) &	0.66 	&	0.61 	&	0.55 	(0.04) &	0.61 	&	0.65 	\\
0.9,	0.5	&	0.57(0.04) 	&	0.57 	&	0.68 	&	0.57 	(0.07) &	0.60 	&	0.65 	&	0.57 	(0.04) &	0.67 	&	0.61 	&	0.57 	(0.06) &	0.62 	&	0.64 	\\
0.9,	0.3	&	0.63(0.09) 	&	0.63 	&	0.66 	&	0.63 	(0.14) &	0.65 	&	0.62 	&	0.63 	(0.10) &	0.71 	&	0.57 	&	0.63 	(0.12) &	0.66 	&	0.61 	\\
0.8,	0.8	&	0.50(0.02) 	&	0.50 	&	0.69 	&	0.50 	(0.04) &	0.57 	&	0.67 	&	0.50 	(0.02) &	0.63 	&	0.62 	&	0.50 	(0.03) &	0.60 	&	0.65 	\\
0.8,	0.7	&	0.52(0.02) 	&	0.52 	&	0.69 	&	0.51 	(0.06) &	0.57 	&	0.67 	&	0.51 	(0.02) &	0.64 	&	0.62 	&	0.51 	(0.03) &	0.60 	&	0.65 	\\
0.8,	0.6	&	0.54(0.03) 	&	0.54 	&	0.69 	&	0.53 	(0.06) &	0.58 	&	0.66 	&	0.53 	(0.03) &	0.65 	&	0.62 	&	0.53 	(0.04) &	0.61 	&	0.65 	\\
0.7,	0.5	&	0.54(0.05) 	&	0.54 	&	0.69 	&	0.54 	(0.08) &	0.58 	&	0.66 	&	0.54 	(0.05) &	0.65 	&	0.61 	&	0.54 	(0.06) &	0.62 	&	0.64 	\\
0.7,	0.3	&	0.60(0.09) 	&	0.60 	&	0.67 	&	0.60 	(0.14) &	0.63 	&	0.63 	&	0.60 	(0.12) &	0.70 	&	0.59 	&	0.60 	(0.13) &	0.65 	&	0.62 	\\
0.6,	0.4	&	0.55(0.07) 	&	0.55 	&	0.69 	&	0.55 	(0.11) &	0.59 	&	0.66 	&	0.55 	(0.07) &	0.66 	&	0.61 	&	0.55 	(0.09) &	0.62 	&	0.64 	\\
0.5,	0.5	&	0.50(0.06) 	&	0.50 	&	0.69 	&	0.50 	(0.10) &	0.58 	&	0.66 	&	0.50 	(0.07) &	0.65 	&	0.61 	&	0.50 	(0.08) &	0.61 	&	0.65 	\\
0.5,	0.2	&	0.61(0.17) 	&	0.61 	&	0.67 	&	0.61 	(0.23) &	0.64 	&	0.62 	&	0.61 	(0.22) &	0.71 	&	0.58 	&	0.61 	(0.23) &	0.66 	&	0.60 	\\
0.4,	0.3	&	0.54(0.12) 	&	0.54 	&	0.69 	&	0.53 	(0.16) &	0.59 	&	0.65 	&	0.53 	(0.13) &	0.67 	&	0.60 	&	0.53 	(0.15) &	0.63 	&	0.63 	\\
0.2,	0.2	&	0.50(0.25) 	&	0.50 	&	0.69 	&	0.50 	(0.37) &	0.61 	&	0.64 	&	0.50 	(0.35) &	0.69 	&	0.59 	&	0.50 	(0.33) &	0.64 	&	0.62 	\\
\hline
\end{tabular}
\end{table}

\begin{table}[ht]
\centering
\caption{Simulated allocation proportion ($N_{n,1}/n$), selection bias (SB) and entropy (Ent)     of three
designs  with allocation targets $\rho_k=\sqrt{p_k}/\sum_j\sqrt{p_j}$, sample size=200, and replication =1000, in a 3-treatment trial.}
\label{tab:optimal-3treatment-1}
\medskip

\tiny

\begin{tabular}{| c | ccccc | ccccc |}
\hline
Responses & \multicolumn{5}{c|}{Asymptotic} & \multicolumn{5}{c|}{DBCD $\gamma=2$} \\
  $p_1$, $p_2$, $p_3$  & $\rho_1(\sigma^2)$ & $\rho_2(\sigma^2)$ & $\rho_3(\sigma^2)$ & SB & Ent & $\rho_1(\sigma^2)$ & $\rho_2(\sigma^2)$ & $\rho_3(\sigma^2)$ & SB & Ent   \\
   \hline
0.9,	0.9,	0.9	&	0.33 	(0.01) &	0.33 	(0.01) &	0.33 	(0.01) &	0.33 	&	1.10 	&	0.33 	(0.05) &	0.33 	(0.05) &	0.33 	(0.05) &	0.39 	&	1.07 	 	\\
0.9,	0.9,	0.7	&	0.35 	(0.01) &	0.35 	(0.01) &	0.31 	(0.02) &	0.35 	&	1.10 	&	0.35 	(0.06) &	0.35 	(0.05) &	0.30 	(0.06) &	0.40 	&	1.07 		\\
0.9,	0.8,	0.6	&	0.36 	(0.01) &	0.34 	(0.02) &	0.30 	(0.03) &	0.36 	&	1.10 	&	0.36 	(0.06) &	0.34 	(0.07) &	0.29 	(0.08) &	0.40 	&	1.07 		\\
0.9,	0.7,	0.5	&	0.38 	(0.02) &	0.34 	(0.03) &	0.28 	(0.04) &	0.38 	&	1.09 	&	0.38 	(0.07) &	0.33 	(0.08) &	0.28 	(0.10) &	0.41 	&	1.06 		\\
0.9,	0.5,	0.3	&	0.43 	(0.05) &	0.32 	(0.05) &	0.25 	(0.09) &	0.43 	&	1.07 	&	0.43 	(0.11) &	0.32 	(0.12) &	0.25 	(0.16) &	0.45 	&	1.04 		\\
0.8,	0.8,	0.8	&	0.33 	(0.01) &	0.33 	(0.01) &	0.33 	(0.01) &	0.33 	&	1.10 	&	0.33 	(0.07) &	0.33 	(0.06) &	0.33 	(0.06) &	0.39 	&	1.07 		\\
0.8,	0.7,	0.6	&	0.36 	(0.02) &	0.33 	(0.02) &	0.31 	(0.03) &	0.36 	&	1.10 	&	0.36 	(0.07) &	0.33 	(0.08) &	0.31 	(0.08) &	0.40 	&	1.07 		\\
0.8,	0.6,	0.5	&	0.38 	(0.03) &	0.33 	(0.04) &	0.30 	(0.04) &	0.38 	&	1.09 	&	0.37 	(0.09) &	0.32 	(0.09) &	0.30 	(0.11) &	0.41 	&	1.06 		\\
0.7,	0.5,	0.4	&	0.38 	(0.04) &	0.32 	(0.05) &	0.29 	(0.07) &	0.38 	&	1.09 	&	0.38 	(0.11) &	0.32 	(0.11) &	0.29 	(0.12) &	0.42 	&	1.06 		\\
0.6,	0.5,	0.5	&	0.35 	(0.04) &	0.32 	(0.05) &	0.32 	(0.05) &	0.35 	&	1.10 	&	0.35 	(0.10) &	0.32 	(0.11) &	0.32 	(0.11) &	0.41 	&	1.06 		\\
0.6,	0.5,	0.3	&	0.38 	(0.06) &	0.35 	(0.06) &	0.27 	(0.09) &	0.38 	&	1.09 	&	0.38 	(0.12) &	0.35 	(0.14) &	0.27 	(0.17) &	0.43 	&	1.05 		\\
0.5,	0.5,	0.3	&	0.36 	(0.07) &	0.36 	(0.07) &	0.28 	(0.10) &	0.36 	&	1.09 	&	0.36 	(0.14) &	0.36 	(0.14) &	0.28 	(0.19) &	0.42 	&	1.05 		\\
0.4,	0.3,	0.2	&	0.39 	(0.13) &	0.34 	(0.13) &	0.27 	(0.17) &	0.39 	&	1.09 	&	0.39 	(0.22) &	0.33 	(0.21) &	0.27 	(0.28) &	0.44 	&	1.04 		\\
0.2,	0.2,	0.2	&	0.33 	(0.22) &	0.33 	(0.22) &	0.33 	(0.22) &	0.33 	&	1.10 	&	0.33 	(0.39) &	0.33 	(0.40) &	0.33 	(0.40) &	0.44 	&	1.04    	\\
\hline
 Responses &  \multicolumn{5}{c|}{GERADE $\alpha=2/3$} &  \multicolumn{5}{c|}{EDBCD $\gamma=2$}   \\
  $p_1$, $p_2$, $p_3$  & $\rho_1(\sigma^2)$ & $\rho_2(\sigma^2)$ & $\rho_3(\sigma^2)$ & SB & Ent & $\rho_1(\sigma^2)$ & $\rho_2(\sigma^2)$ & $\rho_3(\sigma^2)$ & SB & Ent  \\
   \hline
0.9,	0.9,	0.9	&	0.33 	(0.02) &	0.33 	(0.02) &	0.33 	(0.02) &	0.43 	&	1.05 	&	0.33 	(0.03) &	0.33 	(0.03) &	0.33 	(0.03) &	0.42 	&	1.06  	 	\\
0.9,	0.9,	0.7	&	0.34 	(0.02) &	0.35 	(0.02) &	0.31 	(0.03) &	0.44 	&	1.04 	&	0.35 	(0.03) &	0.35 	(0.03) &	0.30 	(0.04) &	0.43 	&	1.06  	 	\\
0.9,	0.8,	0.6	&	0.36 	(0.02) &	0.34 	(0.03) &	0.30 	(0.04) &	0.45 	&	1.04 	&	0.36 	(0.04) &	0.34 	(0.04) &	0.29 	(0.05) &	0.43 	&	1.05  	 	\\
0.9,	0.7,	0.5	&	0.38 	(0.03) &	0.33 	(0.04) &	0.28 	(0.06) &	0.46 	&	1.04 	&	0.38 	(0.04) &	0.33 	(0.05) &	0.28 	(0.07) &	0.44 	&	1.05  	 	\\
0.9,	0.5,	0.3	&	0.43 	(0.07) &	0.32 	(0.07) &	0.25 	(0.12) &	0.49 	&	1.01 	&	0.43 	(0.08) &	0.32 	(0.09) &	0.25 	(0.13) &	0.47 	&	1.02  	 	\\
0.8,	0.8,	0.8	&	0.33 	(0.03) &	0.33 	(0.02) &	0.33 	(0.02) &	0.43 	&	1.05 	&	0.33 	(0.03) &	0.33 	(0.03) &	0.33 	(0.04) &	0.42 	&	1.06  	 	\\
0.8,	0.7,	0.6	&	0.35 	(0.03) &	0.33 	(0.04) &	0.31 	(0.05) &	0.44 	&	1.04 	&	0.36 	(0.04) &	0.33 	(0.05) &	0.31 	(0.05) &	0.43 	&	1.05  	 	\\
0.8,	0.6,	0.5	&	0.37 	(0.04) &	0.32 	(0.05) &	0.30 	(0.06) &	0.45 	&	1.04 	&	0.37 	(0.05) &	0.32 	(0.05) &	0.30 	(0.07) &	0.44 	&	1.05  	 	\\
0.7,	0.5,	0.4	&	0.38 	(0.06) &	0.32 	(0.07) &	0.29 	(0.08) &	0.46 	&	1.03 	&	0.38 	(0.07) &	0.32 	(0.08) &	0.29 	(0.11) &	0.44 	&	1.04  	 	\\
0.6,	0.5,	0.5	&	0.35 	(0.06) &	0.32 	(0.06) &	0.32 	(0.07) &	0.45 	&	1.04 	&	0.35 	(0.07) &	0.32 	(0.07) &	0.32 	(0.08) &	0.44 	&	1.05  	 	\\
0.6,	0.5,	0.3	&	0.38 	(0.08) &	0.35 	(0.09) &	0.27 	(0.14) &	0.47 	&	1.03 	&	0.38 	(0.10) &	0.35 	(0.09) &	0.27 	(0.15) &	0.45 	&	1.03  	 	\\
0.5,	0.5,	0.3	&	0.36 	(0.09) &	0.36 	(0.10) &	0.28 	(0.13) &	0.46 	&	1.03 	&	0.36 	(0.11) &	0.36 	(0.10) &	0.28 	(0.15) &	0.45 	&	1.03  	 	\\
0.4,	0.3,	0.2	&	0.39 	(0.15) &	0.34 	(0.16) &	0.27 	(0.22) &	0.48 	&	1.02 	&	0.39 	(0.17) &	0.34 	(0.20) &	0.27 	(0.28) &	0.46 	&	1.02  	 	\\
0.2,	0.2,	0.2	&	0.33 	(0.33) &	0.33 	(0.33) &	0.33 	(0.35) &	0.48 	&	1.02 	&	0.33 	(0.37) &	0.33 	(0.36) &	0.33 	(0.35) &	0.47 	&	1.02 	 	\\
\hline
\end{tabular}
\end{table}

\begin{table}[ht]
\centering
\caption{Simulated allocation proportion ($N_{n,1}/n$), selection bias (SB) and entropy (Ent)     of three
designs  with allocation targets $\rho_k=\sqrt{p_k}/\sum_j\sqrt{p_j}$, sample size=200, and replication =1000, in a 3-treatment trial.}
\label{tab:optimal-3treatment-2}
\medskip

\tiny

\begin{tabular}{| c |ccccc|ccccc|}
\hline
Responses & \multicolumn{5}{c|}{Asymptotic} & \multicolumn{5}{c|}{DBCD $\gamma=4$} \\
  $p_1$, $p_2$, $p_3$  & $\rho_1(\sigma^2)$ & $\rho_2(\sigma^2)$ & $\rho_3(\sigma^2)$ & SB & Ent & $\rho_1(\sigma^2)$ & $\rho_2(\sigma^2)$ & $\rho_3(\sigma^2)$ & SB & Ent   \\
   \hline
0.9,	0.9,	0.9	&	0.33 	(0.01) &	0.33 	(0.01) &	0.33 	(0.01) &	0.33 	&	1.10 	&	0.33 	(0.03) &	0.33 	(0.03) &	0.33 	(0.03) &	0.42 	&	1.05  	\\
0.9,	0.9,	0.7	&	0.35 	(0.01) &	0.35 	(0.01) &	0.31 	(0.02) &	0.35 	&	1.10 	&	0.35 	(0.03) &	0.35 	(0.04) &	0.30 	(0.04) &	0.43 	&	1.05 	\\
0.9,	0.8,	0.6	&	0.36 	(0.01) &	0.34 	(0.02) &	0.30 	(0.03) &	0.36 	&	1.10 	&	0.36 	(0.04) &	0.34 	(0.04) &	0.29 	(0.05) &	0.43 	&	1.05 	\\
0.9,	0.7,	0.5	&	0.38 	(0.02) &	0.34 	(0.03) &	0.28 	(0.04) &	0.38 	&	1.09 	&	0.38 	(0.05) &	0.33 	(0.05) &	0.28 	(0.08) &	0.44 	&	1.04 	\\
0.9,	0.5,	0.3	&	0.43 	(0.05) &	0.32 	(0.05) &	0.25 	(0.09) &	0.43 	&	1.07 	&	0.43 	(0.09) &	0.32 	(0.09) &	0.25 	(0.14) &	0.47 	&	1.01 	\\
0.8,	0.8,	0.8	&	0.33 	(0.01) &	0.33 	(0.01) &	0.33 	(0.01) &	0.33 	&	1.10 	&	0.33 	(0.04) &	0.33 	(0.04) &	0.33 	(0.04) &	0.42 	&	1.05 	\\
0.8,	0.7,	0.6	&	0.36 	(0.02) &	0.33 	(0.02) &	0.31 	(0.03) &	0.36 	&	1.10 	&	0.36 	(0.05) &	0.33 	(0.05) &	0.31 	(0.06) &	0.43 	&	1.05 	\\
0.8,	0.6,	0.5	&	0.38 	(0.03) &	0.33 	(0.04) &	0.30 	(0.04) &	0.38 	&	1.09 	&	0.37 	(0.05) &	0.32 	(0.07) &	0.30 	(0.07) &	0.44 	&	1.04 	\\
0.7,	0.5,	0.4	&	0.38 	(0.04) &	0.32 	(0.05) &	0.29 	(0.07) &	0.38 	&	1.09 	&	0.38 	(0.08) &	0.32 	(0.09) &	0.29 	(0.10) &	0.45 	&	1.03 	\\
0.6,	0.5,	0.5	&	0.35 	(0.04) &	0.32 	(0.05) &	0.32 	(0.05) &	0.35 	&	1.10 	&	0.35 	(0.08) &	0.32 	(0.08) &	0.32 	(0.08) &	0.44 	&	1.04 	\\
0.6,	0.5,	0.3	&	0.38 	(0.06) &	0.35 	(0.06) &	0.27 	(0.09) &	0.38 	&	1.09 	&	0.38 	(0.11) &	0.35 	(0.10) &	0.27 	(0.15) &	0.45 	&	1.03 	\\
0.5,	0.5,	0.3	&	0.36 	(0.07) &	0.36 	(0.07) &	0.28 	(0.10) &	0.36 	&	1.09 	&	0.36 	(0.12) &	0.36 	(0.11) &	0.28 	(0.14) &	0.45 	&	1.03 	\\
0.4,	0.3,	0.2	&	0.39 	(0.13) &	0.34 	(0.13) &	0.27 	(0.17) &	0.39 	&	1.09 	&	0.39 	(0.20) &	0.34 	(0.20) &	0.27 	(0.27) &	0.47 	&	1.01 	\\
0.2,	0.2,	0.2	&	0.33 	(0.22) &	0.33 	(0.22) &	0.33 	(0.22) &	0.33 	&	1.10 	&	0.33 	(0.36) &	0.33 	(0.36) &	0.33 	(0.36) &	0.48 	&	1.00   	\\
\hline
Responses  &  \multicolumn{5}{c|}{GERADE $\alpha=1/2$} &  \multicolumn{5}{c|}{EDBCD $\gamma=4$}   \\
  $p_1$, $p_2$, $p_3$  & $\rho_1(\sigma^2)$ & $\rho_2(\sigma^2)$ & $\rho_3(\sigma^2)$ & SB & Ent & $\rho_1(\sigma^2)$ & $\rho_2(\sigma^2)$ & $\rho_3(\sigma^2)$ & SB & Ent  \\
   \hline
0.9,	0.9,	0.9	&	0.33 	(0.01)&	0.33 	(0.01) &	0.33 	(0.01) &	0.46 	&	1.00 	&	0.33 	(0.02) &	0.33 	(0.02) &	0.33 	(0.02) &	0.45 	&	1.03   	\\
0.9,	0.9,	0.7	&	0.35 	(0.02)&	0.35 	(0.01) &	0.30 	(0.03) &	0.47 	&	0.99 	&	0.34 	(0.02) &	0.35 	(0.02) &	0.31 	(0.03) &	0.45 	&	1.03   	\\
0.9,	0.8,	0.6	&	0.36 	(0.02)&	0.34 	(0.02) &	0.29 	(0.04) &	0.48 	&	0.99 	&	0.36 	(0.03) &	0.34 	(0.03) &	0.29 	(0.04) &	0.46 	&	1.03   	\\
0.9,	0.7,	0.5	&	0.38 	(0.02)&	0.33 	(0.03) &	0.28 	(0.05) &	0.49 	&	0.98 	&	0.38 	(0.03) &	0.33 	(0.04) &	0.28 	(0.05) &	0.46 	&	1.02   	\\
0.9,	0.5,	0.3	&	0.43 	(0.06)&	0.32 	(0.06) &	0.25 	(0.12) &	0.53 	&	0.96 	&	0.43 	(0.07) &	0.32 	(0.07) &	0.25 	(0.13) &	0.49 	&	0.99   	\\
0.8,	0.8,	0.8	&	0.33 	(0.02)&	0.33 	(0.02) &	0.33 	(0.02) &	0.46 	&	1.00 	&	0.33 	(0.03) &	0.33 	(0.03) &	0.33 	(0.03) &	0.45 	&	1.03   	\\
0.8,	0.7,	0.6	&	0.36 	(0.03)&	0.33 	(0.03) &	0.31 	(0.04) &	0.48 	&	0.99 	&	0.36 	(0.03) &	0.33 	(0.04) &	0.31 	(0.04) &	0.46 	&	1.03   	\\
0.8,	0.6,	0.5	&	0.37 	(0.03)&	0.32 	(0.04) &	0.30 	(0.05) &	0.49 	&	0.99 	&	0.38 	(0.04) &	0.32 	(0.05) &	0.30 	(0.06) &	0.46 	&	1.02   	\\
0.7,	0.5,	0.4	&	0.38 	(0.05)&	0.32 	(0.06) &	0.29 	(0.08) &	0.50 	&	0.98 	&	0.38 	(0.06) &	0.32 	(0.07) &	0.29 	(0.09) &	0.47 	&	1.01   	\\
0.6,	0.5,	0.5	&	0.35 	(0.06)&	0.32 	(0.06) &	0.32 	(0.07) &	0.48 	&	0.99 	&	0.35 	(0.06) &	0.32 	(0.07) &	0.32 	(0.07) &	0.46 	&	1.02   	\\
0.6,	0.5,	0.3	&	0.38 	(0.07)&	0.35 	(0.08) &	0.27 	(0.14) &	0.51 	&	0.97 	&	0.38 	(0.08) &	0.35 	(0.09) &	0.27 	(0.15) &	0.48 	&	1.00   	\\
0.5,	0.5,	0.3	&	0.36 	(0.09)&	0.36 	(0.07) &	0.28 	(0.12) &	0.50 	&	0.97 	&	0.36 	(0.09) &	0.36 	(0.10) &	0.28 	(0.13) &	0.48 	&	1.01   	\\
0.4,	0.3,	0.2	&	0.39 	(0.16)&	0.34 	(0.18) &	0.27 	(0.27) &	0.52 	&	0.96 	&	0.39 	(0.16) &	0.34 	(0.18) &	0.27 	(0.23) &	0.49 	&	0.99   	\\
0.2,	0.2,	0.2	&	0.33 	(0.31)&	0.33 	(0.30) &	0.33 	(0.31) &	0.53 	&	0.95 	&	0.33 	(0.32) &	0.33 	(0.35) &	0.33 	(0.33) &	0.50 	&	0.98 	\\
\hline
\end{tabular}
\end{table}

\begin{table}[ht]
\centering
\caption{Simulated allocation proportion ($N_{n,1}/n$), selection bias (SB) and entropy (Ent)  of three
designs  with allocation target $\rho_1=q_2/(q_1+q_2)$, sample size=200, and replication =1000, in a 2-treatment trial.}
\label{tab:urn-2treatment-1}
\medskip
\tiny
\begin{tabular}{|c|ccc |ccc|ccc|ccc|}
\hline
Response & \multicolumn{3}{c|}{Asymptotic} & \multicolumn{3}{c|}{DBCD $\gamma=2$} & \multicolumn{3}{c|}{ERADE $\alpha=2/3$}&  \multicolumn{3}{c|}{EDBCD $\gamma=2$}\\
  $p_1$, $p_2$  & $\rho_1(\sigma^2)$ & SB & Ent & $\rho_1(\sigma^2)$ & SB & Ent & $\rho_1(\sigma^2)$ & SB & Ent & $\rho_1(\sigma^2)$ & SB & Ent \\
  \hline
0.9,	0.9	&	0.50(2.25) 	&	0.50 	&	0.69 	&	0.50 	(2.70) &	0.67 	&	0.58 	&	0.50 	(2.07) &	0.68 	&	0.59 	&	0.49 	(2.29) &	0.70 	&	0.56 	\\
0.9,	0.7	&	0.75(0.75) 	&	0.75 	&	0.56 	&	0.74 	(0.91) &	0.76 	&	0.50 	&	0.74 	(0.84) &	0.76 	&	0.50 	&	0.74 	(0.88) &	0.76 	&	0.49 	\\
0.9,	0.6	&	0.80(0.48) 	&	0.80 	&	0.50 	&	0.79 	(0.63) &	0.80 	&	0.46 	&	0.78 	(0.53) &	0.79 	&	0.46 	&	0.79 	(0.56) &	0.80 	&	0.44 	\\
0.9,	0.5	&	0.83(0.32) 	&	0.83 	&	0.45 	&	0.82 	(0.39) &	0.83 	&	0.42 	&	0.82 	(0.35) &	0.82 	&	0.42 	&	0.82 	(0.35) &	0.83 	&	0.40 	\\
0.9,	0.3	&	0.88(0.16) 	&	0.88 	&	0.38 	&	0.86 	(0.22) &	0.86 	&	0.36 	&	0.86 	(0.20) &	0.86 	&	0.36 	&	0.86 	(0.19) &	0.87 	&	0.35 	\\
0.8,	0.8	&	0.50(1.00) 	&	0.50 	&	0.69 	&	0.49 	(1.26) &	0.63 	&	0.62 	&	0.50 	(1.10) &	0.66 	&	0.62 	&	0.50 	(0.97) &	0.65 	&	0.61 	\\
0.8,	0.7	&	0.60(0.72) 	&	0.60 	&	0.67 	&	0.60 	(0.73) &	0.64 	&	0.62 	&	0.59 	(0.79) &	0.67 	&	0.61 	&	0.59 	(0.77) &	0.66 	&	0.60 	\\
0.8,	0.6	&	0.67(0.52) 	&	0.67 	&	0.64 	&	0.66 	(0.63) &	0.68 	&	0.59 	&	0.66 	(0.54) &	0.70 	&	0.58 	&	0.66 	(0.57) &	0.69 	&	0.57 	\\
0.7,	0.5	&	0.63(0.35) 	&	0.63 	&	0.66 	&	0.62 	(0.45) &	0.64 	&	0.63 	&	0.62 	(0.36) &	0.67 	&	0.61 	&	0.62 	(0.38) &	0.66 	&	0.61 	\\
0.7,	0.3	&	0.70(0.21) 	&	0.70 	&	0.61 	&	0.69 	(0.27) &	0.70 	&	0.59 	&	0.69 	(0.24) &	0.71 	&	0.56 	&	0.69 	(0.24) &	0.71 	&	0.57 	\\
0.6,	0.4	&	0.60(0.24) 	&	0.60 	&	0.67 	&	0.59 	(0.31) &	0.62 	&	0.64 	&	0.60 	(0.25) &	0.66 	&	0.62 	&	0.59 	(0.26) &	0.63 	&	0.63 	\\
0.5,	0.5	&	0.50(0.25) 	&	0.50 	&	0.69 	&	0.50 	(0.35) &	0.58 	&	0.66 	&	0.50 	(0.27) &	0.64 	&	0.64 	&	0.50 	(0.28) &	0.61 	&	0.65 	\\
0.5,	0.2	&	0.62(0.13) 	&	0.62 	&	0.67 	&	0.61 	(0.20) &	0.63 	&	0.64 	&	0.61 	(0.14) &	0.67 	&	0.62 	&	0.61 	(0.15) &	0.64 	&	0.63 	\\
0.4,	0.3	&	0.54(0.13) 	&	0.54 	&	0.69 	&	0.54 	(0.21) &	0.57 	&	0.67 	&	0.54 	(0.14) &	0.64 	&	0.64 	&	0.53 	(0.16) &	0.60 	&	0.65 	\\
0.2,	0.2	&	0.50(0.06) 	&	0.50 	&	0.69 	&	0.50 	(0.12) &	0.55 	&	0.67 	&	0.50 	(0.07) &	0.63 	&	0.64 	&	0.50 	(0.09) &	0.59 	&	0.66 	\\
\hline
Response & \multicolumn{3}{c|}{Asymptotic} & \multicolumn{3}{c|}{DBCD $\gamma=4$} & \multicolumn{3}{c|}{ERADE $\alpha=1/2$}&  \multicolumn{3}{c|}{EDBCD $\gamma=4$}\\
  $p_1$, $p_2$  & $\rho_1(\sigma^2)$ & SB & Ent & $\rho_1(\sigma^2)$ & SB & Ent & $\rho_1(\sigma^2)$ & SB & Ent & $\rho_1(\sigma^2)$ & SB & Ent \\
  \hline
0.9, 	0.9 	&	0.50(2.25) 	&	0.50 	&	0.69 	&	0.50 	(2.46) &	0.71 	&	0.54 	&	0.50 	(2.25) &	0.72 	&	0.55 	&	0.50 	(2.17) &	0.72 	&	0.51 	\\
0.9, 	0.7 	&	0.75(0.75) 	&	0.75 	&	0.56 	&	0.74 	(0.83) &	0.77 	&	0.48 	&	0.74 	(0.76) &	0.79 	&	0.47 	&	0.74 	(0.77) &	0.78 	&	0.46 \\
0.9, 	0.6 	&	0.80(0.48) 	&	0.80 	&	0.50 	&	0.79 	(0.54) &	0.81 	&	0.43 	&	0.79 	(0.48) &	0.82 	&	0.42 	&	0.79 	(0.50) &	0.81 	&	0.42 \\
0.9, 	0.5 	&	0.83(0.32) 	&	0.83 	&	0.45 	&	0.82 	(0.36) &	0.83 	&	0.40 	&	0.82 	(0.34) &	0.85 	&	0.38 	&	0.82 	(0.37) &	0.83 	&	0.38 \\
0.9, 	0.3 	&	0.88(0.16) 	&	0.88 	&	0.38 	&	0.87 	(0.19) &	0.87 	&	0.33 	&	0.87 	(0.16) &	0.88 	&	0.32 	&	0.86 	(0.17) &	0.87 	&	0.33 \\
0.8, 	0.8 	&	0.50(1.00) 	&	0.50 	&	0.69 	&	0.49 	(1.18) &	0.66 	&	0.59 	&	0.50 	(0.96) &	0.70 	&	0.57 	&	0.50 	(0.98) &	0.68 	&	0.57 \\
0.8, 	0.7 	&	0.60(0.72) 	&	0.60 	&	0.67 	&	0.60 	(0.76) &	0.67 	&	0.59 	&	0.59 	(0.73) &	0.71 	&	0.56 	&	0.60 	(0.74) &	0.68 	&	0.57 \\
0.8, 	0.6 	&	0.67(0.52) 	&	0.67 	&	0.64 	&	0.66 	(0.61) &	0.69 	&	0.57 	&	0.66 	(0.58) &	0.74 	&	0.54 	&	0.66 	(0.56) &	0.71 	&	0.55 \\
0.7, 	0.5 	&	0.63(0.35) 	&	0.63 	&	0.66 	&	0.62 	(0.40) &	0.66 	&	0.61 	&	0.62 	(0.35) &	0.72 	&	0.56 	&	0.62 	(0.37) &	0.67 	&	0.59 \\
0.7, 	0.3 	&	0.70(0.21) 	&	0.70 	&	0.61 	&	0.69 	(0.22) &	0.71 	&	0.57 	&	0.69 	(0.22) &	0.75 	&	0.52 	&	0.69 	(0.23) &	0.72 	&	0.55 \\
0.6, 	0.4 	&	0.60(0.24) 	&	0.60 	&	0.67 	&	0.60 	(0.28) &	0.63 	&	0.63 	&	0.60 	(0.24) &	0.71 	&	0.58 	&	0.59 	(0.27) &	0.65 	&	0.61 \\
0.5, 	0.5 	&	0.50(0.25) 	&	0.50 	&	0.69 	&	0.50 	(0.28) &	0.60 	&	0.64 	&	0.50 	(0.25) &	0.68 	&	0.59 	&	0.50 	(0.27) &	0.63 	&	0.63 \\
0.5, 	0.2 	&	0.62(0.13) 	&	0.62 	&	0.67 	&	0.61 	(0.16) &	0.64 	&	0.63 	&	0.61 	(0.13) &	0.71 	&	0.57 	&	0.61 	(0.14) &	0.65 	&	0.61 \\
0.4, 	0.3 	&	0.54(0.13) 	&	0.54 	&	0.69 	&	0.54 	(0.17) &	0.59 	&	0.65 	&	0.54 	(0.13) &	0.68 	&	0.60 	&	0.54 	(0.15) &	0.63 	&	0.63 \\
0.2, 	0.2 	&	0.50(0.06) 	&	0.50 	&	0.69 	&	0.50 	(0.10) &	0.57 	&	0.66 	&	0.50 	(0.07) &	0.67 	&	0.60 	&	0.50 	(0.08) &	0.61 	&	0.64 \\
\hline
\end{tabular}
\end{table}

\begin{table}[ht]
\centering
\caption{Simulated allocation proportion ($N_{n,1}/n$), selection bias (SB) and  entropy (Ent)  of three
designs  with allocation targets $\rho_k=\frac{1/q_k}{\sum_j(1/q_j)}$, sample size=200, and replication =1000, in a 3-treatment trial}
\label{tab:urn-3treatment-1}
\medskip

\tiny

\begin{tabular}{| c | ccccc | ccccc |}
\hline
Responses & \multicolumn{5}{c|}{Asymptotic} & \multicolumn{5}{c|}{DBCD $\gamma=2$} \\
  $p_1$, $p_2$, $p_3$  & $\rho_1(\sigma^2)$ & $\rho_2(\sigma^2)$ & $\rho_3(\sigma^2)$ & SB & Ent & $\rho_1(\sigma^2)$ & $\rho_2(\sigma^2)$ & $\rho_3(\sigma^2)$ & SB & Ent   \\
   \hline
0.9,	0.9,	0.9	&	0.33 	(2.00) &	0.33 	(2.00) &	0.33 	(2.00 ) &	0.33 	&	1.10 	&	0.33 	(2.38) &	0.33 	(2.27) &	0.33 	(2.24) &	0.56 	&	0.90 	\\
0.9,	0.9,	0.7	&	0.43 	(2.03) &	0.43 	(2.03) &	0.14 	(0.40 ) &	0.43 	&	1.00 	&	0.42 	(2.34) &	0.43 	(2.33) &	0.15 	(0.45) &	0.59 	&	0.85	\\
0.9,	0.8,	0.6	&	0.57 	(1.39) &	0.29 	(1.02) &	0.14 	(0.29 ) &	0.57 	&	0.96 	&	0.55 	(1.64) &	0.29 	(1.17) &	0.15 	(0.37) &	0.61 	&	0.85	\\
0.9,	0.7,	0.5	&	0.65 	(0.98) &	0.22 	(0.59) &	0.13 	(0.21 ) &	0.65 	&	0.88 	&	0.64 	(1.15) &	0.22 	(0.68) &	0.14 	(0.26) &	0.66 	&	0.79	\\
0.9,	0.5,	0.3	&	0.74 	(0.54) &	0.15 	(0.26) &	0.11 	(0.11 ) &	0.74 	&	0.74 	&	0.73 	(0.61) &	0.15 	(0.29) &	0.11 	(0.15) &	0.74 	&	0.69	\\
0.8,	0.8,	0.8	&	0.33 	(0.89) &	0.33 	(0.89) &	0.33 	(0.89 ) &	0.33 	&	1.10 	&	0.33 	(1.09) &	0.33 	(1.04) &	0.33 	(1.10) &	0.50 	&	0.97	\\
0.8,	0.7,	0.6	&	0.46 	(0.76) &	0.31 	(0.55) &	0.23 	(0.34 ) &	0.46 	&	1.06 	&	0.46 	(0.94) &	0.30 	(0.64) &	0.23 	(0.43) &	0.52 	&	0.97	\\
0.8,	0.6,	0.5	&	0.53 	(0.64) &	0.26 	(0.37) &	0.21 	(0.24 ) &	0.53 	&	1.02 	&	0.52 	(0.76) &	0.26 	(0.45) &	0.21 	(0.30) &	0.55 	&	0.94	\\
0.7,	0.5,	0.4	&	0.48 	(0.41) &	0.29 	(0.25) &	0.24 	(0.17 ) &	0.48 	&	1.05 	&	0.47 	(0.47) &	0.28 	(0.31) &	0.24 	(0.22) &	0.50 	&	0.99	\\
0.6,	0.5,	0.5	&	0.38 	(0.31) &	0.31 	(0.23) &	0.31 	(0.23 ) &	0.38 	&	1.09 	&	0.38 	(0.41) &	0.31 	(0.32) &	0.31 	(0.31) &	0.45 	&	1.03	\\
0.6,	0.5,	0.3	&	0.42 	(0.29) &	0.34 	(0.23) &	0.24 	(0.12 ) &	0.42 	&	1.07 	&	0.42 	(0.37) &	0.34 	(0.30) &	0.24 	(0.17) &	0.46 	&	1.02	\\
0.5,	0.5,	0.3	&	0.37 	(0.21) &	0.37 	(0.21) &	0.26 	(0.11 ) &	0.37 	&	1.09 	&	0.37 	(0.28) &	0.37 	(0.29) &	0.26 	(0.17) &	0.44 	&	1.04	\\
0.4,	0.3,	0.2	&	0.38 	(0.13) &	0.33 	(0.10) &	0.29 	(0.07 ) &	0.38 	&	1.09 	&	0.38 	(0.20) &	0.33 	(0.16) &	0.29 	(0.12) &	0.43 	&	1.05	\\
0.2,	0.2,	0.2	&	0.33 	(0.06) &	0.33 	(0.06) &	0.33 	(0.06 ) &	0.33 	&	1.10 	&	0.33 	(0.10) &	0.33 	(0.11) &	0.33 	(0.11) &	0.40 	&	1.06  	\\
\hline
 Responses &  \multicolumn{5}{c|}{GERADE $\alpha=2/3$} &  \multicolumn{5}{c|}{EDBCD $\gamma=2$}   \\
  $p_1$, $p_2$, $p_3$  & $\rho_1(\sigma^2)$ & $\rho_2(\sigma^2)$ & $\rho_3(\sigma^2)$ & SB & Ent & $\rho_1(\sigma^2)$ & $\rho_2(\sigma^2)$ & $\rho_3(\sigma^2)$ & SB & Ent  \\
   \hline
0.9,	0.9,	0.9	&	0.34 	(2.08) &	0.33 	(2.10) &	0.33 	(2.16) &	0.55 	&	0.93 	&	0.33 	(2.02) &	0.33 	(2.04) &	0.33 	(2.09) &	0.58 	&	0.87  	\\
0.9,	0.9,	0.7	&	0.42 	(2.10) &	0.42 	(1.98) &	0.15 	(0.47) &	0.59 	&	0.88 	&	0.42 	(2.15) &	0.42 	(2.21) &	0.15 	(0.48) &	0.61 	&	0.83  	\\
0.9,	0.8,	0.6	&	0.55 	(1.35) &	0.29 	(1.02) &	0.15 	(0.31) &	0.61 	&	0.86 	&	0.56 	(1.53) &	0.29 	(1.14) &	0.15 	(0.32) &	0.62 	&	0.82  	\\
0.9,	0.7,	0.5	&	0.63 	(1.10) &	0.22 	(0.66) &	0.14 	(0.27) &	0.66 	&	0.80 	&	0.64 	(1.09) &	0.22 	(0.69) &	0.14 	(0.24) &	0.67 	&	0.77  	\\
0.9,	0.5,	0.3	&	0.72 	(0.58) &	0.16 	(0.30) &	0.11 	(0.14) &	0.73 	&	0.69 	&	0.73 	(0.58) &	0.16 	(0.28) &	0.11 	(0.14) &	0.74 	&	0.68  	\\
0.8,	0.8,	0.8	&	0.33 	(0.95) &	0.34 	(0.94) &	0.33 	(0.94) &	0.52 	&	0.98 	&	0.33 	(0.94) &	0.33 	(0.95) &	0.33 	(0.98) &	0.53 	&	0.95  	\\
0.8,	0.7,	0.6	&	0.45 	(0.76) &	0.31 	(0.60) &	0.23 	(0.35) &	0.53 	&	0.96 	&	0.46 	(0.77) &	0.31 	(0.58) &	0.23 	(0.38) &	0.54 	&	0.94  	\\
0.8,	0.6,	0.5	&	0.51 	(0.66) &	0.27 	(0.40) &	0.21 	(0.26) &	0.56 	&	0.93 	&	0.52 	(0.65) &	0.27 	(0.41) &	0.21 	(0.27) &	0.56 	&	0.92  	\\
0.7,	0.5,	0.4	&	0.47 	(0.41) &	0.29 	(0.26) &	0.24 	(0.18) &	0.53 	&	0.98 	&	0.47 	(0.47) &	0.28 	(0.29) &	0.24 	(0.19) &	0.52 	&	0.97  	\\
0.6,	0.5,	0.5	&	0.38 	(0.31) &	0.31 	(0.25) &	0.31 	(0.24) &	0.49 	&	1.01 	&	0.38 	(0.35) &	0.31 	(0.27) &	0.31 	(0.27) &	0.47 	&	1.01  	\\
0.6,	0.5,	0.3	&	0.42 	(0.31) &	0.34 	(0.24) &	0.24 	(0.13) &	0.50 	&	1.00 	&	0.42 	(0.31) &	0.34 	(0.27) &	0.24 	(0.13) &	0.49 	&	1.00  	\\
0.5,	0.5,	0.3	&	0.37 	(0.23) &	0.37 	(0.24) &	0.26 	(0.13) &	0.48 	&	1.02 	&	0.37 	(0.24) &	0.36 	(0.25) &	0.26 	(0.14) &	0.47 	&	1.02  	\\
0.4,	0.3,	0.2	&	0.38 	(0.14) &	0.33 	(0.11) &	0.29 	(0.08) &	0.47 	&	1.03 	&	0.38 	(0.16) &	0.33 	(0.12) &	0.29 	(0.10) &	0.45 	&	1.03  	\\
0.2,	0.2,	0.2	&	0.33 	(0.07) &	0.33 	(0.07) &	0.33 	(0.06) &	0.45 	&	1.04 	&	0.33 	(0.07) &	0.33 	(0.08) &	0.33 	(0.07) &	0.44 	&	1.05 	\\
\hline
\end{tabular}
\end{table}

\begin{table}[ht]
\centering
\caption{Simulated allocation proportion ($N_{n,1}/n$), selection bias (SB) and  entropy (Ent)  of three
designs  with allocation targets $\rho_k=\frac{1/q_k}{\sum_j(1/q_j)}$, sample size=200, and replication =1000, in a 3-treatment trial}
\label{tab:urn-3treatment-2}
\medskip

\tiny

\begin{tabular}{| c | ccccc | ccccc |}
\hline
Responses & \multicolumn{5}{c|}{Asymptotic} & \multicolumn{5}{c|}{DBCD $\gamma=4$} \\
  $p_1$, $p_2$, $p_3$  & $\rho_1(\sigma^2)$ & $\rho_2(\sigma^2)$ & $\rho_3(\sigma^2)$ & SB & Ent & $\rho_1(\sigma^2)$ & $\rho_2(\sigma^2)$ & $\rho_3(\sigma^2)$ & SB & Ent   \\
   \hline
0.9,	0.9,	0.9	&	0.33 	(2.00) &	0.33 	(2.00) &	0.33 	(2.00 ) &	0.33 	&	1.10 	&	0.33 	(1.96) &	0.33 	(1.93) &	0.33 	(2.00) &	0.61 	&	0.82  	\\
0.9,	0.9,	0.7	&	0.43 	(2.03) &	0.43 	(2.03) &	0.14 	(0.40 ) &	0.43 	&	1.00 	&	0.42 	(2.15) &	0.43 	(2.15) &	0.15 	(0.43) &	0.63 	&	0.79 	\\
0.9,	0.8,	0.6	&	0.57 	(1.39) &	0.29 	(1.02) &	0.14 	(0.29 ) &	0.57 	&	0.96 	&	0.56 	(1.40) &	0.28 	(1.03) &	0.15 	(0.29) &	0.64 	&	0.78 	\\
0.9,	0.7,	0.5	&	0.65 	(0.98) &	0.22 	(0.59) &	0.13 	(0.21 ) &	0.65 	&	0.88 	&	0.64 	(1.05) &	0.22 	(0.68) &	0.14 	(0.21) &	0.68 	&	0.75 	\\
0.9,	0.5,	0.3	&	0.74 	(0.54) &	0.15 	(0.26) &	0.11 	(0.11 ) &	0.74 	&	0.74 	&	0.73 	(0.59) &	0.15 	(0.28) &	0.11 	(0.13) &	0.75 	&	0.65 	\\
0.8,	0.8,	0.8	&	0.33 	(0.89) &	0.33 	(0.89) &	0.33 	(0.89 ) &	0.33 	&	1.10 	&	0.33 	(1.03) &	0.33 	(1.00) &	0.33 	(0.98) &	0.54 	&	0.92 	\\
0.8,	0.7,	0.6	&	0.46 	(0.76) &	0.31 	(0.55) &	0.23 	(0.34 ) &	0.46 	&	1.06 	&	0.46 	(0.78) &	0.31 	(0.58) &	0.23 	(0.37) &	0.55 	&	0.92 	\\
0.8,	0.6,	0.5	&	0.53 	(0.64) &	0.26 	(0.37) &	0.21 	(0.24 ) &	0.53 	&	1.02 	&	0.52 	(0.70) &	0.27 	(0.42) &	0.21 	(0.27) &	0.57 	&	0.90 	\\
0.7,	0.5,	0.4	&	0.48 	(0.41) &	0.29 	(0.25) &	0.24 	(0.17 ) &	0.48 	&	1.05 	&	0.47 	(0.45) &	0.29 	(0.30) &	0.24 	(0.19) &	0.52 	&	0.96 	\\
0.6,	0.5,	0.5	&	0.38 	(0.31) &	0.31 	(0.23) &	0.31 	(0.23 ) &	0.38 	&	1.09 	&	0.38 	(0.33) &	0.31 	(0.26) &	0.31 	(0.27) &	0.48 	&	1.00 	\\
0.6,	0.5,	0.3	&	0.42 	(0.29) &	0.34 	(0.23) &	0.24 	(0.12 ) &	0.42 	&	1.07 	&	0.42 	(0.35) &	0.34 	(0.27) &	0.24 	(0.16) &	0.49 	&	0.99 	\\
0.5,	0.5,	0.3	&	0.37 	(0.21) &	0.37 	(0.21) &	0.26 	(0.11 ) &	0.37 	&	1.09 	&	0.37 	(0.27) &	0.37 	(0.25) &	0.26 	(0.14) &	0.47 	&	1.01 	\\
0.4,	0.3,	0.2	&	0.38 	(0.13) &	0.33 	(0.10) &	0.29 	(0.07 ) &	0.38 	&	1.09 	&	0.38 	(0.18) &	0.33 	(0.13) &	0.29 	(0.10) &	0.45 	&	1.03 	\\
0.2,	0.2,	0.2	&	0.33 	(0.06) &	0.33 	(0.06) &	0.33 	(0.06 ) &	0.33 	&	1.10 	&	0.33 	(0.08) &	0.33 	(0.09) &	0.33 	(0.08) &	0.43 	&	1.04   	\\
\hline
  Responses &  \multicolumn{5}{c|}{GERADE $\alpha=1/2$} &  \multicolumn{5}{c|}{EDBCD $\gamma=4$}   \\
  $p_1$, $p_2$, $p_3$  & $\rho_1(\sigma^2)$ & $\rho_2(\sigma^2)$ & $\rho_3(\sigma^2)$ & SB & Ent & $\rho_1(\sigma^2)$ & $\rho_2(\sigma^2)$ & $\rho_3(\sigma^2)$ & SB & Ent  \\
   \hline
0.9,	0.9,	0.9	&	0.34 	(1.83)&	0.33 	(1.99) &	0.33 	(1.97) &	0.60 	&	0.87 	&	0.33 	(1.96) &	0.33 	(1.93) &	0.34 	(1.92) &	0.62 	&	0.79  	\\
0.9,	0.9,	0.7	&	0.42 	(1.89)&	0.42 	(1.85) &	0.15 	(0.44) &	0.63 	&	0.81 	&	0.42 	(1.90) &	0.43 	(2.00) &	0.15 	(0.41) &	0.64 	&	0.76  	\\
0.9,	0.8,	0.6	&	0.56 	(1.32)&	0.29 	(0.98) &	0.15 	(0.30) &	0.65 	&	0.79 	&	0.56 	(1.33) &	0.29 	(0.96) &	0.15 	(0.29) &	0.65 	&	0.76  	\\
0.9,	0.7,	0.5	&	0.64 	(0.97)&	0.22 	(0.60) &	0.14 	(0.23) &	0.69 	&	0.74 	&	0.64 	(0.92) &	0.22 	(0.57) &	0.13 	(0.20) &	0.69 	&	0.72  	\\
0.9,	0.5,	0.3	&	0.73 	(0.53)&	0.16 	(0.26) &	0.11 	(0.12) &	0.75 	&	0.64 	&	0.73 	(0.57) &	0.16 	(0.27) &	0.11 	(0.13) &	0.75 	&	0.63  	\\
0.8,	0.8,	0.8	&	0.33 	(0.90)&	0.33 	(0.88) &	0.33 	(0.86) &	0.56 	&	0.91 	&	0.33 	(0.92) &	0.33 	(0.94) &	0.33 	(0.88) &	0.56 	&	0.89  	\\
0.8,	0.7,	0.6	&	0.46 	(0.78)&	0.31 	(0.56) &	0.23 	(0.36) &	0.58 	&	0.90 	&	0.46 	(0.82) &	0.31 	(0.58) &	0.23 	(0.36) &	0.56 	&	0.90  	\\
0.8,	0.6,	0.5	&	0.52 	(0.63)&	0.26 	(0.41) &	0.21 	(0.25) &	0.60 	&	0.87 	&	0.52 	(0.69) &	0.26 	(0.42) &	0.21 	(0.28) &	0.58 	&	0.88  	\\
0.7,	0.5,	0.4	&	0.47 	(0.38)&	0.29 	(0.24) &	0.24 	(0.16) &	0.57 	&	0.92 	&	0.47 	(0.42) &	0.29 	(0.27) &	0.24 	(0.19) &	0.54 	&	0.93  	\\
0.6,	0.5,	0.5	&	0.38 	(0.31)&	0.31 	(0.22) &	0.31 	(0.25) &	0.53 	&	0.95 	&	0.38 	(0.36) &	0.31 	(0.25) &	0.31 	(0.24) &	0.50 	&	0.97  	\\
0.6,	0.5,	0.3	&	0.42 	(0.30)&	0.34 	(0.25) &	0.24 	(0.13) &	0.55 	&	0.94 	&	0.42 	(0.33) &	0.33 	(0.26) &	0.24 	(0.13) &	0.51 	&	0.97  	\\
0.5,	0.5,	0.3	&	0.37 	(0.22)&	0.36 	(0.21) &	0.26 	(0.11) &	0.53 	&	0.96 	&	0.37 	(0.23) &	0.37 	(0.24) &	0.26 	(0.13) &	0.50 	&	0.98  	\\
0.4,	0.3,	0.2	&	0.38 	(0.13)&	0.33 	(0.11) &	0.29 	(0.07) &	0.51 	&	0.97 	&	0.38 	(0.14) &	0.33 	(0.11) &	0.29 	(0.08) &	0.48 	&	1.00  	\\
0.2,	0.2,	0.2	&	0.33 	(0.06)&	0.33 	(0.06) &	0.33 	(0.06) &	0.50 	&	0.98 	&	0.33 	(0.07) &	0.33 	(0.07) &	0.33 	(0.07) &	0.46 	&	1.02 	\\
\hline
\end{tabular}
\end{table}

\end{document}